\documentclass[reqno,11pt,letterpaper]{amsart}

\usepackage[
letterpaper,
textwidth=6.0in,
textheight=8.8in,
centering
]{geometry}

\usepackage{amsmath, amssymb, amsthm, amsfonts}
\usepackage{mathrsfs}
\usepackage[dvipsnames]{xcolor}
\usepackage{enumerate}
\usepackage{graphicx}
\usepackage{etoolbox}
\usepackage{bm}
\usepackage{tikz}

\usepackage[T1]{fontenc}
\usepackage{lmodern} 

\newtheorem{theorem}{Theorem}[section]
\newtheorem{lemma}[theorem]{Lemma}

\newtheorem{proposition}[theorem]{Proposition}
\newtheorem{remark}[theorem]{Remark}
\newtheorem{problem}[theorem]{Problem}

\theoremstyle{definition}

\numberwithin{equation}{section}

\usepackage[numbers]{natbib}

\BeforeBeginEnvironment{thebibliography}{%
\small
\setlength{\baselineskip}{1\baselineskip}
}

\usepackage[colorlinks]{hyperref}

\newcommand{\lrb}[1]{\left( #1 \right)}
\newcommand{\lrs}[1]{\left[ #1 \right]}

\def\p{\partial}

\def\ga{\gamma}
\def\Ga{\Gamma}
\def\om{\omega}
\def\al{\alpha}
\def\si{\sigma}
\def\de{\delta}
\def\la{\lambda}

\newcommand{\ml}[1]{\mathcal{#1}}

\newcommand{\mk}[1]{\mathfrak{#1}}

\def\n{\bm{n}}
\def\x{\bm{x}}
\def\y{\bm{y}}

\def\u{\bm{u}}

\def\w{\bm{w}}

\def\bU{\bar{U}}
\def\bu{\bar{u}}
\def\bc{\bar{c}}
\def\bB{\bar{B}}
\def\bp{\bar{p}}
\def\bS{\bar{S}}
\def\bM{\bar{M}}
\def\brho{\bar{\rho}}
\def\bxi{\bar{\xi}}

\def\dN{\dot{\ml{N}}}
\def\dU{\dot{U}}
\def\du{\dot{u}}
\def\dB{\dot{B}}
\def\dS{\dot{S}}

\def\div{\mathrm{div\,}}
\def\curl{\mathrm{curl\,}}

\begin{document}

\title[Stability and location of steady transonic shocks]
{Stability of steady transonic shocks with constant states in 3D rectangular nozzles under small external forces }

\author[S.K. Weng]{Shangkun Weng}
\address[Shangkun Weng]{\newline School of Mathematics and Statistics,
Wuhan University, Wuhan, Hubei 430072, People's Republic of China}
\email{skweng@whu.edu.cn}

\author[W.G. Yang]{Wengang Yang}
\address[Wengang Yang]{\newline School of Mathematics and Statistics, Hubei Key Laboratory of Mathematics \& Intelligent Computation and Interdisciplinary Sciences, Wuhan University of Technology, Wuhan, 430070, People’s Republic of China.}
\email{yangwg@whut.edu.cn}

\begin{abstract}
This paper investigates the stability of steady transonic shocks with constant states in rectangular nozzles under small three dimensional perturbations of the external force and the exit pressure. Mathematically, the problem can be formulated as a nonlinear free boundary problem for a hyperbolic-elliptic coupled PDE system. One of the key difficulties lies in the arbitrary location of the background shock, which may lead to multiple shock solutions for
the same external force and exit pressure under perturbations,
revealing the instability of the background shock. We reformulate the steady Euler equations in terms of deformation and vorticity to effectively decompose the hyperbolic and elliptic modes. Linearizing the reformulated system and the Rankine-Hugoniot conditions at background transonic shocks, we find that the class of nonzero horizontal external forces helps to uniquely determine the shock location, where the conservation of mass plays a critical role. A delicate iteration is developed to establish the existence and structural stability of transonic shock solutions. Our analysis shows that the stability mechanisms of steady transonic shocks in flat and divergent nozzles are quite different in many aspects.
\end{abstract}

\date{}
\maketitle

\section{Introduction}
We consider steady transonic shocks for a polytropic gas in a
three-dimensional rectangular nozzle subject to a small conservative
external force. The flow is governed by the full compressible Euler
equations
\begin{equation}\label{Euler0}
\begin{cases}
\nabla\cdot (\rho \u)=0,\\
\nabla\cdot (\rho \u\otimes\u)+\nabla p=\rho \bm f,\\
\nabla\cdot \lrb{(E+p) \u }=\rho \bm f \cdot \u.
\end{cases}
\end{equation}
Here $\u=(u_1,u_2,u_3)$ is the velocity, $\rho$ is the density,
$e$ is the specific internal energy, and
$E=\rho(\frac12|\u|^2+e)$ is the total energy per unit volume.
The vector $\bm f$ denotes the external force per unit mass. For a polytropic gas, the equation of state and the specific internal energy are given by
\begin{equation}\label{eq:state}
p=S\rho^{\ga},
\qquad
e=\frac{p}{(\ga-1)\rho},
\end{equation}
where $\ga>1$ is the adiabatic exponent and $S$ is the entropy
variable. The local sound speed is $c=(\ga p/\rho)^{1/2}$ and the
Mach number is $|\u|/c$. In this paper, we study the existence and structural stability of transonic shocks with constant states under small conservative external forces and perturbations of the exit pressure. A central issue is to determine the shock position from the prescribed data, since the constant background states provide no information about its location.

\subsection{History and Literature}

The study of shock waves has a long history in gas
dynamics. For time dependent flows, substantial progress has been made
in understanding shock formation and stability.
Christodoulou \cite{Christodoulou2007} developed a geometric
approach for relativistic fluids, and Christodoulou-Miao
\cite{ChristodoulouMiao2014} treated the corresponding problem
for nonrelativistic flows.
Luk-Speck \cite{LukSpeck2018,LukSpeck2024} established stability
results allowing nonzero vorticity, while
Buckmaster-Shkoller-Vicol
\cite{BuckmasterShkollerVicol2022,BuckmasterShkollerVicol2023b}
constructed solutions that develop shocks from smooth initial
data in two and three dimensions.
Buckmaster-Iyer \cite{BuckmasterIyer2022} investigated unstable
shock formation for two dimensional isentropic flows.
Multidimensional shock interactions for a Chaplygin gas were
analyzed by Serre \cite{Serre2009}.

For steady flows, a classical problem concerns
transonic shocks in nozzles. In their classical monograph,
Courant-Friedrichs \cite{CourantFriedrichs1948} described how
a supersonic flow may be compressed to a subsonic state through
a shock so that the exit pressure agrees with the prescribed
receiver pressure. Later studies based on quasi one dimensional
models examined the dependence of transonic shocks on the nozzle
geometry and the boundary data; see
\cite{Liu1982,EmbidGoodmanMajda1984}.

For the multidimensional steady Euler equations, two typical
classes of background transonic shocks have been studied
extensively: shocks connecting two constant states in flat nozzles,
and shocks with nonconstant states in divergent nozzles or under
external forces.
For the first class, the background shock can be located at any
interior section of the nozzle, even when the upstream state
and the corresponding constant exit pressure are fixed. For potential
flows, Chen-Feldman \cite{ChenFeldman2003,ChenFeldman2004}
developed a free boundary approach to study perturbations of planar
transonic shocks. Xin-Yin \cite{XinYin2008a}
studied transonic shock problems with different exit conditions and
showed that the pressure boundary value problem can be ill-posed
under perturbations of the exit pressure. For the full Euler
equations, transonic shock problems in flat or almost flat nozzles
were studied with suitable restrictions on the exit data; see
\cite{Chen2008,ChenYuan2008}.
In particular, Chen \cite{Chen2008} and Chen-Yuan
\cite{ChenYuan2008} established structural stability for
three dimensional duct problems, with the shock front passing
through a prescribed point and the exit pressure prescribed up to
an additive constant. A central question for this class of shocks is therefore how to
determine the shock position from the prescribed boundary data.
Fang-Xin \cite{FangXin2021} studied this question for the
two dimensional full Euler equations in an almost flat finite nozzle
with prescribed receiver pressure. They determined the admissible
initial shock positions through the solvability condition of a
linearized free boundary problem, and then constructed nonlinear
transonic shock solutions near these positions. Their results also
show that more than one admissible shock position may occur for the
same receiver pressure, depending on the nozzle geometry. Fang-Gao
\cite{FangGao2021} extended the analysis of admissible shock
positions to three dimensional axisymmetric cylindrical nozzles. Fang-Gao-Xiang-Zhao \cite{FangGaoXiangZhao2024} also
established axisymmetric transonic shocks with large swirl in a
finite cylindrical nozzle, where the swirl plays an essential
role in determining the shock position.

The second class includes radially symmetric shocks in divergent
sectors and spherically symmetric shocks in conic nozzles. For fixed
nozzle geometry and incoming flow, a suitable prescribed exit
pressure determines the position of these background shocks.
Li-Xin-Yin \cite{LiXinYin2009a,LiXinYin2009b}
studied transonic shocks in two dimensional divergent nozzles with
prescribed exit pressure. They further established structural
stability under perturbations of both the nozzle walls and the exit
pressure in \cite{LiXinYin2013}. Related transonic shock problems for the full Euler equations
in two and three dimensions were studied by Xin-Yin
\cite{XinYin2008b}. Bae-Feldman \cite{BaeFeldman2011} studied
transonic shocks in multidimensional divergent nozzles for a
nonisentropic potential flow model.
For three dimensional axisymmetric Euler flows in conic nozzles,
Li-Xin-Yin \cite{LiXinYin2010b}
obtained results on the existence and stability of transonic shocks
and the dependence of the shock position on the exit pressure.
Weng-Xie-Xin \cite{WengXieXin2021} established structural
stability under axisymmetric perturbations of nozzle shape
by introducing modified Lagrangian coordinates.
Park \cite{Park2021} established the stability of three
dimensional axisymmetric transonic shocks in divergent nozzles
under perturbations of the incoming flow and the exit pressure,
using a stream function formulation.
For the full Euler system under general three dimensional
perturbations, Liu-Xu-Yuan \cite{LiuXuYuan2016} proved the
stability of spherical transonic shocks under an additional structural condition
on the background solution. Weng-Xin
\cite{WengXin2026a} established the existence and stability of
cylindrical transonic shocks under three dimensional perturbations
of the incoming flow and the exit pressure. In their setting, no
additional structural condition on the background shock is imposed,
and the shock front is not required to pass through a prescribed
point. Spherical transonic shocks in a hemispherical shell were
also studied in \cite{Weng2025} by introducing the
``spherical projection coordinates''.

Besides the nozzle geometry, a suitable external force can also help
determine the shock position. Weng-Yang \cite{WengYang2025a}
established the structural stability of two dimensional isentropic
transonic shocks under an external force, and Weng-Zhang-Zhou
\cite{WengZhangZhou2025} obtained the corresponding three dimensional
result in a rectangular cylinder. In these two works, a nonzero
axial force is already present in the background problem, and the
corresponding flow states are nonconstant. With the incoming flow
and the external force fixed, a suitable prescribed exit pressure uniquely
determines the background shock position. Both studies use the
deformation-curl decomposition introduced by Weng-Xin
\cite{WengXin2019} to effectively decouple the hyperbolic and
elliptic modes of the steady Euler system. This method is also
used in the cylindrical transonic shock problem discussed above
\cite{WengXin2026a}. Further applications include smooth transonic
flows \cite{WengXin2026b}, the Euler-Poisson system
\cite{Weng2019, WengYangZhang2024}, and transonic MHD shocks with aligned
fields \cite{WengYang2025b}. In the present paper, we also use this
decomposition to study transonic shocks for the three dimensional
full Euler equations under small conservative external forces.

\subsection{Problem formulation}

In this subsection, the transonic shock problem is formulated in
a rectangular nozzle with a prescribed entrance state, slip
conditions on the walls and a prescribed exit pressure.

Throughout the paper, $0<\al<1$ is fixed, and
$\|f\|_{k,\al;D}$ denotes the $C^{k,\al}(\overline D)$ norm.
For vector functions, this norm is the sum of the component norms.

Let $L>0$ and set
\begin{equation*}
\ml C=[-1,1]\times[-1,1],
\qquad
\ml N=[0,L]\times\ml C.
\end{equation*}
The transverse variables are denoted by $x'=(x_2,x_3)\in\ml C$.
The entrance and exit are the planes $x_1=0$ and $x_1=L$,
respectively. The four lateral walls remain fixed.

The external force has the form
$\bm f=-\si\nabla\Phi$, where $\si>0$ is a small parameter and
$\Phi\in C^{3,\al}(\overline{\ml N})$ is prescribed.
The potential $\Phi$ is assumed to be constant on the entrance
and is normalized by
\begin{equation}\label{phi_entrance}
\Phi(0,x')=0,
\qquad x'\in\ml C.
\end{equation}
The following compatibility conditions are imposed on the lateral
walls:
\begin{equation}\label{phi_wall}
\begin{cases}
\p_{x_2}\Phi(x_1,\pm1,x_3)=0,\\
\p_{x_3}\Phi(x_1,x_2,\pm1)=0,
\end{cases}
\end{equation}
and
\begin{equation}\label{phi_wall3}
\begin{cases}
\p_{x_2}^3\Phi(x_1,\pm1,x_3)=0,
&(x_1,x_3)\in[0,L]\times[-1,1],\\
\p_{x_3}^3\Phi(x_1,x_2,\pm1)=0,
&(x_1,x_2)\in[0,L]\times[-1,1].
\end{cases}
\end{equation}

Define the generalized Bernoulli function by
\begin{equation}\label{def_Ber}
B=\frac12|\u|^2
+\frac{\ga}{\ga-1}S\rho^{\ga-1}
+\si\Phi.
\end{equation}
Then the Euler system \eqref{Euler0} can be written as
\begin{equation}\label{Euler1}
\begin{cases}
\p_{x_1}(\rho u_1)+\p_{x_2}(\rho u_2)+\p_{x_3}(\rho u_3)=0,\\
\p_{x_1}(\rho u_1^2+p)+\p_{x_2}(\rho u_1u_2)
+\p_{x_3}(\rho u_1u_3)=-\si\rho\p_{x_1}\Phi,\\
\p_{x_1}(\rho u_1u_2)+\p_{x_2}(\rho u_2^2+p)
+\p_{x_3}(\rho u_2u_3)=-\si\rho\p_{x_2}\Phi,\\
\p_{x_1}(\rho u_1u_3)+\p_{x_2}(\rho u_2u_3)
+\p_{x_3}(\rho u_3^2+p)=-\si\rho\p_{x_3}\Phi,\\
\p_{x_1}(\rho u_1B)+\p_{x_2}(\rho u_2B)
+\p_{x_3}(\rho u_3B)=0.
\end{cases}
\end{equation}

Introduce the variables $\ml U=(\u,S,B)$. The deformation-curl
decomposition of Weng-Xin \cite{WengXin2019} gives the following
reformulation of \eqref{Euler0}, which is widely used in many circumstances, especially in the three dimensional transonic shock problem \cite{Weng2025,WengXin2026a}.

By \eqref{def_Ber}, the density and pressure can be expressed as
\begin{equation}\label{def_rho_p}
\begin{split}
\rho
&=\rho(B,S,|\u|^2,\Phi)
=\lrb{
\frac{\ga-1}{\ga S}
\lrb{B-\frac12|\u|^2-\si\Phi}
}^{\frac{1}{\ga-1}},\\
p
&=p(B,S,|\u|^2,\Phi)
=\lrb{
\frac{\ga-1}{\ga S^{1/\ga}}
\lrb{B-\frac12|\u|^2-\si\Phi}
}^{\frac{\ga}{\ga-1}}.
\end{split}
\end{equation}
The sound speed is given by
\begin{equation}\label{sound_speed}
c^2=c^2(B,S,|\u|^2,\Phi)
=(\ga-1)\lrb{B-\frac12|\u|^2-\si\Phi}.
\end{equation}

The entropy $S$ and the generalized Bernoulli function $B$
are constant along streamlines and satisfy the transport equations
\begin{equation}\label{trans_BS}
\rho(u_1\p_{x_1}+u_2\p_{x_2}+u_3\p_{x_3})(S,B)=0.
\end{equation}

Define the vorticity $\bm\om=(\om_1,\om_2,\om_3)=\curl\u$ by
\begin{equation*}
\om_1=\p_{x_2}u_3-\p_{x_3}u_2,
\qquad
\om_2=\p_{x_3}u_1-\p_{x_1}u_3,
\qquad
\om_3=\p_{x_1}u_2-\p_{x_2}u_1.
\end{equation*}
When $u_1>0$, the two transverse momentum equations yield
\begin{equation}\label{eq_om23}
\begin{split}
\om_2
&=\frac{1}{u_1}\lrb{
\om_1u_2+\p_{x_3}B
-\frac{\rho^{\ga-1}}{\ga-1}\p_{x_3}S
},\\
\om_3
&=\frac{1}{u_1}\lrb{
\om_1u_3-\p_{x_2}B
+\frac{\rho^{\ga-1}}{\ga-1}\p_{x_2}S
}.
\end{split}
\end{equation}
Substituting \eqref{eq_om23} into the identity
$\div\bm\om=0$ gives a transport equation for $\om_1$:
\begin{equation}\label{eq_om1}
\begin{split}
&\lrb{
\p_{x_1}+\frac{u_2}{u_1}\p_{x_2}
+\frac{u_3}{u_1}\p_{x_3}
}\om_1
+\om_1\sum_{i=2}^3\p_{x_i}\lrb{\frac{u_i}{u_1}}\\
={}&
\p_{x_2}B\,\p_{x_3}\lrb{\frac{1}{u_1}}
-\frac{\p_{x_2}S}{\ga-1}
\p_{x_3}\lrb{\frac{\rho^{\ga-1}}{u_1}}\\
&-\p_{x_3}B\,\p_{x_2}\lrb{\frac{1}{u_1}}
+\frac{\p_{x_3}S}{\ga-1}
\p_{x_2}\lrb{\frac{\rho^{\ga-1}}{u_1}}.
\end{split}
\end{equation}

Substituting the expression for $\rho$ in \eqref{def_rho_p}
into the mass equation and using \eqref{trans_BS} yields
\begin{equation}\label{def_eq}
\sum_{i,j=1}^3
\bigl(c^2(B,S,|\u|^2,\Phi)\de_{ij}-u_iu_j\bigr)
\p_{x_i}u_j
=\si\u\cdot\nabla\Phi.
\end{equation}
Together with the vorticity equations, this gives the
deformation-curl system for the velocity:
\begin{equation}\label{dc_sys}
\begin{cases}
\displaystyle
\sum_{i,j=1}^3(c^2\de_{ij}-u_iu_j)\p_{x_i}u_j
=\si\u\cdot\nabla\Phi,\\
\p_{x_2}u_3-\p_{x_3}u_2=\om_1,\\
\p_{x_3}u_1-\p_{x_1}u_3=\om_2,\\
\p_{x_1}u_2-\p_{x_2}u_1=\om_3.
\end{cases}
\end{equation}

Thus the steady Euler system is decomposed into the transport
equations \eqref{trans_BS} and \eqref{eq_om1} for
$S$, $B$ and $\om_1$, and the first order deformation-curl
system \eqref{dc_sys} for $\u$, with
$\om_2$ and $\om_3$ given by \eqref{eq_om23}.
The corresponding equivalence is stated below; see
\cite[Lemma~2.1]{WengXin2026a} for the full Euler system without
an external force and
\cite[Lemma~2.1]{WengZhangZhou2025} for isentropic flows with
an external force. The proof follows the same argument and
is omitted.

\begin{lemma}[Equivalence]\label{lem_equivalence}
Let $(\rho,\u,S)$ be a $C^2$ state in a domain $D$ with $\rho>0,\,u_1>0$. Define $p$ and $B$ by \eqref{eq:state} and \eqref{def_Ber},
respectively, and let $\bm\om=\curl\u$.
Then the following two statements are equivalent:
\begin{enumerate}
\item
$(\rho,\u,S)$ satisfies the steady Euler system
\eqref{Euler0} in $D$.

\item
$(\u,S,B)$ satisfies the transport equations
\eqref{trans_BS}, the algebraic relations \eqref{eq_om23}
and the deformation-curl system \eqref{dc_sys}
in $D$.
\end{enumerate}
\end{lemma}

The aim is to establish the structural stability of one dimensional
transonic shocks connecting two constant states under small external
forces and perturbations of the exit pressure. The background shock
is a normal shock of the Euler equations without an external force.
It connects two constant states
\begin{equation}\label{bg_states}
\ml{\bU}_\pm=(\bu_\pm,0,0,\bS_\pm,\bB).
\end{equation}
Denote the corresponding densities and pressures by
$\brho_\pm>0$ and $\bp_\pm>0$, respectively, and set
\begin{equation*}
\bp_\pm=\bS_\pm\brho_\pm^\ga,
\qquad
\bc_\pm^2=\frac{\ga\bp_\pm}{\brho_\pm},
\qquad
\bM_\pm=\frac{\bu_\pm}{\bc_\pm}.
\end{equation*}
The common Bernoulli constant is
\begin{equation}\label{bg_B}
\bB
=\frac12\bu_\pm^2
+\frac{\ga}{\ga-1}\bS_\pm\brho_\pm^{\ga-1}.
\end{equation}
For the background states, let
$[\bar f]:=\bar f_+-\bar f_-$.
The normal shock relations are
\begin{align}
[\brho\bu]&=0,\label{bk_rh1}\\
[\brho\bu^2+\bp]&=0,\label{bk_rh2}\\
[\bB]&=0.\label{bk_rh3}
\end{align}
The common mass and momentum fluxes are denoted by
\begin{equation*}
\brho\bu:=\brho_-\bu_-=\brho_+\bu_+,
\qquad
\brho\bu^2+\bp
:=\brho_-\bu_-^2+\bp_-
=\brho_+\bu_+^2+\bp_+.
\end{equation*}
The background states are assumed to satisfy
\begin{equation}\label{bg_shock_cond}
\bM_->1>\bM_+>0,
\qquad
[\bp]>0,
\qquad
[\bS]>0.
\end{equation}

Then, the piecewise constant flow
\begin{equation}\label{bg_shock}
\ml{\bU}(x)
=
\begin{cases}
\ml{\bU}_-,&0<x_1<L_s,\\
\ml{\bU}_+,&L_s<x_1<L,
\end{cases}
\end{equation}
is a transonic shock solution for every $L_s\in(0,L)$. Thus, the shock may be located anywhere in $(0,L)$, and the background states provide no information about its position.

In the perturbed problem, both the flow states and the shock
front are unknown. The upstream and downstream states are
denoted by $\ml U_-$ and $\ml U_+$, respectively.
The entrance state remains $\ml{\bU}_-$, and the external
force is $\bm f=-\si\nabla\Phi$.

The shock front is represented as
\begin{equation}\label{def_shock}
\ml S
:=\{(x_1,x_2,x_3):
x_1=\xi(x_2,x_3),\ (x_2,x_3)\in\ml C\}.
\end{equation}

The jump of a quantity $f$ across $\ml S$ is denoted by
$[f]:=f_+(\xi(x'),x')-f_-(\xi(x'),x')$.
The Rankine-Hugoniot conditions follow from the conservation
form \eqref{Euler1} and read
\begin{align}
[\rho u_1]
-\p_{x_2}\xi[\rho u_2]
-\p_{x_3}\xi[\rho u_3]
&=0,\label{rh1}\\
[\rho u_1^2+p]
-\p_{x_2}\xi[\rho u_1u_2]
-\p_{x_3}\xi[\rho u_1u_3]
&=0,\label{rh2}\\
[\rho u_1u_2]
-\p_{x_2}\xi[\rho u_2^2+p]
-\p_{x_3}\xi[\rho u_2u_3]
&=0,\label{rh3}\\
[\rho u_1u_3]
-\p_{x_2}\xi[\rho u_2u_3]
-\p_{x_3}\xi[\rho u_3^2+p]
&=0,\label{rh4}\\
[B]&=0.\label{rh5}
\end{align}

The first four conditions can be rewritten to separate the
equations for the shock slopes from the remaining jump conditions.
For this purpose, define on $\ml S$
\begin{equation*}
\begin{split}
J_{0j}(\ml U_+,\ml U_-,\Phi)
&=\lrs{\rho(B,S,|\u|^2,\Phi)u_j},
\qquad j=1,2,3,\\
J_{ij}(\ml U_+,\ml U_-,\Phi)
&=\lrs{
\rho(B,S,|\u|^2,\Phi)u_iu_j
+\de_{ij}p(B,S,|\u|^2,\Phi)
},
\qquad i,j=1,2,3.
\end{split}
\end{equation*}
Then \eqref{rh1}-\eqref{rh4} take the form
\begin{equation}\label{def_G0123}
\ml J_i(\ml U_+,\ml U_-;
\Phi,\p_{x_2}\xi,\p_{x_3}\xi)
\quad:=J_{i1}-J_{i2}\p_{x_2}\xi-J_{i3}\p_{x_3}\xi=0,
\quad i=0,1,2,3.
\end{equation}
Provided that $J_{22}J_{33}-J_{23}J_{32}\neq0$, the equations
$\ml J_2=\ml J_3=0$ give
\begin{equation}\label{def_mlS_23}
\begin{cases}
\displaystyle
\p_{x_2}\xi
=\frac{J_{21}J_{33}-J_{23}J_{31}}
{J_{22}J_{33}-J_{23}J_{32}}
:=\ml S_2(\ml U_+,\ml U_-;\Phi),\\
\displaystyle
\p_{x_3}\xi
=\frac{J_{22}J_{31}-J_{21}J_{32}}
{J_{22}J_{33}-J_{23}J_{32}}
:=\ml S_3(\ml U_+,\ml U_-;\Phi).
\end{cases}
\end{equation}
Substituting these expressions into $\ml J_0=\ml J_1=0$
yields
\begin{equation}\label{def_G01_0}
G_i(\ml U_+,\ml U_-;\Phi)
:=J_{i1}-\ml S_2J_{i2}-\ml S_3J_{i3}=0,
\qquad i=0,1.
\end{equation}
Under the same nonvanishing condition,
\eqref{rh1}-\eqref{rh4} are equivalent to these two equations
together with \eqref{def_mlS_23}.

At the nozzle exit, the pressure is prescribed as
\begin{equation*}
p_+(L,x')=\bp_++\si p_{ex}(x'),
\end{equation*}
where $p_{ex}\in C^{2,\al}(\overline{\ml C})$ satisfies
the compatibility conditions
\begin{equation}\label{exit_comp}
\p_{x_2}p_{ex}(\pm1,x_3)=0,
\qquad
\p_{x_3}p_{ex}(x_2,\pm1)=0.
\end{equation}

With the boundary data and jump conditions specified above,
the free boundary problem is formulated as follows.

\begin{problem}\label{prob1}
Find a shock front $\ml S$ of the form \eqref{def_shock}
and flow states
\begin{equation*}
\ml U_\pm=(u_{1\pm},u_{2\pm},u_{3\pm},S_\pm,B_\pm)
\end{equation*}
satisfying the following conditions.
\begin{enumerate}
\item
In the upstream domain
\begin{equation}\label{def_N-}
\ml N_-
:=\{(x_1,x_2,x_3):
0<x_1<\xi(x'),\ (x_2,x_3)\in\ml C\},
\end{equation}
$\ml U_-$ satisfies the Euler equations \eqref{Euler0},
the entrance condition
\begin{equation}
\ml U_-(0,x_2,x_3)=\ml{\bU}_-,
\qquad (x_2,x_3)\in\ml C,
\end{equation}
and the slip conditions
\begin{equation}\label{bdy_U-_Ga0_Wall}
\begin{cases}
u_{2-}(x_1,\pm1,x_3)=0,
&0\leq x_1\leq\xi(\pm1,x_3),\quad x_3\in[-1,1],\\
u_{3-}(x_1,x_2,\pm1)=0,
&0\leq x_1\leq\xi(x_2,\pm1),\quad x_2\in[-1,1].
\end{cases}
\end{equation}

\item
In the downstream domain
\begin{equation}\label{def_N+}
\ml N_+
:=\{(x_1,x_2,x_3):
\xi(x')<x_1<L,\ (x_2,x_3)\in\ml C\},
\end{equation}
$\ml U_+$ satisfies the Euler equations \eqref{Euler0},
the slip conditions
\begin{equation}\label{bdy_U+_wall}
\begin{cases}
u_{2+}(x_1,\pm1,x_3)=0,
&\xi(\pm1,x_3)\leq x_1\leq L,\quad x_3\in[-1,1],\\
u_{3+}(x_1,x_2,\pm1)=0,
&\xi(x_2,\pm1)\leq x_1\leq L,\quad x_2\in[-1,1],
\end{cases}
\end{equation}
and the exit pressure condition
\begin{equation}\label{bdy_p+}
p_+(L,x_2,x_3)=\bp_++\si p_{ex}(x_2,x_3),
\qquad (x_2,x_3)\in\ml C.
\end{equation}

\item
Across $\ml S$, the two states $\ml U_-$ and $\ml U_+$
satisfy the Rankine-Hugoniot conditions
\eqref{rh1}-\eqref{rh5}.
\end{enumerate}
\end{problem}

The next subsection states the main result for
Problem~\ref{prob1}.

\subsection{Main results and comments}

To state the main result, define
\begin{equation}\label{pos_F}
F(t)
=\bar K\int_{\ml C}\Phi(t,x')\,dx'
-\frac{\brho_+}{1-\bM_+^2}
\int_{\ml C}\Phi(L,x')\,dx',
\qquad 0\leq t\leq L,
\end{equation}
where
\begin{equation*}
\bar K=\brho_+
+\frac{\ga-1}{\ga+1}\brho_->0.
\end{equation*}

\begin{theorem}\label{thm_nonlin}
Let $\ml{\bU}_\pm$ be the background states specified above.
Suppose that $\Phi\in C^{3,\al}(\overline{\ml N})$ satisfies
\eqref{phi_entrance}, \eqref{phi_wall} and
\eqref{phi_wall3}, and that $p_{ex}$
satisfies \eqref{exit_comp}.
Assume that
\begin{equation}\label{lin_pos_range}
\inf_{0<t<L}F(t)
<\int_{\ml C}p_{ex}(x')\,dx'
<\sup_{0<t<L}F(t)
\end{equation}
and
\begin{equation}\label{lin_pos_mono}
\int_{\ml C}\p_{x_1}\Phi(t,x')\,dx'\neq0,
\qquad 0<t<L.
\end{equation}
Then there exists a unique $\bxi^*\in(0,L)$ satisfying
\begin{equation}\label{lin_pos_eq}
F(\bxi^*)=\int_{\ml C}p_{ex}(x')\,dx'.
\end{equation}

Moreover, there exists $\si_*>0$ such that, for every
$0<\si\leq\si_*$, Problem~\ref{prob1} admits a unique transonic
shock solution $(\ml U_-,\ml U_+;\xi)$ with
\begin{equation}\label{final_reg}
\ml U_\pm\in C^{2,\al}(\overline{\ml N_\pm}),
\qquad
\xi\in C^{3,\al}(\overline{\ml C}),
\end{equation}
and
\begin{equation}\label{final_est}
\sum_\pm
\|\ml U_\pm-\ml{\bU}_\pm\|_{2,\al;\ml N_\pm}
+\|\xi-\bxi^*\|_{3,\al;\ml C}
\leq C\si.
\end{equation}

The flow is supersonic upstream and subsonic downstream:
\begin{equation*}
u_{1-}^2-c_-^2>0
\quad\text{in }\ml N_-,
\qquad
c_+^2-|\u_+|^2>0
\quad\text{in }\ml N_+.
\end{equation*}

The entropy condition
\begin{equation*}
S_+(\xi(x'),x')-S_-(\xi(x'),x')>0,
\qquad x'\in\overline{\ml C},
\end{equation*}
holds across the shock. The constants $\si_*$ and $C$ depend only
on the background states, the nozzle, $\al$ and the prescribed
data.
\end{theorem}

The theorem shows that, under the stated assumptions, a transonic shock persists when small nonzero horizontal external forces are exerted and the exit pressure is perturbed. No point on the shock front is fixed in advance. The flow on each side remains close to the corresponding constant state, and the shock front remains close to a plane. The background shock can be placed anywhere in $(0,L)$. To obtain an initial shock position for the perturbed problem, we linearize the deformation-curl formulation of the steady Euler equations and the Rankine-Hugoniot conditions at the background shocks. Solving the transport equations for the entropy and the Bernoulli function then leads to a deformation-curl system for the velocity in supersonic and subsonic regimes with mixed boundary conditions on the shock and normal boundary conditions on others.The integral solvability condition from the above boundary value problems can uniquely determine the shock location $\bxi^*$. It is clear from our analysis that conservation of mass plays a critical role. The key assumption on the external force concerns its horizontal component: its integral over each interior cross section is nonzero, as required by \eqref{lin_pos_mono}. Together with the range condition \eqref{lin_pos_range} on the exit pressure, this ensures that \eqref{lin_pos_eq} determines a unique approximating shock position $\bxi^*$. Starting from this approximation, the nonlinear construction determines both the position and the shape of the perturbed shock.

There are three main issues in passing from the initial approximation to a nonlinear solution.
First, the shock position must remain uniquely determined
throughout the iteration. Its determining equation is the
integral solvability condition for the downstream
deformation-curl system, which expresses conservation of mass.
With the upstream solution $\ml U_-$ already constructed and
the iteration functions $\widehat V$ and
$\widehat{\bm q}=(\widehat q_2,\widehat q_3)$ fixed, this
condition reads
\begin{equation*}
\ml I(\widehat V,\widehat{\bm q};
\ml U_-,\Phi,p_{ex};\ddot\xi^*)=0,
\end{equation*}
where $\ddot\xi^*=\xi(1,1)-\bxi^*$; see \eqref{pos_eq}.
For $|\ddot\xi^*|\leq C\si$, its derivative satisfies
\begin{equation}\label{I_deriv}
\p_{\ddot\xi^*}\ml I
=\frac{1-\bM_+^2}{\brho\bu}\si\bar K
\int_{\ml C}\p_{x_1}\Phi(\bxi^*,y')\,dy'
+o(\si).
\end{equation}
Condition \eqref{lin_pos_mono} makes the leading coefficient
nonzero. Together with $\ml I(0)=O(\si^2)$, this allows the
implicit function theorem, applied to $\si^{-1}\ml I=0$,
to determine a unique displacement $\ddot\xi^*=O(\si)$.
The two tangential derivatives of the shock are updated
through the tangential Rankine-Hugoniot conditions. Their
compatibility condition provides the boundary data for the
first component of the vorticity. The updated shock front
is then recovered from these derivatives and the value
$\xi(1,1)=\bxi^*+\ddot\xi^*$.

Second, after transforming the subsonic region to a fixed
domain and formulating the linear problem for the iteration,
the source terms in the curl equations are not divergence
free in general. Following \cite{WengXin2026a}, an auxiliary
scalar function with homogeneous mixed boundary conditions
is introduced to enlarge the deformation-curl system.
Its gradient corrects the source terms so that the corrected
curl data are divergence free. The auxiliary function is
shown to vanish identically at the fixed point, and the
enlarged system then reduces to the original
deformation-curl system.

Third, the estimates must hold up to the edges where the
shock or the exit meets the flat walls, and where the walls
meet one another. The compatibility conditions on the
potential and the exit pressure allow even and odd extensions
across the walls. The transport equations and the shock
relations preserve these extensions. The iteration has
uniform $C^{2,\al}$ bounds for the flow and the shock slopes,
and is a contraction in $C^{1,\al}$. The resulting shock
front belongs to $C^{3,\al}$. Returning to the original
coordinates gives the solution in
Theorem~\ref{thm_nonlin}.

\begin{remark}
The present construction is inspired by the idea of Fang-Xin
\cite{FangXin2021}: an initial shock position is first determined
through a solvability condition, and a nonlinear solution is then
constructed near this position. The choice of decomposition is
important in carrying out this approach in three dimensions.

Fang-Xin use a first order elliptic system for the pressure and
flow angle, coupled with transport equations for the entropy
and Bernoulli function. This decomposition was first observed by Serre
\cite{Serre1995}. This elliptic formulation relies on the
two dimensional structure and does not extend directly to
general three dimensional flows. The deformation-curl
decomposition used here instead gives an elliptic system for
the three velocity components in the subsonic region, coupled
with transport equations for the entropy and the first component
of the vorticity. This decomposition allows the construction in three dimensions
and makes clear that the crucial solvability condition determining
the initial shock position essentially comes from conservation
of mass.
\end{remark}

\begin{remark}
The nozzle geometry in \cite{FangXin2021} and the external force
in the present problem play similar roles in determining the
shock position. In \cite{FangXin2021}, a nozzle with both expanding
and contracting portions may admit several shock solutions for
the same exit pressure. Here a change in the direction of the
external force can produce a similar phenomenon if
\eqref{lin_pos_mono} is removed.

For example, take
\begin{equation*}
\Phi(x_1)
=\frac{L}{2\pi}\left(\cos\frac{2\pi x_1}{L}-1\right),
\qquad
\bm f=-\si\nabla\Phi
=\si\left(\sin\frac{2\pi x_1}{L},0,0\right).
\end{equation*}
The regularity and compatibility conditions on the potential
are satisfied. However, the integral appearing in
\eqref{I_deriv} is
\begin{equation*}
\int_{\ml C}\p_{x_1}\Phi(t,x')\,dx'
=-4\sin\frac{2\pi t}{L},
\end{equation*}
which changes sign at $t=L/2$. Thus
\eqref{lin_pos_mono} no longer holds. Since $\Phi(L)=0$,
the position function in \eqref{pos_F} becomes
\begin{equation*}
F(t)=\frac{4\bar K L}{2\pi}
\left(\cos\frac{2\pi t}{L}-1\right).
\end{equation*}
It decreases on $(0,L/2)$ and increases on $(L/2,L)$.
Taking the constant exit pressure perturbation
$p_{ex}=-\bar K L/(2\pi)$, which satisfies the compatibility
and range conditions, the solvability condition
\eqref{lin_pos_eq} reads
\begin{equation*}
\frac{4\bar K L}{2\pi}
\left(\cos\frac{2\pi t}{L}-1\right)
=-\frac{4\bar K L}{2\pi}.
\end{equation*}
Consequently, there are two admissible initial shock positions,
\begin{equation*}
t_1=\frac L4,\qquad t_2=\frac{3L}{4}.
\end{equation*}
At both positions, the integral in \eqref{I_deriv} is nonzero,
so the derivative of the nonlinear solvability condition has
a nonzero leading coefficient. The nonlinear construction
can therefore be carried out near each initial position,
giving two distinct transonic shock solutions with the same
entrance state, external force and exit pressure. This
nonuniqueness shows how changes in the direction of the
external force can lead to instability of the background
shock, in a manner similar to the nozzle wall perturbations
considered in \cite{FangXin2021}.

\end{remark}

\begin{remark}
The stability mechanisms for steady transonic shocks in
divergent and flat nozzles are different, although the
deformation-curl decomposition is important in both problems.
In the divergent nozzle considered in \cite{WengXin2026a},
the expanding geometry introduces the shock displacement as
a principal term in the linearized Rankine-Hugoniot conditions.
These conditions give an algebraic equation for the shock front,
whose leading term is a scalar multiple of the perturbation
of the radial velocity. The entropy transports the information
about the shock position into the deformation-curl system for
the downstream velocity. After the curl equations are solved
and an auxiliary potential function is introduced, the remaining
problem reduces to a second order elliptic equation with a
nonlocal term. This equation has Neumann conditions on the
nozzle walls and at the exit, and an oblique derivative condition
on the shock. The latter is obtained from a Poisson equation
on the shock surface with Neumann boundary conditions.
The expanding geometry also gives the sign in the oblique
derivative condition needed for unique solvability.

For the constant background states in the present flat nozzle,
the linearized Rankine-Hugoniot conditions contain no such
leading term involving the shock displacement. Using
Bernoulli's law, the continuity equation can be rewritten as
the modified divergence equation \eqref{def_eq}, with a source
term due to the external force. Together with the curl equations,
its linearization leads to a second order elliptic equation
for an auxiliary potential function with Neumann conditions
on the entire boundary. Solvability requires a balance between
the source term and the boundary data. Integrating the
linearized modified divergence equation in the upstream and
downstream regions and using the linearized Rankine-Hugoniot
conditions gives \eqref{lin_pos_eq}. Thus the external force
enters the condition for the shock position through conservation
of mass.
\end{remark}

The rest of the paper is organized as follows.
Section~\ref{sec:linear} constructs the linear approximation of
the upstream and downstream flows and derives the solvability
condition from conservation of mass that determines the
approximating shock position. The shock slopes are then obtained
from the tangential Rankine-Hugoniot conditions.
In Section~\ref{sec:nonlinear}, the nonlinear supersonic flow is
constructed throughout the nozzle, and the downstream equations
and the boundary and shock conditions are reformulated on a fixed
domain. Section~\ref{sec:iter} constructs the
nonlinear iteration, determines the shock position through the
nonlinear solvability condition, and proves convergence of the
iteration. The proof of the main theorem is completed by verifying
that the fixed point satisfies the nonlinear equations and shock
conditions, establishing uniqueness in the stated perturbation
class, and returning to the original coordinates.

\section{The initial approximation and the shock position}
\label{sec:linear}
This section constructs a linear approximation of the flow and the shock.
The upstream problem supplies the data at the shock, while the
solvability condition for the downstream problem determines the
reference position $\bxi^*$. The tangential jump conditions then
determine the two shock derivatives.

\subsection{Linearized Rankine-Hugoniot conditions}
For a position $\bxi^*\in(0,L)$ to be determined below, the plane
$x_1=\bxi^*$ divides $\ml N$ into the two regions
\begin{equation}\label{lin_domains}
\begin{split}
\dN_-&=\{(x_1,x_2,x_3): 0<x_1<\bxi^*,\ (x_2,x_3)\in \ml{C}\},\\
\dN_+&=\{(x_1,x_2,x_3): \bxi^*<x_1<L,\ (x_2,x_3)\in \ml{C}\}.
\end{split}
\end{equation}

The linear perturbations on $\dN_\pm$ are denoted by
\begin{equation}\label{lin_vars}
\ml{\dU}_\pm
=(\du_{1\pm},\du_{2\pm},\du_{3\pm},\dS_\pm,\dB_\pm).
\end{equation}
The corresponding velocity and vorticity are
$\dot{\bm u}_\pm=(\du_{1\pm},\du_{2\pm},\du_{3\pm})$
and $\dot{\bm\om}_\pm:=\curl\dot{\bm u}_\pm$.
To linearize the shock conditions, write
\begin{equation*}
\ml U_\pm=\ml{\bU}_\pm+\ml{\dU}_\pm.
\end{equation*}

Expanding the shock conditions at the background states gives
\begin{equation}\label{lin_S_exp}
\ml S_i(\ml U_+,\ml U_-;\Phi)
=a_0\bigl(\du_{i+}-\du_{i-}\bigr)
+O\!\left(
\bigl(|\ml{\dU}_+|+|\ml{\dU}_-|+\si|\Phi|\bigr)^2
\right),
\quad i=2,3.
\end{equation}
Here
\begin{equation}\label{def_a0}
a_0=\frac{\brho\bu}{[\bp]}.
\end{equation}

The remaining jump functions have the expansions
\begin{equation}\label{def_G01}
G_i(\ml U_+,\ml U_-;\Phi)
=\beta_i^+\cdot\ml{\dU}_+
+\beta_i^-\cdot\ml{\dU}_-
+\si b_i\Phi
+O\!\left(
\bigl(|\ml{\dU}_+|+|\ml{\dU}_-|+\si|\Phi|\bigr)^2
\right),
\quad i=0,1,
\end{equation}
where
\begin{equation}\label{def_lin_coefs}
\begin{split}
\beta_0^\pm&=\pm \brho \bu(\frac{1-\bM_\pm^2}{\bu_\pm},0,0,-\frac{1}{(\ga-1)\bS_\pm},\frac{1}{\bc_\pm^2}),\\
b_0&=\brho \bu \lrb{\frac{1}{\bc_-^2}- \frac{1}{\bc_+^2}},\\
\beta_1^\pm&=\pm\bigl(
\brho\bu(1-\bM_\pm^2),0,0,
-\frac{\brho\bu^2+\bp}{(\ga-1)\bS_\pm},
\brho_\pm(1+\bM_\pm^2)
\bigr),\\
b_1&=\brho_-(1+\bM_-^2)-\brho_+(1+\bM_+^2).
\end{split}
\end{equation}

The linear supersonic problem is therefore considered first.

\subsection{Supersonic flow}
The upstream linear solution is first constructed on the whole
nozzle $\ml N$ and then restricted to $\dN_-$. It follows from \eqref{trans_BS}-\eqref{def_eq} that
\begin{equation}\label{lin_up_eq}
\begin{cases}
\bu_-\p_{x_1}(\dS_-,\dB_-)=0,\\
\p_{x_1}\dot\om_{1-}=0,\\
\p_{x_3}\du_{1-}-\p_{x_1}\du_{3-}
=\frac{1}{\bu_-}\lrb{
\p_{x_3}\dB_- -\frac{\brho_-^{\ga-1}}{\ga-1}\p_{x_3}\dS_-
},\\
\p_{x_1}\du_{2-}-\p_{x_2}\du_{1-}
=\frac{1}{\bu_-}\lrb{
-\p_{x_2}\dB_-+\frac{\brho_-^{\ga-1}}{\ga-1}\p_{x_2}\dS_-
},\\
(1-\bM_-^2)\p_{x_1}\du_{1-}+\p_{x_2}\du_{2-}+\p_{x_3}\du_{3-}=\si \frac{\bu_-}{\bc_-^2}\p_{x_1}\Phi,
\end{cases}
\end{equation}
in $\ml N$. The entrance and wall conditions are
\begin{equation}\label{lin_up_inlet}
\ml{\dot U}_-(0,x')=0,\qquad x'\in\ml{C},
\end{equation}
and
\begin{equation}\label{lin_up_wall}
\begin{cases}
\du_{2-}(x_1,\pm1,x_3)=0 ,\,\,(x_1,x_3)\in [0,L]\times [-1,1],\\
\du_{3-}(x_1,x_2,\pm1)=0 ,\,\,(x_1,x_2)\in [0,L]\times [-1,1].
\end{cases}
\end{equation}

Equations \eqref{lin_up_eq}, with
\eqref{lin_up_inlet}-\eqref{lin_up_wall}, form a
constant coefficient hyperbolic problem in the supersonic region.
The assumptions on $\Phi$ in Section~1 ensure the compatibility
needed in the following lemma.

\begin{lemma}\label{lem_lin_up}
There exists a unique solution $\ml{\dU}_-$ to
\eqref{lin_up_eq}-\eqref{lin_up_wall} satisfying
\begin{equation}\label{lin_up_est}
\|\ml{\dU}_-\|_{2,\al;\ml N}
\leq C\|\si\Phi\|_{3,\al;\ml N}
\leq C\si,
\end{equation}
where $C$ depends on the background state, the nozzle, and the
prescribed function $\Phi$, but not on $\si$.

Furthermore, the solution satisfies the following properties:
\begin{equation}\label{lin_up_invar}
\dS_-=\dB_-=\dot\om_{1-}\equiv 0\qquad \text{in }\,\ml N,
\end{equation}
and the following integral identity:
\begin{equation}\label{lin_up_int}
\int_{\ml C}\du_{1-}(t,x')\,dx'
=\si\frac{\bM_-^2}{\bu_-(1-\bM_-^2)}
\int_{\ml C}\Phi(t,x')\,dx',\qquad 0\leq t\leq L.
\end{equation}
\end{lemma}

\begin{proof}
The transport equations in \eqref{lin_up_eq} and the entrance data
$\dS_-(0,x')=\dB_-(0,x')=\dot\om_{1-}(0,x')=0$ give
\begin{equation}
\dS_-=\dB_-=\dot\om_{1-}\equiv0 \qquad \text{in }\,\ml N.
\end{equation}

The velocity $\dot{\bm u}_-=(\du_{1-},\du_{2-},\du_{3-})$ then
solves
\begin{equation}
\begin{cases}
\p_{x_2}\du_{3-}-\p_{x_3}\du_{2-}=0,\\
\p_{x_3}\du_{1-}-\p_{x_1}\du_{3-}=0,\\
\p_{x_1}\du_{2-}-\p_{x_2}\du_{1-}=0,\\
(1-\bM_-^2)\p_{x_1}\du_{1-}+\p_{x_2}\du_{2-}+\p_{x_3}\du_{3-}=\si \frac{\bu_-}{\bc_-^2}\p_{x_1}\Phi, 
\end{cases}
\end{equation}
with boundary conditions
\begin{equation}
\begin{cases}
\du_{i-}(0,x')=0,\,i=1,2,3,\,x'\in \ml{C},\\
\du_{2-}(x_1,\pm1,x_3)=0 ,\,(x_1,x_3)\in [0,L]\times [-1,1],\\
\du_{3-}(x_1,x_2,\pm1)=0 ,\,(x_1,x_2)\in [0,L]\times [-1,1].
\end{cases}
\end{equation}

Since this velocity is curl free and $\ml N$ is simply connected,
there exists a potential $\dot\phi_-(x_1,x_2,x_3)$ such that
\begin{equation*}
\p_{x_1}\dot\phi_-=\du_{1-},\,\p_{x_2}\dot\phi_-=\du_{2-},\,
\p_{x_3}\dot\phi_-=\du_{3-}.
\end{equation*}

Consequently,  $\dot\phi_-$ satisfies
\begin{equation}\label{lin_phi_up}
\begin{cases}
(1-\bM_-^2)\p_{x_1}^2\dot\phi_-
+\p_{x_2}^2\dot\phi_-+\p_{x_3}^2\dot\phi_-
=\si\frac{\bu_-}{\bc_-^2}\p_{x_1}\Phi,\\
\p_{x_i} \dot\phi_-(0,x')=0,\,i=1,2,3,\,x'\in \ml{C},\\
\p_{x_2}\dot\phi_-(x_1,\pm 1,x_3)=0 ,\,(x_1,x_3)\in [0,L]\times [-1,1],\\
\p_{x_3}\dot\phi_-(x_1,x_2,\pm 1)=0 ,\,(x_1,x_2)\in [0,L]\times [-1,1].
\end{cases}
\end{equation}

Without loss of generality, we may take $\dot\phi_-(0,x')=0$ for $\forall x'\in \ml C$.
The wall compatibility conditions \eqref{phi_wall} and
\eqref{phi_wall3} allow even extension
of the data across the lateral walls. There exists a unique
solution $\dot\phi_-$ to \eqref{lin_phi_up} with
\begin{equation*}
\|\nabla\dot\phi_-\|_{2,\al;\ml N}
\leq C\|\si\Phi\|_{3,\al;\ml N}.
\end{equation*}
Uniqueness preserves the reflection symmetry, so the lateral wall
conditions hold. This proves the existence, uniqueness, and estimate
stated in the lemma.

Integrating the deformation equation in \eqref{lin_phi_up} over $[0,t]\times\ml C$
and using the entrance and wall conditions gives
\begin{equation*}
(1-\bM_-^2)\int_{\ml C}\du_{1-}(t,x')\,dx'
=\si\frac{\bu_-}{\bc_-^2}
\int_{\ml C}\bigl(\Phi(t,x')-\Phi(0,x')\bigr)\,dx'.
\end{equation*}
Since $\Phi(0,x')=0$ and $\bM_-^2/\bu_-=\bu_-/\bc_-^2$,
this proves \eqref{lin_up_int}.
\end{proof}

The upstream solution is now known at every section of the nozzle.

\subsection{\texorpdfstring{The determination of $\ml{\dU}_+$ and $\bxi^*$}{The subsonic flow and the shock position}}
Similarly, the perturbation $\ml{\dU}_+$ satisfies the linearized Euler system
\begin{equation}\label{lin_down_eq}
\begin{cases}
\bu_+\p_{x_1}(\dS_+,\dB_+)=0,\\
\p_{x_1}\dot\om_{1+}=0,\\
\p_{x_3}\du_{1+}-\p_{x_1}\du_{3+}
=\frac{1}{\bu_+}\lrb{
\p_{x_3}\dB_+-\frac{\brho_+^{\ga-1}}{\ga-1}\p_{x_3}\dS_+
},\\
\p_{x_1}\du_{2+}-\p_{x_2}\du_{1+}
=\frac{1}{\bu_+}\lrb{
-\p_{x_2}\dB_++\frac{\brho_+^{\ga-1}}{\ga-1}\p_{x_2}\dS_+
},\\
(1-\bM_+^2)\p_{x_1}\du_{1+}+\p_{x_2}\du_{2+}+\p_{x_3}\du_{3+}=\si \frac{\bu_+}{\bc_+^2}\p_{x_1}\Phi,
\end{cases}
\end{equation}
in $\dN_+$. The slip conditions on the nozzle walls are
\begin{equation}\label{lin_down_wall}
\begin{cases}
\du_{2+}(x_1,\pm1,x_3)=0 ,\,(x_1,x_3)\in [\bxi^*,L]\times [-1,1],\\
\du_{3+}(x_1,x_2,\pm1)=0 ,\,(x_1,x_2)\in [\bxi^*,L]\times [-1,1].
\end{cases}
\end{equation}

Once the supersonic state $\ml{\dU}_-$ is known, dropping the high order terms in \eqref{def_G01} yields
\begin{equation}\label{lin_RH}
\beta_i^+\cdot \ml{\dU}_+(\bxi^*,x') + \beta_i^-\cdot \ml{\dU}_-(\bxi^*,x')+\si b_i \Phi(\bxi^*,x')=0,\,i=0,1,
\end{equation}
and
\begin{equation}\label{jump_dot_B}
\dB_-(\bxi^*,x')=\dB_+(\bxi^*,x').
\end{equation}

Then, we have the following results.

\begin{lemma}\label{RH_bd_condi}
The boundary values on the shock $x_1=\bxi^*$ are
\begin{equation}\label{bd_dot_u1_S}
\begin{cases}
\du_{1+}(\bxi^*,x')=b_u^s \du_{1-}(\bxi^*,x')+ \si b_\Phi^u \Phi(\bxi^*,x'),\\
\dS_+(\bxi^*,x')=b_s^s \du_{1-}(\bxi^*,x')+\si b_\Phi^s \Phi(\bxi^*,x'),
\end{cases}
\end{equation}
where 
\begin{equation}\label{def_RH_coefs}
\begin{split}
b_u^s&=\frac{\bar{\rho}_-\bar{p}_-}{\bar{\rho}_+\bar{p}_+}\frac{\bM_-^2-1}{\bM_+^2-1},\\
b_s^s&=(\ga-1)\frac{\bS_+}{\bp_+\bu_-}(\bM_-^2-1)[\bp],\\
b_\Phi^s&=(\ga-1)\frac{\bS_+}{\bp_+}(b_1 -\bu_+ b_0),\\
b_\Phi^u&=\frac{1}{\bM_+^2-1}
\frac{(\brho_+\bu_+^2+\bp_+)b_0-\brho_+\bu_+b_1}
{\bp_+\brho_+}.
\end{split}
\end{equation}

\end{lemma}

\begin{proof}
In this proof, write the background fluxes as
\begin{equation*}
m_0:=\brho_\pm\bu_\pm,
\qquad
m_1:=\brho_\pm\bu_\pm^2+\bp_\pm.
\end{equation*}
The first equation in \eqref{lin_down_eq}, together with
\eqref{jump_dot_B} and \eqref{lin_up_invar}, gives
$\dB_+\equiv0$. Hence, on $x_1=\bxi^*$,
\eqref{lin_RH} becomes
\begin{equation}\label{lin_RH_pair}
\begin{cases}
\displaystyle
\brho_+(1-\bM_+^2)\du_{1+}
-\frac{m_0}{(\ga-1)\bS_+}\dS_+
+\si b_0\Phi
=\brho_-(1-\bM_-^2)\du_{1-},\\[2mm]
\displaystyle
m_0(1-\bM_+^2)\du_{1+}
-\frac{m_1}{(\ga-1)\bS_+}\dS_+
+\si b_1\Phi
=m_0(1-\bM_-^2)\du_{1-}.
\end{cases}
\end{equation}
Subtracting $\bu_+$ times the first equation from the second one and
using
\begin{equation*}
m_1-\bu_+m_0=\bp_+,
\qquad
[\bp]=m_0(\bu_- -\bu_+),
\end{equation*}
we obtain
\begin{equation*}
-\frac{\bp_+}{(\ga-1)\bS_+}\dS_+
+\si(b_1-\bu_+b_0)\Phi
=\frac{1-\bM_-^2}{\bu_-}[\bp]\du_{1-}.
\end{equation*}
Therefore,
\begin{equation*}
\dS_+=b_s^s\du_{1-}+\si b_\Phi^s\Phi.
\end{equation*}
Substituting this identity into the first equation in
\eqref{lin_RH_pair} and using the background
Rankine-Hugoniot relations gives
\begin{equation*}
\du_{1+}=b_u^s\du_{1-}+\si b_\Phi^u\Phi.
\end{equation*}
This proves \eqref{bd_dot_u1_S}.
\end{proof}

For later use, denote the two boundary data in
\eqref{bd_dot_u1_S} by
\begin{equation}\label{def_lin_shock_data}
\begin{aligned}
\dot g_1(\ml{\dU}_-,\Phi;\bxi^*)(x')
&:=b_u^s\du_{1-}(\bxi^*,x')+\si b_\Phi^u\Phi(\bxi^*,x'),\\
\dot g_4(\ml{\dU}_-,\Phi;\bxi^*)(x')
&:=b_s^s\du_{1-}(\bxi^*,x')+\si b_\Phi^s\Phi(\bxi^*,x').
\end{aligned}
\end{equation}
With these arguments fixed, we write $\dot g_1(x')$ and $\dot g_4(x')$.
Then $\du_{1+}(\bxi^*,x')=\dot g_1(x')$, and the entropy satisfies
\begin{equation}\label{lin_S_problem}
\p_{x_1}\dS_+=0\quad\text{in }\dN_+,
\qquad \dS_+(\bxi^*,x')=\dot g_4(x'),
\end{equation}
and hence
\begin{equation}\label{lin_S_sol}
\dS_+(x_1,x')=\dot g_4(x')\qquad\text{in }\dN_+.
\end{equation}

At the exit of the nozzle, the pressure is prescribed by
\begin{equation}
p_+(L,x')=\bp_+ +\si p_{ex}(x').
\end{equation}

Expanding the pressure function in \eqref{def_rho_p} at either
background state gives
\begin{equation}\label{Exp_pressure}
\begin{split}
p(\ml U;\Phi)
={}&\bp-\brho\bu\du_1
-\frac{\bp}{(\ga-1)\bS}\dS
+\brho\dB-\si\brho\Phi
+O\left(
\bigl(|\ml{\dU}|+\si|\Phi|\bigr)^2
\right).
\end{split}
\end{equation}

Since $\dB_+=0$, the linearized pressure condition is
\begin{equation}\label{lin_exit_p}
\si p_{ex}(x')=-\brho \bu \du_{1+}(L,x')-\frac{\bp_+ }{(\ga-1)\bS_+} \dS_{+}(L,x')-\si \brho_+ \Phi(L,x').
\end{equation}

Solving \eqref{lin_exit_p} for $\du_{1+}(L,x')$ gives
the exit datum
\begin{equation}\label{def_lin_exit_data}
\begin{aligned}
&\dot g_e(\ml{\dU}_-,\Phi,p_{ex};\bxi^*)(x')\\
&\quad:=-\frac{\bp_+}{(\ga-1)\bS_+\brho\bu}\dot g_4(x')
-\frac{\si}{\brho\bu}p_{ex}(x')-\frac{\si}{\bu_+}\Phi(L,x')\\
&\quad=b_u^{ex}\du_{1-}(\bxi^*,x')
-\frac{\si}{\brho\bu}p_{ex}(x')-\frac{\si}{\bu_+}\Phi(L,x')
-\si\frac{\bp_+b_\Phi^s}{(\ga-1)\bS_+\brho\bu}\Phi(\bxi^*,x'),
\end{aligned}
\end{equation}
where
\begin{equation}\label{def_bue}
b_u^{ex}=\frac{1-\bM_-^2}{\brho_-\bu_-^2}[\bp].
\end{equation}
Thus $\du_{1+}(L,x')=\dot g_e(x')$, with the data and
position arguments omitted as above.

The two tangential Rankine-Hugoniot conditions prescribe the shock
derivatives by
\begin{equation}\label{def_lin_shock_shape}
\p_{x_i}\dot\xi
=\frac{\brho\bu}{[\bp]}
\bigl(\du_{i+}-\du_{i-}\bigr)(\bxi^*,x'),
\qquad i=2,3.
\end{equation}
The equality
$\p_{x_3}\p_{x_2}\dot\xi=\p_{x_2}\p_{x_3}\dot\xi$
requires the compatibility condition
\begin{equation}\label{lin_shape_comp}
(\p_{x_2}\du_{3+}-\p_{x_3}\du_{2+})(\bxi^*,x')
=(\p_{x_2}\du_{3-}-\p_{x_3}\du_{2-})(\bxi^*,x').
\end{equation}
Since $\ml C$ is simply connected, this condition is also
sufficient. By the definition of the first vorticity component, it
is equivalent to
\begin{equation}\label{lin_shock_vort}
\dot\om_{1+}(\bxi^*,x')
=\dot\om_{1-}(\bxi^*,x').
\end{equation}

Together with $\p_{x_1}\dot\om_{1+}=0$ in
\eqref{lin_down_eq}, this gives
\begin{equation*}
\dot\om_{1+}\equiv0
\qquad\text{in }\dN_+.
\end{equation*}

We are now in a position to determine the shock location $x_1=\bxi^*$, and to solve the downstream state $(\du_{1+},\du_{2+},\du_{3+})$. Denote the source term by
\begin{equation}\label{def_lin_force}
\dot F_0(\x):=\si\frac{\bu_+}{\bc_+^2}\p_{x_1}\Phi(\x).
\end{equation}
The velocity satisfies
\begin{equation}\label{lin_vel_eq}
\begin{cases}
(1-\bM_+^2)\p_{x_1}\du_{1+}
+\p_{x_2}\du_{2+}+\p_{x_3}\du_{3+}=\dot F_0,\\
\p_{x_2}\du_{3+}-\p_{x_3}\du_{2+}=0,\\
\p_{x_3}\du_{1+}-\p_{x_1}\du_{3+}
=-\dfrac{\brho_+^{\ga-1}}{(\ga-1)\bu_+}\p_{x_3}\dS_+,\\
\p_{x_1}\du_{2+}-\p_{x_2}\du_{1+}
=\dfrac{\brho_+^{\ga-1}}{(\ga-1)\bu_+}\p_{x_2}\dS_+,
\end{cases}
\qquad\text{in }\dN_+,
\end{equation}
with boundary conditions
\begin{equation}\label{lin_vel_bc}
\begin{cases}
\du_{1+}(\bxi^*,x')=\dot g_1(x'),\qquad
\du_{1+}(L,x')=\dot g_e(x'),\\
\du_{2+}(x_1,\pm1,x_3)=0,\qquad
\du_{3+}(x_1,x_2,\pm1)=0.
\end{cases}
\end{equation}

The following compatibility condition determines whether this
velocity problem can be solved. It will also determine $\bxi^*$.

\begin{lemma}\label{lem_lin_solva}
For $\bxi^*\in(0,L)$, problem \eqref{lin_vel_eq}-\eqref{lin_vel_bc}
admits a unique solution if and only if
\begin{equation}\label{comp_condi}
\iint_{\ml{C}} p_{ex}(x')\,dx'
=\bar K \iint_{\ml{C}} \Phi(\bxi^*,x')\, dx'
-\frac{\brho_+}{1-\bM_+^2} \iint_{\ml{C}}\Phi(L,x')\, dx',
\end{equation}
where
\begin{equation}\label{def_barK_0}
\bar K=
\left\{
\frac{\bM_-^2}{\bu_-^2}
+\frac{1}{\bu_+\bu_-}
\frac{\bM_+^2}{1-\bM_+^2}
\right\}[\bp]
+\bu_+b_0-b_1-\brho\bu\,b_\Phi^u.
\end{equation}
\end{lemma}

\begin{proof}
The curl equations in \eqref{lin_vel_eq}, together with
$\p_{x_1}\dS_+=0$, show that the field
\begin{equation*}
\lrb{
\du_{1+}+\brho_+^{\ga-1}\frac{\dS_+}{(\ga-1)\bu_+},\du_{2+},\du_{3+}
}
\end{equation*}
is curl free. Since $\dN_+$ is simply
connected, there exists a potential $\dot\phi_+(\x)$ such that
\begin{equation}\label{lin_phi_vel}
\p_{x_1}\dot\phi_+
=\du_{1+}+\frac{\brho_+^{\ga-1}}{(\ga-1)\bu_+}\dS_+,
\qquad
\p_{x_i}\dot\phi_+=\du_{i+},\quad i=2,3.
\end{equation}

The boundary value problem for $\dot\phi_+$ reads
\begin{equation}\label{lin_phi_down}
\begin{cases}
(1-\bM_+^2)\p_{x_1}^2\dot\phi_+ + \p_{x_2}^2\dot\phi_+ + \p_{x_3}^2\dot\phi_+
=\si \frac{\bu_+}{\bc_+^2} \p_{x_1}\Phi,\\
\p_{x_1}\dot\phi_+(\bxi^*,x')=\dot g_1(x')
+\frac{\brho_+^{\ga-1}}{(\ga-1)\bu_+}\dot g_4(x'),\\
\p_{x_1}\dot\phi_+(L,x')=-\si \frac{1}{\brho \bu } p_{ex}(x')
-\si\frac{1}{\bu_+}\Phi(L,x'),\\
\p_{x_2}\dot\phi_+(x_1,\pm1,x_3)=0,\\
\p_{x_3}\dot\phi_+(x_1,x_2,\pm 1)=0,
\end{cases}
\end{equation}

Integrating the first equation in \eqref{lin_phi_down} and using
the boundary conditions gives
\begin{equation*}
\begin{split}
&\iiint_{[\bxi^*,L]\times\ml C}
\bigl((1-\bM_+^2)\p_{x_1}^2\dot\phi_+
+\p_{x_2}^2\dot\phi_+
+\p_{x_3}^2\dot\phi_+\bigr)\,d\x\\
=&(1-\bM_+^2)\iint_{\ml C}
\bigl(\p_{x_1}\dot\phi_+(L,x')
-\p_{x_1}\dot\phi_+(\bxi^*,x')\bigr)\,dx'\\
=&(1-\bM_+^2)\iint_{\ml C}
\bigl(\du_{1+}(L,x')-\du_{1+}(\bxi^*,x')\bigr)\,dx'\\
=&\si\frac{\bu_+}{\bc_+^2}\iint_{\ml C}
\bigl(\Phi(L,x')-\Phi(\bxi^*,x')\bigr)\,dx'.
\end{split}
\end{equation*}

In terms of the data just defined, this is
\begin{equation}\label{lin_vel_comp}
(1-\bM_+^2)\int_{\ml C}(\dot g_e-\dot g_1)\,dx'
=\int_{\dN_+}\dot F_0\,d\x.
\end{equation}
Applying the boundary conditions for $\du_{1+}$ at $x_1=\bxi^*$ and $x_1=L$, respectively, yields
\begin{equation}
\begin{split}
&\du_{1+}(L,x')-\du_{1+}(\bxi^*,x')\\
=&b_u^{ex}\du_{1-}(\bxi^*,x')-b_u^s \du_{1-}(\bxi^*,x')-\si \frac{1}{\brho \bu } p_{ex}(x')\\
&-\si\frac{1}{\bu_+}\Phi(L,x')
-\si \frac{\bp_+ }{(\ga-1)\bS_+}\frac{b_\Phi^s}{\brho \bu} \Phi(\bxi^*,x')-
\si b_\Phi^u \Phi(\bxi^*,x')
\end{split}
\end{equation}
Integrating this difference over $\ml C$ and using
\eqref{lin_vel_comp} gives
\begin{equation}
\begin{split}
&\iint_{\ml C}
\bigl(\du_{1+}(L,x')-\du_{1+}(\bxi^*,x')\bigr)\,dx'\\
=&(b_u^{ex}-b_u^s) \iint_{\ml{C}}\du_{1-}(\bxi^*,x')\,dx'
-\si\frac{1}{\brho \bu}\iint_{\ml{C}} p_{ex}(x')\,dx'\\
&-\si\iint_{\ml C}
\left(
\frac{1}{\bu_+}\Phi(L,x')
+\left(
\frac{b_1-\bu_+b_0}{\brho\bu}+b_\Phi^u
\right)\Phi(\bxi^*,x')
\right)\,dx'\\
=&\si\frac{1}{\bu_+}\frac{\bM_+^2}{1-\bM_+^2}
\iint_{\ml C}
\bigl(\Phi(L,x')-\Phi(\bxi^*,x')\bigr)\,dx'.
\end{split}
\end{equation}

Substituting the upstream identity \eqref{lin_up_int} and
collecting the coefficients of the potential gives
\begin{equation}
\begin{split}
\iint_{\ml{C}} p_{ex}(x')\,dx'
=\bar K \iint_{\ml{C}} \Phi(\bxi^*,x')\, dx'
-\frac{\brho_+}{1-\bM_+^2} \iint_{\ml{C}}\Phi(L,x')  \,dx',
\end{split}
\end{equation}
where
\begin{equation}\label{def_barK}
\begin{split}
\bar K
=&\brho\bu\Biggl\{
\frac{b_u^{ex}-b_u^s}{\bu_-}
\frac{\bM_-^2}{1-\bM_-^2}
+\frac{1}{\bu_+}
\frac{\bM_+^2}{1-\bM_+^2}
-\frac{b_1-\bu_+b_0}{\brho\bu}
-b_\Phi^u
\Biggr\}.
\end{split}
\end{equation}
A direct computation gives the expression for $\bar K$ in
\eqref{def_barK_0}.

Thus \eqref{comp_condi} is necessary. Conversely, assume that
\eqref{comp_condi} holds. The calculation above shows that it is
exactly the compatibility condition for the Neumann problem
\eqref{lin_phi_down}. Hence this problem has a solution, unique up
to an additive constant. Define
\begin{equation*}
\du_{1+}
:=\p_{x_1}\dot\phi_+
-\frac{\brho_+^{\ga-1}}{(\ga-1)\bu_+}\dS_+,
\qquad
\du_{i+}:=\p_{x_i}\dot\phi_+,
\quad i=2,3.
\end{equation*}
Then $\dot{\bm u}_+$ satisfies \eqref{lin_vel_eq} and its
boundary conditions. The upstream wall conditions and \eqref{phi_wall} give
\begin{equation*}
\p_{x_i}\dot g_1=\p_{x_i}\dot g_4=0
\quad\text{on }x_i=\pm1,\qquad i=2,3.
\end{equation*}
Together with \eqref{exit_comp}, these conditions allow the data
in \eqref{lin_phi_down} to be extended evenly across the lateral
walls.

The Neumann estimate on the extended
slab then gives
\begin{equation*}
\|\dot{\bm u}_+\|_{2,\al;\dN_+}
\leq C\si.
\end{equation*}
This also gives the required regularity at the edges. We have proved
the sufficiency and completed the proof.
\end{proof}

The expression for $\bar K$, namely \eqref{def_barK_0} or \eqref{def_barK}, is far too complicated to determine its sign, and the sign of $\bar K$ is vital for our results, as shown in \eqref{pos_F}. Surprisingly, $\bar K$ has a quite simple form, and it is always positive. This will be carried out in the following lemma.

\begin{lemma}\label{lem_K_positive}
The coefficient $\bar K$ appearing in \eqref{comp_condi} satisfies
\begin{equation}\label{exp:K_positive}
\bar K
=
\brho_+
+\frac{\ga-1}{\ga+1}\brho_-
>0.
\end{equation}
\end{lemma}

\begin{proof}
For this calculation, denote the background fluxes by
\begin{equation*}
m_0:=\brho_\pm\bu_\pm,
\qquad
m_1:=\brho_\pm\bu_\pm^2+\bp_\pm.
\end{equation*}
The background shock relations give
\begin{equation}\label{K_ratio}
\la
:=
\frac{\bu_+}{\bu_-}
=
\frac{\brho_-}{\brho_+}\in(0,1).
\end{equation}
The momentum Rankine-Hugoniot condition and the Bernoulli condition
for the background shock are
\begin{equation}\label{K_bg_rel}
\brho_-\bu_-^2+\bp_-
=
\brho_+\bu_+^2+\bp_+,
\qquad
\frac{\bc_-^2}{\ga-1}+\frac{\bu_-^2}{2}
=
\frac{\bc_+^2}{\ga-1}+\frac{\bu_+^2}{2}.
\end{equation}
Using $\bp_\pm=\brho_\pm\bc_\pm^2/\ga$ and
$\bu_+=\la\bu_-$, these two relations become
\begin{equation*}
\frac{\bc_+^2}{\la}-\bc_-^2
=
\ga\bu_-^2(1-\la),
\qquad
\bc_+^2-\bc_-^2
=
\frac{\ga-1}{2}\bu_-^2(1-\la^2).
\end{equation*}
Solving for $\bc_-^2$ and $\bc_+^2$, we obtain
\begin{equation}\label{K_sound}
\bc_-^2=\frac{X}{2}\bu_-^2,\quad
\bc_+^2=\frac{Y}{2}\la \bu_-^2,
\end{equation}
where
\begin{equation}\label{K_XY}
X:=(\ga+1)\la-(\ga-1),
\qquad
Y:=(\ga+1)-(\ga-1)\la.
\end{equation}

Moreover,
\begin{equation}\label{K_bg_id}
\begin{split}
&\brho_+=\frac{\brho_-}{\la},
\qquad
[\bp]=\brho_-\bu_-^2(1-\la),\\
&\bc_-^2=\frac{\bu_-^2}{2}X,
\qquad
\bc_+^2=\frac{\bu_-^2\la}{2}Y,\\
&\bM_-^2=\frac{2}{X},
\qquad
\bM_+^2=\frac{2\la}{Y},\\
&\bp_-=\frac{\brho_-\bu_-^2}{2\ga}X,
\qquad
\bp_+=\frac{\brho_-\bu_-^2}{2\ga}Y.
\end{split}
\end{equation}

To simplify the three terms in \eqref{def_barK_0}, write
\begin{equation*}
\bar K=\bar K_1+\bar K_2+\bar K_3,
\end{equation*}
where
\begin{equation}\label{K_parts}
\begin{split}
\bar K_1
&:=
\lrb{
\frac{\bM_-^2}{\bu_-^2}
+\frac{1}{\bu_+\bu_-}
\frac{\bM_+^2}{1-\bM_+^2}
}[\bp],\\
\bar K_2
&:=
\bu_+b_0-b_1,\\
\bar K_3
&:=
-\frac{\brho\bu}{\bM_+^2-1}
\frac{m_1b_0-m_0b_1}{\bp_+\brho_+}.
\end{split}
\end{equation}

For $\bar K_1$, the identities
\begin{equation*}
1-\bM_+^2
=
\frac{(\ga+1)(1-\la)}{Y},
\qquad
X+(\ga+1)(1-\la)=2
\end{equation*}
give
\begin{equation}\label{K1_calc}
\begin{split}
\bar K_1
&=\left(
\frac{2}{X\bu_-^2}
+\frac{2}{\bu_-^2(Y-2\la)}
\right)\brho_-\bu_-^2(1-\la)\\
&=2\brho_-\left(
\frac{1-\la}{X}+\frac{1-\la}{Y-2\la}
\right)\\
&=2\brho_-\left(
\frac{1-\la}{X}+\frac{1}{\ga+1}
\right),
\end{split}
\end{equation}
where we have used $Y-2\la=(\ga+1)(1-\la)$.

Next, the definitions of $b_0$ and $b_1$ give
\begin{equation*}
\begin{split}
b_0
&=\frac{2\brho_-}{\bu_-}
\lrb{   \frac{1}{X}-\frac{1}{\la Y}},\\
b_1
&=
\brho_-\lrb{1+\frac{2}{X}-\frac{1}{\la} -\frac{2}{Y}}.
\end{split}
\end{equation*}
Therefore,
\begin{equation}\label{K2_calc}
\begin{split}
\bar K_2&=2\la \brho_-  \lrb{   \frac{1}{X}-\frac{1}{\la Y}}
- \brho_-\lrb{1+\frac{2}{X}-\frac{1}{\la} -\frac{2}{Y}}\\
&=\brho_- \lrb{
\frac{2\la}{X}-1-\frac{2}{X}+\frac{1}{\la} 
}\\
&=
\brho_-(1-\la)
\lrb{
\frac{1}{\la}-\frac{2}{X}
}\\
&=
-\frac{(\ga-1)\brho_-(1-\la)^2}{\la X}.
\end{split}
\end{equation}
Here the last equality uses
\begin{equation*}
X-2\la=-(\ga-1)(1-\la).
\end{equation*}

It remains to calculate $\bar K_3$. From the definitions of $m_0$
and $m_1$,
\begin{equation*}
\begin{split}
m_1&=\brho_\pm \bu_\pm^2+\bp_\pm=m_0 \bu_-+\bp_-=m_0 \bu_-+\brho_- \frac{\bc_-^2}{\ga}\\
&=m_0\bu_-\lrb{
1+\frac{X}{2\ga}}
=m_0 \bu_- \frac{(\la+1)(\ga+1)}{2 \ga}
\end{split}
\end{equation*}

Consequently,
\begin{equation*}
\begin{split}
m_1b_0-m_0b_1=&\frac{m_0^2}{\bu_-} \frac{(\la+1)(\ga+1)}{ \ga}
\lrb{   \frac{1}{X}-\frac{1}{\la Y}}\\
&-m_0 \brho_-\lrb{1+\frac{2}{X}-\frac{1}{\la} -\frac{2}{Y}}\\
=&\frac{m_0^2}{\bu_-} \lrb{
\frac{(\la-1)\ga+\la+1}{\ga X}-\frac{1+\ga+\la(1-\ga)}{\la \ga Y}-1+\frac{1}{\la}
}\\
=&\frac{m_0^2}{\bu_-}
\frac{\la-1-\la \ga+\ga}{\ga \la }=\frac{m_0^2}{\bu_-}
\frac{(\la-1)(1-\ga)}{\ga \la }
\end{split}
\end{equation*}
Moreover,
\begin{equation*}
\bp_+\brho_+
=
\frac{m_0^2}{\ga\bM_+^2}=\frac{1}{\la \ga} \frac{m_0^2}{2} Y.
\end{equation*}
It follows that
\begin{equation}\label{K3_calc}
\begin{split}
\bar K_3=& 2\frac{ m_0 Y}{(\ga+1)(1-\la)}\frac{m_0^2}{\bu_-}
\frac{(\la-1)(1-\ga)}{m_0^2 Y  }  \\
&=2\brho_-  \frac{\ga-1}{\ga+1}.
\end{split}
\end{equation}

Hence
\begin{equation*}
\begin{split}
\bar K&=\bar K_1+\bar K_2+\bar K_3\\
&= 2\brho_-\lrb{\frac{1-\la}{X}+\frac{1}{\ga+1}}
-\frac{(\ga-1)\brho_-(1-\la)^2}
{\la X}+2\brho_-  \frac{\ga-1}{\ga+1}.\\
\end{split}
\end{equation*}

Combining the terms with denominator $\la X$ and using the
definition of $X$ gives
\begin{equation}\label{K_sum}
\begin{split}
\frac{\bar K}{\brho_-}&=2\frac{1-\la}{X}+\frac{2}{\ga+1}-\frac{(\ga-1)(1-\la)^2}{\la X}+2\frac{\ga-1}{\ga+1}\\
&=\frac{(1-\la)(\la-\ga+\ga\la+1)}{\la X}+\frac{2\ga}{\ga+1}\\
&=\frac{1-\la}{\la}+\frac{2\ga}{\ga+1}\\
&=\frac{\brho_+}{\brho_-}+\frac{\ga-1}{\ga+1}.
\end{split}
\end{equation}

This is exactly \eqref{exp:K_positive}, and the proof is complete.
\end{proof}

The solvability condition \eqref{comp_condi}, with $\bar K$
given by \eqref{exp:K_positive}, is exactly the position equation
\eqref{lin_pos_eq}. Under \eqref{lin_pos_range} and
\eqref{lin_pos_mono}, it determines a unique $\bxi^*\in(0,L)$.

For this position $\bxi^*$,
Lemma~\ref{lem_lin_solva} yields a unique downstream solution
$\ml{\dU}_+$. The two tangential Rankine-Hugoniot conditions then
determine $\nabla_{x'}\dot\xi$ through
\eqref{def_lin_shock_shape}. Thus, the position
$\bxi^*$ and the shock slopes
$\nabla_{x'}\dot\xi$ are obtained successively and remain separate
components of the linear approximation.

The preceding results give the initial approximating solution
used in the nonlinear construction.

\begin{proposition}[The initial approximating solution and shock position]
\label{prop:lin}
Assume that \eqref{lin_pos_range} and
\eqref{lin_pos_mono} hold. Then there exists a unique
linear approximating solution
\begin{equation}\label{lin_solution}
\bigl(
\ml{\dU}_-,\ml{\dU}_+;
\bxi^*,\nabla_{x'}\dot\xi
\bigr)
\end{equation}
with the following properties:
\begin{enumerate}
\item
The initial approximating position $\bxi^*\in(0,L)$ is the unique
solution of \eqref{comp_condi}.

\item
The upstream perturbation $\ml{\dU}_-$ is the unique solution of
\eqref{lin_up_eq}-\eqref{lin_up_wall}. Moreover,
\begin{equation}\label{lin_up_const}
\dS_-=\dB_-=\dot\om_{1-}=0
\qquad\text{in }\ml N.
\end{equation}

\item
For the position $\bxi^*$ determined above, the downstream
perturbation $\ml{\dU}_+$ is the unique solution of
\eqref{lin_down_eq} subject to the linearized Rankine-Hugoniot
conditions, the slip boundary conditions and the prescribed exit
pressure. Moreover,
\begin{equation}\label{lin_down_const}
\dB_+=\dot\om_{1+}=0
\qquad\text{in }\dN_+.
\end{equation}

\item
The tangential derivatives of the linearized shock are determined by
\begin{equation}
\p_{x_i}\dot\xi
=\frac{\brho\bu}{[\bp]}
\bigl(\du_{i+}(\bxi^*,x')-\du_{i-}(\bxi^*,x')\bigr),
\qquad i=2,3.
\end{equation}
\end{enumerate}
Furthermore,
\begin{equation}\label{est_lin_approx}
\|\ml{\dU}_-\|_{2,\al;\ml N}
+
\|\ml{\dU}_+\|_{2,\al;\dN_+}
+
\|\nabla_{x'}\dot\xi\|_{2,\al;\ml C}
\leq C\si,
\end{equation}
where $C$ depends only on the background states, the nozzle and the
prescribed data.
\end{proposition}

The linear approximation determines a reference position and
the leading flow and shock slopes. These quantities specify
the neighborhood for the nonlinear construction.

\section{Reformulation of the nonlinear problem on the fixed domain}
\label{sec:nonlinear}
We turn to the nonlinear shock problem. The upstream solution is first constructed in the whole nozzle. The downstream equations and the shock conditions are then written
on the fixed domain associated with the position $\bxi^*$ found in
Section~2.

\subsection{The nonlinear supersonic flow}
\label{sub:up}

The upstream flow can be constructed independently of the unknown
shock. The initial boundary value problem in the $x_1$ direction
is solved first, and the solution is then compared with the linear
approximation from Section~2.

To write the upstream equations in symmetric form, use the
velocity, pressure and entropy as unknowns:
\begin{equation}
\label{up_W}
\bm W_-
:=
(u_{1-},u_{2-},u_{3-},p_-,S_-)^{\mathsf T},
\qquad
\rho_-=\lrb{\frac{p_-}{S_-}}^{\frac1\ga},
\qquad
c_-^2=\frac{\ga p_-}{\rho_-}.
\end{equation}
In these variables, the steady Euler equations are equivalent to
\begin{equation}
\label{up_euler}
\begin{cases}
\rho_-(\u_-\cdot\nabla)\u_-+\nabla p_-
=-\si\rho_-\nabla\Phi,\\
\displaystyle
\frac{1}{\rho_-c_-^2}\u_-\cdot\nabla p_-
+\div\u_-=0,\\
\u_-\cdot\nabla S_-=0.
\end{cases}
\end{equation}
Collecting the coefficients of each derivative gives the matrix form
\begin{equation}
\label{up_sys}
\sum_{i=1}^{3}A_i(\bm W_-)\p_{x_i}\bm W_-
=\bm F(\bm W_-;\Phi),
\end{equation}
with coefficient matrices and force vector
\begin{equation}
\label{up_matrices}
A_i(\bm W_-)
=
\begin{pmatrix}
\rho_-u_{i-}I_3&\bm e_i&\bm 0\\
\bm e_i^{\mathsf T}&\dfrac{u_{i-}}{\ga p_-}&0\\
\bm 0^{\mathsf T}&0&u_{i-}
\end{pmatrix},
\qquad
\bm F(\bm W_-;\Phi)
=
\begin{pmatrix}
-\si\rho_-\nabla\Phi\\0\\0
\end{pmatrix}.
\end{equation}
Here $\bm e_i$ is the $i$th standard basis vector of $\mathbb R^3$,
and $I_3$ is the $3\times3$ identity matrix. It follows from
\eqref{up_matrices} that all the matrices $A_i$
are symmetric and that $A_1$ is positive definite whenever
$p_->0$, $\rho_->0$ and $u_{1-}>c_->0$. Thus
\eqref{up_sys} is symmetric hyperbolic
in the $x_1$-direction near the background supersonic state
\begin{equation}
\label{up_W_bg}
\bar{\bm W}_-
:=
(\bu_-,0,0,\bp_-,\bS_-)^{\mathsf T}.
\end{equation}

The entrance and wall conditions are
\begin{equation}
\label{up_inlet}
\bm W_-(0,x')=\bar{\bm W}_-,
\qquad x'\in\ml C,
\end{equation}
and
\begin{equation}
\label{up_wall}
\begin{cases}
u_{2-}(x_1,\pm1,x_3)=0,
&(x_1,x_3)\in[0,L]\times[-1,1],\\
u_{3-}(x_1,x_2,\pm1)=0,
&(x_1,x_2)\in[0,L]\times[-1,1].
\end{cases}
\end{equation}
\begin{proposition}[The nonlinear supersonic flow]
\label{prop_up}
Let $\Phi$ satisfy the assumptions above. There exists a constant
$\si_0>0$, depending only on $\ml{\bU}_-$, $L$ and
$\|\Phi\|_{3,\al;\ml N}$, such that, for $0<\si\leq\si_0$, system
\eqref{up_sys} with conditions
\eqref{up_inlet} and
\eqref{up_wall} admits a unique solution
\begin{equation}\label{up_reg}
\bm W_-\in C^{2,\al}(\overline{\ml N})
\end{equation}
in a sufficiently small neighborhood of $\bar{\bm W}_-$.
The corresponding upstream state is
\begin{equation}
\label{up_state}
\ml U_-=(\u_-,S_-,B_-),
\qquad
B_-=
\frac12|\u_-|^2
+\frac{\ga}{\ga-1}\frac{p_-}{\rho_-}
+\si\Phi
\end{equation}
and its perturbation by
\begin{equation}\label{up_pert}
V_-:=\ml U_- -\ml{\bU}_-.
\end{equation}
The entropy and Bernoulli function remain constant, and the
velocity is irrotational:
\begin{equation}
\label{up_invar}
S_-\equiv\bS_-,
\qquad
B_-\equiv\bB,
\qquad
\curl\u_-\equiv0
\qquad\text{in }\ml N.
\end{equation}
Moreover,
\begin{equation}
\label{up_supersonic}
u_{1-}\geq\frac12\bu_-,
\qquad
u_{1-}^2-c_-^2
\geq\frac12\lrb{\bu_-^2-\bc_-^2}>0
\qquad\text{in }\ml N,
\end{equation}
and
\begin{equation}
\label{up_est}
\|\bm W_- -\bar{\bm W}_-\|_{2,\al;\ml N}
+\|V_-\|_{2,\al;\ml N}
\leq C\|\si\Phi\|_{3,\al;\ml N}
\leq C\si.
\end{equation}
For the linear upstream solution in Lemma~\ref{lem_lin_up}, we have
\begin{equation}
\label{up_quad_est}
\|V_- -\ml{\dU}_-\|_{1,\al;\ml N}
\leq C\si^2.
\end{equation}
In addition, the following boundary identities hold:
\begin{equation}
\label{up_wall_comp}
\begin{cases}
u_{2-}=\p_{x_2}^2u_{2-}=0,
\quad
\p_{x_2}(u_{1-},u_{3-},p_-,\rho_-)=0,
&\text{on }x_2=\pm1,\\
u_{3-}=\p_{x_3}^2u_{3-}=0,
\quad
\p_{x_3}(u_{1-},u_{2-},p_-,\rho_-)=0,
&\text{on }x_3=\pm1.
\end{cases}
\end{equation}
Here and below, $C$ is independent of $\si$.
\end{proposition}

\begin{proof}
Extend $\Phi$ evenly across $x_2=\pm1$ and $x_3=\pm1$, and then
extend it $4$-periodically in both tangential variables. Conditions
\eqref{phi_wall} and
\eqref{phi_wall3} ensure that the resulting
extension belongs to $C^{3,\al}$. Correspondingly, across an
$x_2$-wall, we extend $p_-,u_{1-},u_{3-},S_-$ evenly and $u_{2-}$
oddly; across an $x_3$-wall, we extend
$p_-,u_{1-},u_{2-},S_-$ evenly and $u_{3-}$ oddly. The system and
the constant entrance data are invariant under these two reflections.

Apply the classical solution theory for quasilinear symmetric
hyperbolic systems to the periodic problem, as in
\cite[Lemma~1.1]{WengXin2026a}; see also
\cite{BenzoniGavageSerre2007}.
A continuation argument on the fixed interval
$[0,L]$ gives a unique solution satisfying
\begin{equation}
\label{up_W_est}
\|\bm W_- -\bar{\bm W}_-\|_{2,\al;\ml N}
\leq C\|\si\Phi\|_{3,\al;\ml N}
\leq C\si,
\end{equation}
provided that $\si_0$ is sufficiently small. This estimate gives
\eqref{up_supersonic}. The reflection symmetry yields
\eqref{up_wall} and \eqref{up_wall_comp}.

The entropy equation and its entrance condition are
\begin{equation}
\label{up_S_eq}
\begin{cases}
\u_-\cdot\nabla S_-=0,
&\text{in }\ml N,\\
S_-(0,x')=\bS_-,
&x'\in\ml C.
\end{cases}
\end{equation}
Since $u_{1-}>0$ and the velocity is tangent to the nozzle wall, every
streamline can be traced back to the entrance. Hence
$S_-\equiv\bS_-$. Taking the scalar product of the momentum
equation with $\u_-$ and using the entropy equation gives the
Bernoulli transport problem
\begin{equation}
\label{up_B_eq}
\begin{cases}
\u_-\cdot\nabla B_-=0,
&\text{in }\ml N,\\
B_-(0,x')=\bB,
&x'\in\ml C,
\end{cases}
\end{equation}
where the entrance value follows from
\eqref{phi_entrance}. The same characteristic argument gives
$B_-\equiv\bB$.

Let $\bm\om_-:=\curl\u_-$. Using the definitions of $B_-$ and $S_-$,
we rewrite the momentum equation as
\begin{equation}
\label{up_momentum}
\u_-\times\bm\om_-
=
\nabla B_-
-\frac{\rho_-^{\ga-1}}{\ga-1}\nabla S_-.
\end{equation}
It follows that $\u_-\times\bm\om_-=0$. Since $u_{1-}>0$, we have
\begin{equation}
\label{up_om23}
\om_{2-}=\frac{u_{2-}}{u_{1-}}\om_{1-},
\qquad
\om_{3-}=\frac{u_{3-}}{u_{1-}}\om_{1-}.
\end{equation}
Substituting \eqref{up_om23} into
$\div\bm\om_-=0$ yields
\begin{equation}
\label{up_om1_eq}
\begin{cases}
\displaystyle
\lrb{
\p_{x_1}
+\frac{u_{2-}}{u_{1-}}\p_{x_2}
+\frac{u_{3-}}{u_{1-}}\p_{x_3}
}\om_{1-}
+\lrb{
\p_{x_2}\lrb{\frac{u_{2-}}{u_{1-}} }
+\p_{x_3}\lrb{\frac{u_{3-}}{u_{1-}} }
}\om_{1-}=0,
&\text{in }\ml N,\\
\om_{1-}(0,x')=0,
&x'\in\ml C.
\end{cases}
\end{equation}
The entrance condition follows from \eqref{up_inlet}.
Equation \eqref{up_om1_eq} gives $\om_{1-}\equiv0$, and
\eqref{up_om23} then yields $\bm\om_-\equiv0$.
This proves \eqref{up_invar}.

For the initial approximating solution constructed in Lemma~\ref{lem_lin_up}, define
\begin{equation}
\label{lin_up_W}
\dot{\bm W}_-
:=
(\du_{1-},\du_{2-},\du_{3-},\dot p_-,\dS_-)^{\mathsf T},
\end{equation}
where \eqref{Exp_pressure} gives the linearized pressure
\begin{equation}\label{lin_up_p}
\dot p_-
=-\brho_-\bu_-\du_{1-}
-\frac{\bp_-}{(\ga-1)\bS_-}\dS_-
+\brho_-\dB_-
-\si\brho_-\Phi.
\end{equation}
Using \eqref{lin_up_eq} and \eqref{lin_up_invar}, a direct
computation shows that $\dot{\bm W}_-$ satisfies
\begin{equation}
\label{lin_up_sys}
\sum_{i=1}^{3}
A_i(\bar{\bm W}_-)\p_{x_i}\dot{\bm W}_-
=
\bm F(\bar{\bm W}_-;\Phi),
\qquad
\dot{\bm W}_-(0,x')=0,
\end{equation}
with the homogeneous wall conditions \eqref{lin_up_wall}.
In particular,
\begin{equation}
\label{lin_up_W_invar}
\dS_-=\dB_-=0,
\qquad
\dot{\u}_-=(\du_{1-},\du_{2-},\du_{3-}).
\end{equation}

The error after subtracting the linear approximation is
\begin{equation}
\label{up_Z}
\bm Z_- :=
\bm W_- -\bar{\bm W}_- -\dot{\bm W}_-.
\end{equation}
Subtracting \eqref{lin_up_sys} from \eqref{up_sys} gives
\begin{equation}\label{up_Z_sys}
\begin{split}
\sum_{i=1}^{3}A_i(\bar{\bm W}_-)\p_{x_i}\bm Z_-=&
-\sum_{i=1}^{3}
\bigl(A_i(\bm W_-)-A_i(\bar{\bm W}_-)\bigr)
\p_{x_i}(\bm W_- -\bar{\bm W}_-)\\
&+\bm F(\bm W_-;\Phi)-\bm F(\bar{\bm W}_-;\Phi).
\end{split}
\end{equation}
For $i=2,3$, let $Z_{i-}$ denote the $i$th component of $\bm Z_-$.
The boundary conditions are
\begin{equation}
\label{up_Z_eq}
\begin{cases}
\bm Z_-(0,x')=0,
&x'\in\ml C,\\
Z_{2-}(x_1,\pm1,x_3)=0,
&(x_1,x_3)\in[0,L]\times[-1,1],\\
Z_{3-}(x_1,x_2,\pm1)=0,
&(x_1,x_2)\in[0,L]\times[-1,1].
\end{cases}
\end{equation}
The linear estimate and \eqref{up_W_est} give
\begin{equation}
\label{up_Z_est}
\|\bm Z_-\|_{1,\al;\ml N}
\leq C\lrb{
\|\bm W_- -\bar{\bm W}_-\|_{2,\al;\ml N}^2
+\si\|\Phi\|_{2,\al;\ml N}
\|\bm W_- -\bar{\bm W}_-\|_{1,\al;\ml N}
}
\leq C\si^2.
\end{equation}
Finally, \eqref{up_invar} and
\eqref{lin_up_W_invar} imply
\begin{equation*}
V_-
=
(u_{1-}-\bu_-,u_{2-},u_{3-},0,0),
\qquad
\ml{\dU}_-
=
(\du_{1-},\du_{2-},\du_{3-},0,0).
\end{equation*}
Thus \eqref{up_Z_est} yields
\eqref{up_quad_est}, which completes the proof.
\end{proof}

\subsection{The change of variables and reformulation in the fixed domain}
\label{sub:fixed}

Let $\ml U_-$ be the solution obtained in
Proposition~\ref{prop_up}. The downstream equations and the
shock conditions are now written on the fixed domain $\dN_+$.

\subsubsection{The change of coordinates}

The change of coordinates sends the shock to $y_1=\bxi^*$ and
keeps the exit and the lateral walls fixed.

Let the shock be $x_1=\xi(x')$, where $x'=(x_2,x_3)\in\ml C$, and
suppose that $\xi$ is sufficiently close to $\bxi^*$. The downstream
domain $\ml N_+$ in \eqref{def_N+} is transformed by
\begin{equation}\label{shock_flatten}
y_1=\bxi^*+\frac{L-\bxi^*}{L-\xi(x')}(x_1-\xi(x')),
\qquad y'=x'.
\end{equation}
The inverse transformation is
\begin{equation}\label{shock_inverse}
x_1=X_1(\y;\xi)
:=y_1+\frac{L-y_1}{L-\bxi^*}\bigl(\xi(y')-\bxi^*\bigr),
\qquad x'=y'.
\end{equation}
Write $X_\xi(\y)=(X_1(\y;\xi),y')$. The transformed domain is
$\dN_+=(\bxi^*,L)\times\ml C$, with boundary faces
\begin{equation}\label{fix_faces}
\begin{aligned}
\Ga_s^*&=\{\bxi^*\}\times\ml C,
&\Ga_e&=\{L\}\times\ml C,\\
\Ga_2^\pm&=(\bxi^*,L)\times\{\pm1\}\times[-1,1],
&\Ga_3^\pm&=(\bxi^*,L)\times[-1,1]\times\{\pm1\}.
\end{aligned}
\end{equation}
The downstream state in the new coordinates is denoted by
\begin{equation*}
\ml U(\y):=\ml U_+(X_\xi(\y)).
\end{equation*}
By the chain rule,
\begin{equation}\label{def_Dxi}
\begin{aligned}
\p_{x_1}=\ml D_1^\xi
&:=\frac{L-\bxi^*}{L-\xi(y')}\p_{y_1},\\
\p_{x_i}=\ml D_i^\xi
&:=\p_{y_i}
-\frac{L-y_1}{L-\xi(y')}\p_{y_i}\xi(y')\p_{y_1},
\qquad i=2,3.
\end{aligned}
\end{equation}
The potential is evaluated at $\Phi(X_1(\y;\xi),y')$. On the
shock, downstream functions are evaluated at $(\bxi^*,y')$, and
the upstream solution is evaluated at $(\xi(y'),y')$.

\subsubsection{The Rankine-Hugoniot conditions}

The jump conditions determine boundary values for the downstream
state and two equations for the shock derivatives.

With the functions $G_i$ and $\ml S_i$ defined in \eqref{def_mlS_23} and \eqref{def_G01_0}, the shock
conditions become
\begin{equation}\label{fix_RH}
\begin{cases}
G_i\bigl(\ml U(\bxi^*,y'),\ml U_-(\xi(y'),y');
\Phi(\xi(y'),y')\bigr)=0,
&i=0,1,\\
B(\bxi^*,y')=B_-(\xi(y'),y'),\\
\p_{y_i}\xi(y')
=\ml S_i\bigl(\ml U(\bxi^*,y'),\ml U_-(\xi(y'),y');
\Phi(\xi(y'),y')\bigr),
&i=2,3.
\end{cases}
\end{equation}
At the background states and zero force, the denominator in the
formulas for $\ml S_2,\ml S_3$ is
\begin{equation}\label{fix_slope_det}
J_{22}J_{33}-J_{23}J_{32}=[\bp]^2>0.
\end{equation}
Thus the two tangential R-H conditions determine
$\p_{y_2}\xi$ and $\p_{y_3}\xi$ for states close to the background.
The derivative of the remaining three conditions with respect to
the downstream variables $(u_1,S,B)$ has determinant
\begin{equation}\label{fix_RH_det}
\left.
\det\frac{\p(G_0,G_1,[B])}{\p(u_1,S,B)}
\right|_{\substack{\ml U_+=\ml{\bU}_+\\\si\Phi=0}}
=-\frac{\brho_+\bp_+}{(\ga-1)\bS_+}(1-\bM_+^2)\neq0.
\end{equation}
Since $B_-=\bB$, we have
\begin{equation*}
B(\bxi^*,y')=\bB.
\end{equation*}
The implicit function theorem applied to $G_0=G_1=0$ with
$B=\bB$ gives two smooth functions $g_1$ and $g_4$ such that
\begin{equation}\label{fix_shock_data}
\begin{aligned}
u_1(\bxi^*,y')-\bu_+
&=g_1\bigl(u_2(\bxi^*,y'),u_3(\bxi^*,y'),
\ml U_-(\xi(y'),y');\Phi(\xi(y'),y')\bigr),\\
S(\bxi^*,y')-\bS_+
&=g_4\bigl(u_2(\bxi^*,y'),u_3(\bxi^*,y'),
\ml U_-(\xi(y'),y');\Phi(\xi(y'),y')\bigr).
\end{aligned}
\end{equation}
Using \eqref{up_invar}, the linear terms given in
Lemma~\ref{RH_bd_condi} yield
\begin{equation}\label{fix_g1_exp}
\begin{aligned}
g_1
={}&b_u^s\bigl(u_{1-}(\xi(y'),y')-\bu_-\bigr)
+\si b_\Phi^u\Phi(\xi(y'),y')\\
&+R_1\bigl(u_2(\bxi^*,y'),u_3(\bxi^*,y'),
\ml U_-(\xi(y'),y'),\si\Phi(\xi(y'),y')\bigr),
\end{aligned}
\end{equation}
and
\begin{equation}\label{fix_g4_exp}
\begin{aligned}
g_4
={}&b_s^s\bigl(u_{1-}(\xi(y'),y')-\bu_-\bigr)
+\si b_\Phi^s\Phi(\xi(y'),y')\\
&+R_4\bigl(u_2(\bxi^*,y'),u_3(\bxi^*,y'),
\ml U_-(\xi(y'),y'),\si\Phi(\xi(y'),y')\bigr).
\end{aligned}
\end{equation}
Here $R_1$ and $R_4$ are remainders of second order.

Substituting \eqref{fix_shock_data} and $B=\bB$ into the two
shock equations in \eqref{fix_RH} gives
\begin{equation}\label{fix_slope_exp}
\begin{aligned}
\p_{y_i}\xi(y')
={}&a_0\bigl(u_i(\bxi^*,y')-u_{i-}(\xi(y'),y')\bigr)\\
&+R_i\bigl(u_2(\bxi^*,y'),u_3(\bxi^*,y'),
\ml U_-(\xi(y'),y'),\si\Phi(\xi(y'),y')\bigr),
\quad i=2,3,
\end{aligned}
\end{equation}
where $a_0$ is given by \eqref{def_a0}, and $R_2,R_3$ are
remainders of second order.

On the walls, the slip conditions and the formulas for
$\ml S_2,\ml S_3$ give
\begin{equation}\label{shock_wall}
\p_{y_2}\xi(\pm1,y_3)=0,
\qquad \p_{y_3}\xi(y_2,\pm1)=0.
\end{equation}

\subsubsection{The Bernoulli function and the entropy}

The transformed Bernoulli transport equation is
\begin{equation}\label{fix_B_eq}
\sum_{i=1}^3u_i\ml D_i^\xi B=0
\qquad\text{in }\dN_+,
\end{equation}
with $B(\bxi^*,y')=\bB$ from the Bernoulli jump condition.
For a flow and shock close to the background, the coefficient of
$\p_{y_1}$ is positive. The slip conditions keep the characteristics
inside the nozzle, so every backward characteristic meets
$\Ga_s^*$. Consequently,
\begin{equation}\label{fix_B}
B\equiv\bB\qquad\text{in }\dN_+.
\end{equation}
The entropy satisfies
\begin{equation}\label{fix_S_eq}
\sum_{i=1}^3u_i\ml D_i^\xi S=0
\qquad\text{in }\dN_+,
\end{equation}
with the boundary condition obtained in
\eqref{fix_shock_data}:
\begin{equation}\label{fix_S_bc}
S(\bxi^*,y')-\bS_+
=g_4\bigl(u_2(\bxi^*,y'),u_3(\bxi^*,y'),
\ml U_-(\xi(y'),y');\Phi(\xi(y'),y')\bigr).
\end{equation}

\subsubsection{The vorticity field}

In the new coordinates, equation~\eqref{eq_om1} for the first
vorticity component becomes
\begin{equation}\label{fix_om1_eq}
\begin{aligned}
&\left(\ml D_1^\xi+\frac{u_2}{u_1}\ml D_2^\xi
+\frac{u_3}{u_1}\ml D_3^\xi\right)\omega_1
+\omega_1\sum_{i=2}^3\ml D_i^\xi\left(\frac{u_i}{u_1}\right)\\
=&
-\frac{\ml D_2^\xi S}{\ga-1}
\ml D_3^\xi\left(\frac{\rho^{\ga-1}}{u_1}\right)
+\frac{\ml D_3^\xi S}{\ga-1}
\ml D_2^\xi\left(\frac{\rho^{\ga-1}}{u_1}\right).
\end{aligned}
\end{equation}

The two equations for $\xi$ in \eqref{fix_RH}
describe the same shock surface. Equating their mixed derivatives
gives
\begin{equation}\label{shape_compat}
\begin{aligned}
&\p_{y_2}\ml S_3\bigl(\ml U(\bxi^*,y'),
\ml U_-(\xi(y'),y');\Phi(\xi(y'),y')\bigr)\\
=&
\p_{y_3}\ml S_2\bigl(\ml U(\bxi^*,y'),
\ml U_-(\xi(y'),y');\Phi(\xi(y'),y')\bigr).
\end{aligned}
\end{equation}
Utilizing \eqref{fix_slope_exp},
condition \eqref{shape_compat} becomes
\begin{equation}\label{fix_curl_bc}
\begin{aligned}
&(\p_{y_2}u_3-\p_{y_3}u_2)(\bxi^*,y')\\
={}&\p_{y_2}\bigl(u_{3-}(\xi(y'),y')\bigr)
-\p_{y_3}\bigl(u_{2-}(\xi(y'),y')\bigr)\\
&-\frac1{a_0}\p_{y_2}R_3(y')
+\frac1{a_0}\p_{y_3}R_2(y').
\end{aligned}
\end{equation}
By the definition of $\omega_1$ and \eqref{def_Dxi}, we obtain
\begin{equation}\label{trace_curl}
\begin{aligned}
&(\p_{y_2}u_3-\p_{y_3}u_2)(\bxi^*,y')\\
=&\omega_1(\bxi^*,y')
+\frac{L-\bxi^*}{L-\xi(y')}
\Bigl(\p_{y_2}\xi(y')\p_{y_1}u_3(\bxi^*,y')
-\p_{y_3}\xi(y')\p_{y_1}u_2(\bxi^*,y')\Bigr).
\end{aligned}
\end{equation}
Combining \eqref{fix_curl_bc} and \eqref{trace_curl} gives the
boundary condition
\begin{equation}\label{fix_om1_bc}
\begin{aligned}
\omega_1(\bxi^*,y')
={}&\p_{y_2}\bigl(u_{3-}(\xi(y'),y')\bigr)
-\p_{y_3}\bigl(u_{2-}(\xi(y'),y')\bigr)\\
&-\frac1{a_0}\p_{y_2}R_3(y')
+\frac1{a_0}\p_{y_3}R_2(y')\\
&-\frac{L-\bxi^*}{L-\xi(y')}
\Bigl(\p_{y_2}\xi(y')\p_{y_1}u_3(\bxi^*,y')
-\p_{y_3}\xi(y')\p_{y_1}u_2(\bxi^*,y')\Bigr).
\end{aligned}
\end{equation}

Once $\omega_1$ is known, the other two components are determined
by \eqref{eq_om23} and $B=\bB$:
\begin{equation}\label{fix_om23}
\begin{aligned}
\omega_2&=\frac{u_2}{u_1}\omega_1
-\frac{\rho^{\ga-1}}{(\ga-1)u_1}\ml D_3^\xi S,\\
\omega_3&=\frac{u_3}{u_1}\omega_1
+\frac{\rho^{\ga-1}}{(\ga-1)u_1}\ml D_2^\xi S.
\end{aligned}
\end{equation}
These transport and algebraic relations are coupled to the
velocity equations below.

\subsubsection{The deformation-curl system for the velocity}

The deformation equation \eqref{def_eq} and the curl equations
become
\begin{equation}\label{fix_dc_eqs}
\begin{cases}
\displaystyle\sum_{i,j=1}^3(c^2\de_{ij}-u_iu_j)\ml D_i^\xi u_j
=\displaystyle\si\sum_{i=1}^3u_i
\ml D_i^\xi\bigl(\Phi(X_1(\y;\xi),y')\bigr),\\
\ml D_2^\xi u_3-\ml D_3^\xi u_2=\omega_1,\\
\ml D_3^\xi u_1-\ml D_1^\xi u_3=\omega_2,\\
\ml D_1^\xi u_2-\ml D_2^\xi u_1=\omega_3,
\end{cases}
\qquad\text{in }\dN_+.
\end{equation}
Here $\rho$ is given by \eqref{def_rho_p} with $B=\bB$ and
$\Phi$ evaluated at $(X_1(\y;\xi),y')$, and $c^2=\ga S\rho^{\ga-1}$.

At the shock, the boundary value of $u_1-\bu_+$ is given by
$g_1$ in \eqref{fix_shock_data}. At $y_1=L$, we use
$p=\bp_++\si p_{ex}$ and $B=\bB$ to solve for the positive value
of $u_1$. This defines the exit boundary function
\begin{equation}\label{def_exit_map}
\begin{aligned}
&g_e(S,u_2,u_3;\Phi,p_{ex})(y')\\
\quad:=\Bigg(&
2\bigl(\bB-\si\Phi(L,y')\bigr)
-\frac{2\ga}{\ga-1}S(L,y')^{1/\ga}
\bigl(\bp_++\si p_{ex}(y')\bigr)^{(\ga-1)/\ga}\\
&-u_2(L,y')^2-u_3(L,y')^2
\Bigg)^{1/2}-\bu_+.
\end{aligned}
\end{equation}
The boundary conditions for \eqref{fix_dc_eqs} are
\begin{equation}\label{fix_vel_bc}
\begin{cases}
u_1(\bxi^*,y')-\bu_+
=g_1\bigl(u_2(\bxi^*,y'),u_3(\bxi^*,y'),
\ml U_-(\xi(y'),y');\Phi(\xi(y'),y')\bigr),\\
u_1(L,y')-\bu_+=g_e(S,u_2,u_3;\Phi,p_{ex})(y'),\\
u_2=0\quad\text{on }\Ga_2^\pm,
\qquad u_3=0\quad\text{on }\Ga_3^\pm.
\end{cases}
\end{equation}
The integral solvability condition for this boundary value problem
will determine the shock location in the next section.

The following result uses the linear approximation
$\ml{\dU}_+$ and $\nabla_{y'}\dot\xi$
from Proposition~\ref{prop:lin} on $\dN_+$ and
$\ml C$, respectively.

\begin{theorem}\label{thm_fix_domain}
Under the assumptions of Theorem~\ref{thm_nonlin}, there are
constants $M>1$, $r_*>0$, and $\si_*>0$ such that, for
$0<\si\leq\si_*$, the transformed problem formulated in
Subsection~\ref{sub:fixed} has a unique solution
$(\ml U,\xi)$ in the following class:
\begin{equation}\label{fix_solution_class}
\begin{aligned}
&\|\ml U-\ml{\bU}_+\|_{2,\al;\dN_+}
+\|\nabla_{y'}\xi\|_{2,\al;\ml C}\leq M\si,\\
&\|\ml U-\ml{\bU}_+-\ml{\dU}_+\|_{1,\al;\dN_+}
+\|\nabla_{y'}\xi-\nabla_{y'}\dot\xi\|_{1,\al;\ml C}
\leq\si^{3/2},\\
&|\xi(1,1)-\bxi^*|\leq r_*\si.
\end{aligned}
\end{equation}
Here $\ml U=(u_1,u_2,u_3,S,B)\in
C^{2,\al}(\overline{\dN_+})$, and
$\xi\in C^{3,\al}(\overline{\ml C})$.
The class also satisfies the wall compatibility conditions
\begin{equation}\label{fix_wall_comp}
\begin{cases}
u_2=\p_{y_2}^2u_2=0,\quad
\p_{y_2}(u_1,u_3,S,B)=0,
&\text{on }\Ga_2^\pm,\\
u_3=\p_{y_3}^2u_3=0,\quad
\p_{y_3}(u_1,u_2,S,B)=0,
&\text{on }\Ga_3^\pm,
\end{cases}
\end{equation}
and
\begin{equation}\label{fix_xi_wall}
\begin{cases}
\p_{y_2}\xi=\p_{y_2}^3\xi=0,
&\text{on }y_2=\pm1,\\
\p_{y_3}\xi=\p_{y_3}^3\xi=0,
&\text{on }y_3=\pm1.
\end{cases}
\end{equation}

The solution satisfies $B=\bB$ and the sharper estimate
\begin{equation}\label{lin_cmp}
\|\ml U-\ml{\bU}_+-\ml{\dU}_+\|_{1,\al;\dN_+}
+\|\nabla_{y'}\xi-\nabla_{y'}\dot\xi\|_{1,\al;\ml C}
\leq C\si^2.
\end{equation}
\end{theorem}

The next section proves this theorem. Returning to the original
coordinates then gives Theorem~\ref{thm_nonlin}.
\section{The nonlinear iteration}
\label{sec:iter}

Theorem~\ref{thm_fix_domain} is proved by iteration. Since $y'=x'$,
the data $\dot g_1,\dot g_4,\dot g_e$ on
$\ml C$ retain the notation of Section~2.
The first step is to specify the unknowns and the iteration space.

\subsection{The iteration setting}
\label{sub:iter_space}

Set
\begin{equation}\label{it_state}
V:=\ml U-\ml{\bU}_+=(V_1,V_2,V_3,V_4,V_5).
\end{equation}
Since $B=\bB$ by \eqref{fix_B}, we have $V_5=0$, and
\begin{equation*}
\u=\bu_+\bm e_1+(V_1,V_2,V_3),\qquad S=\bS_++V_4.
\end{equation*}
We write $\xi(y')=\bxi^*+\ddot\xi(y')$ and set
\begin{equation}\label{it_shock_variables}
\ddot\xi^*:=\ddot\xi(1,1),\qquad
q_i:=\p_{y_i}\ddot\xi=\p_{y_i}\xi,\quad i=2,3,
\qquad \bm q:=(q_2,q_3).
\end{equation}
The two derivatives satisfy
\begin{equation}\label{it_shape_comp}
\p_{y_3}q_2=\p_{y_2}q_3\qquad\text{in }\ml C.
\end{equation}
Conversely, for every pair satisfying \eqref{it_shape_comp},
the shock is recovered from $(\ddot\xi^*,\bm q)$ by
\begin{equation}\label{it_shock_formula}
\xi(y_2,y_3)=\bxi^*+\ddot\xi^*
+\int_1^{y_2}q_2(s,1)\,ds
+\int_1^{y_3}q_3(y_2,s)\,ds.
\end{equation}
In particular, this formula has derivatives $q_2,q_3$ and takes
the value $\bxi^*+\ddot\xi^*$ at $(1,1)$.

The iteration variables are $(V,\bm q)$. We require
$V\in C^{2,\al}(\overline{\dN}_+)$,
$q_i\in C^{2,\al}(\overline{\ml C})$, $i=2,3$, and $V_5=0$.
The wall conditions are
\begin{equation}\label{it_shape_wall}
\begin{cases}
q_2=\p_{y_2}^2q_2=0,&y_2=\pm1,\\
q_3=\p_{y_3}^2q_3=0,&y_3=\pm1,
\end{cases}
\end{equation}
and
\begin{equation}\label{it_wall2}
V_2=\p_{y_2}V_1=\p_{y_2}V_3
=\p_{y_2}V_4=\p_{y_2}^2V_2=0
\qquad\text{on }\Ga_2^\pm.
\end{equation}
On the other pair of walls, the corresponding conditions are
\begin{equation}\label{it_wall3}
V_3=\p_{y_3}V_1=\p_{y_3}V_2
=\p_{y_3}V_4=\p_{y_3}^2V_3=0
\qquad\text{on }\Ga_3^\pm.
\end{equation}
Let $\mk X$ denote the class satisfying these conditions and
\eqref{it_shape_comp}. For a constant $M>1$, to be chosen in the
estimates below, define
\begin{equation}\label{it_high_space}
\mk X_{M\si}:=\left\{
(V,\bm q)\in\mk X:\,
\|V\|_{2,\al;\dN_+}
+\sum_{i=2}^3\|q_i\|_{2,\al;\ml C}\leq M\si
\right\}.
\end{equation}
The linear solution defines the smaller neighborhood
\begin{equation}\label{it_space}
\mk K:=\left\{
(V,\bm q)\in\mk X_{M\si}:\,
\|V-\ml{\dU}_+\|_{1,\al;\dN_+}
+\sum_{i=2}^3\|q_i-\p_{y_i}\dot\xi\|_{1,\al;\ml C}
\leq\si^{3/2}
\right\}.
\end{equation}
For $0<\si\leq1$, \eqref{est_lin_approx} implies
\begin{equation}\label{it_low_bound}
\|\widehat V\|_{1,\al;\dN_+}
+\sum_{i=2}^3\|\widehat q_i\|_{1,\al;\ml C}
\leq C\si,
\end{equation}
for $(\widehat V,\widehat{\bm q})\in\mk K$, with $C$ independent of $M$.
The $C^{2,\al}$ bound is used to construct the updated functions,
and their differences are estimated in $C^{1,\al}$.

Fix an input $(\widehat V,\widehat{\bm q})\in\mk K$, where
$\widehat{\bm q}=(\widehat q_2,\widehat q_3)$, and set
\begin{equation}\label{it_input_state}
\widehat\u:=\bu_+\bm e_1+
(\widehat V_1,\widehat V_2,\widehat V_3).
\end{equation}
For the input shock derivatives $\widehat{\bm q}$, define
\begin{equation}\label{it_input_shock}
\widehat\xi(y_2,y_3;\ddot\xi^*)=\bxi^*+\ddot\xi^*
+\int_1^{y_2}\widehat q_2(s,1)\,ds
+\int_1^{y_3}\widehat q_3(y_2,s)\,ds.
\end{equation}
The displacement $\ddot\xi^*$ is determined for each input
$(\widehat V,\widehat{\bm q})$ by the solvability condition
in Step~3. In Step~4, the updated shock is reconstructed
using this displacement and the updated derivatives $\bm q$.
Before the displacement is determined, the construction and
estimates hold uniformly
for sufficiently small $|\ddot\xi^*|$, so that $0<\widehat\xi<L$.
In particular,
\begin{equation*}
\|\widehat\xi-\bxi^*\|_{3,\al;\ml C}
\leq |\ddot\xi^*|+C\sum_{i=2}^3
\|\widehat q_i\|_{2,\al;\ml C}
\leq |\ddot\xi^*|+CM\si.
\end{equation*}
The dependence of the iteration data on the input, the prescribed
data and the displacement is displayed as, for example,
\begin{equation*}
\ddot g_4(\widehat V,\widehat{\bm q};
\ml U_-,\Phi;\ddot\xi^*)(y').
\end{equation*}
After specifying these arguments, we shorten this to
$\ddot g_4(y')$ when no confusion can arise.

\subsection{Construction of the iteration map}
\label{sub:iter_map}

The construction begins with the entropy and vorticity transport
problems. The solvability condition for the enlarged deformation-curl
system then determines the displacement, after which the velocity
is obtained from that system. Finally, the two shock derivatives
are updated. In the first two steps, all quantities retain their
dependence on the displacement until it is determined.

\noindent
\textit{Step~1. The entropy.}
For the given input, \eqref{fix_S_eq} gives the entropy equation
\begin{equation}\label{it_S_eq}
\p_{y_1}V_4+\widehat b_2\p_{y_2}V_4
+\widehat b_3\p_{y_3}V_4=0\qquad\text{in }\dN_+,
\end{equation}
where
\begin{equation}\label{it_S_coeff}
\widehat b_i:=\frac{\widehat u_i}{
\displaystyle
\frac{L-\bxi^*}{L-\widehat\xi(y')}\widehat u_1
-\frac{L-y_1}{L-\widehat\xi(y')}
\sum_{j=2}^3\widehat u_j\widehat q_j},\qquad i=2,3.
\end{equation}
The denominator is $\bu_++O(\si+|\ddot\xi^*|)>0$, and
\begin{equation*}
\sum_{i=2}^3\|\widehat b_i\|_{2,\al;\dN_+}\leq C\si.
\end{equation*}
The boundary condition is
\begin{equation}\label{it_S_bc}
V_4(\bxi^*,y')=\ddot g_4(y'),\qquad y'\in\ml C,
\end{equation}
where \eqref{fix_g4_exp} gives
\begin{equation}\label{it_S_data}
\begin{split}
\ddot g_4(\widehat V,\widehat{\bm q};
\ml U_-,\Phi;\ddot\xi^*)(y')
:=&g_4\bigl(
\widehat V_2(\bxi^*,y'),\widehat V_3(\bxi^*,y'),
\ml U_-(\widehat\xi(y'),y');\Phi(\widehat\xi(y'),y')\bigr)\\
=&b_s^sV_{1-}(\widehat\xi(y'),y')
+\si b_\Phi^s\Phi(\widehat\xi(y'),y')\\
&+\ml R_4\bigl(\widehat V_2(\bxi^*,y'),
\widehat V_3(\bxi^*,y'),V_-(\widehat\xi(y'),y'),
\si\Phi(\widehat\xi(y'),y')\bigr).
\end{split}
\end{equation}
Here $\ml R_4$ denotes $R_4$ evaluated at the iteration
arguments displayed above. Thus,
\begin{equation}\label{rem_S}
\begin{split}
\|\ml R_4\|_{2,\al;\ml C}
&\leq C\Bigl(
\|\widehat V\|_{2,\al;\dN_+}
+\|V_-\|_{2,\al;\ml N}
+\si\|\Phi\|_{2,\al;\ml N}\Bigr)^2
\leq C\si^2,\\
\|\ddot g_4\|_{2,\al;\ml C}
&\leq C\Bigl(\|V_-\|_{2,\al;\ml N}
+\si\|\Phi\|_{2,\al;\ml N}
+\|\ml R_4\|_{2,\al;\ml C}\Bigr)
\leq C\si.
\end{split}
\end{equation}
The reflection properties of the Rankine-Hugoniot conditions and
\eqref{up_wall_comp} yield
\begin{equation}\label{it_S_wall_data}
\begin{cases}
\p_{y_2}\ddot g_4=\p_{y_2}\ml R_4=0,
&\text{on }y_2=\pm1,\\
\p_{y_3}\ddot g_4=\p_{y_3}\ml R_4=0,
&\text{on }y_3=\pm1.
\end{cases}
\end{equation}

For $\y=(y_1,y')$, let $\bm Y(t;\y)=(Y_2(t;\y),Y_3(t;\y))$
solve the characteristic equations
\begin{equation}\label{S_chars}
\frac{dY_i}{dt}=\widehat b_i(t,\bm Y(t;\y)),
\qquad Y_i(y_1;\y)=y_i,\qquad i=2,3,
\end{equation}
and set
\begin{equation*}
\bm\beta(\y;\widehat\xi,\widehat V)
:=\bm Y(\bxi^*;\y).
\end{equation*}
Thus $\bm\beta$ gives the transverse coordinates where the
characteristic reaches $y_1=\bxi^*$. Since $\widehat b_2=0$ on
$y_2=\pm1$ and $\widehat b_3=0$ on $y_3=\pm1$, the
characteristics remain in $\overline{\ml C}$. We obtain
\begin{equation}\label{S_formula}
V_4(\y)=\ddot g_4\bigl(
\bm\beta(\y;\widehat\xi,\widehat V)\bigr).
\end{equation}
The characteristic equations and this formula imply
\begin{equation*}
\begin{split}
\|\bm\beta(\cdot;\widehat\xi,\widehat V)-y'\|_{2,\al;\dN_+}
&\leq C\sum_{i=2}^3\|\widehat b_i\|_{2,\al;\dN_+}\leq C\si,\\
\|V_4\|_{2,\al;\dN_+}
&\leq C\|\ddot g_4\|_{2,\al;\ml C}\leq C\si.
\end{split}
\end{equation*}
The wall conditions imply
\begin{equation*}
\p_{y_2}V_4=0\quad\text{on }\Ga_2^\pm,
\qquad
\p_{y_3}V_4=0\quad\text{on }\Ga_3^\pm.
\end{equation*}
The extensions across the walls give the stated regularity
up to their intersections.

To compare $V_4$ with the linear entropy, we use
$\p_{y_1}\dS_+=0$ and subtract its equation from
\eqref{it_S_eq}. This gives
\begin{equation}\label{S_lin_diff}
\p_{y_1}(V_4-\dS_+)
+\sum_{i=2}^3\widehat b_i\p_{y_i}(V_4-\dS_+)
=-\sum_{i=2}^3\widehat b_i\p_{y_i}\dS_+
\quad\text{in }\dN_+,
\end{equation}
with boundary value
\begin{equation*}
\begin{split}
(V_4-\dS_+)(\bxi^*,y')
=&\ddot g_4(y')-\dot g_4(y')\\
=&b_s^s\bigl(V_{1-}(\widehat\xi(y'),y')
-\du_{1-}(\bxi^*,y')\bigr)\\
&+\si b_\Phi^s\bigl(\Phi(\widehat\xi(y'),y')
-\Phi(\bxi^*,y')\bigr)+\ml R_4(y').
\end{split}
\end{equation*}
By \eqref{up_quad_est} and the fundamental theorem of calculus,
\begin{equation*}
\begin{split}
\|\ddot g_4-\dot g_4\|_{1,\al;\ml C}
&\leq C\Bigl(
\|V_- -\ml{\dU}_-\|_{1,\al;\ml N}
+\si\|\widehat\xi-\bxi^*\|_{1,\al;\ml C}
+\|\ml R_4\|_{1,\al;\ml C}\Bigr)\\
&\leq C\bigl(\si^2+\si|\ddot\xi^*|\bigr).
\end{split}
\end{equation*}
The transport estimate applied to
\eqref{S_lin_diff} now yields
\begin{equation}\label{est_it_S}
\begin{split}
\|V_4-\dS_+\|_{1,\al;\dN_+}
&\leq C\left(
\|\ddot g_4-\dot g_4\|_{1,\al;\ml C}
+\sum_{i=2}^3\|\widehat b_i\p_{y_i}\dS_+\|_{1,\al;\dN_+}
\right)\\
&\leq C\bigl(\si^2+\si|\ddot\xi^*|\bigr).
\end{split}
\end{equation}

\medskip
\noindent
\textit{Step~2. The vorticity field.}
Equation \eqref{fix_om1_eq} gives the transport equation
\begin{equation}\label{it_om1_eq}
\p_{y_1}\om_1+\widehat b_2\p_{y_2}\om_1
+\widehat b_3\p_{y_3}\om_1+\mu\om_1=f_\om
\qquad\text{in }\dN_+,
\end{equation}
where $\widehat b_i$ are given by \eqref{it_S_coeff}, and
\begin{equation}\label{it_om_coeff}
\begin{aligned}
\mu&:=\frac1\chi\sum_{i=2}^3
\ml D_i^{\widehat\xi}\left(\frac{\widehat u_i}{\widehat u_1}\right),
\qquad f_\om:=\frac{\ml F_\om}{\chi},\\
\chi&:=\frac{L-\bxi^*}{L-\widehat\xi(y')}
-\frac{L-y_1}{L-\widehat\xi(y')}
\sum_{i=2}^3\frac{\widehat u_i}{\widehat u_1}\widehat q_i.
\end{aligned}
\end{equation}
Here $\chi>0$ by the positivity of the denominator in
\eqref{it_S_coeff}. The source term is
\begin{equation}\label{it_Fomega}
\begin{aligned}
\ml F_\om
:=\ml D_2^{\widehat\xi}
\left(\frac{\rho^{\ga-1}}{(\ga-1)\widehat u_1}\right)
\ml D_3^{\widehat\xi}S
-\ml D_3^{\widehat\xi}
\left(\frac{\rho^{\ga-1}}{(\ga-1)\widehat u_1}\right)
\ml D_2^{\widehat\xi}S,
\end{aligned}
\end{equation}
where
\begin{equation*}
S:=\bS_++V_4,\qquad
\rho(\y):=\rho\bigl(\bB,S,|\widehat\u|^2,
\Phi(X_{\widehat\xi}(\y))\bigr).
\end{equation*}

Substituting the updated tangential velocity into the linear part of
\eqref{fix_slope_exp} and evaluating the remainders at the input yields
\begin{equation}\label{shock_slopes}
\begin{split}
q_i(y')={}&a_0\bigl(V_i(\bxi^*,y')
-V_{i-}(\widehat\xi(y'),y')\bigr)\\
&+\ml R_i\bigl(\widehat V_2(\bxi^*,y'),
\widehat V_3(\bxi^*,y'),V_-(\widehat\xi(y'),y'),
\si\Phi(\widehat\xi(y'),y')\bigr),\qquad i=2,3.
\end{split}
\end{equation}
Here $\ml R_i$ denotes $R_i$ evaluated at the displayed arguments, for $i=2,3$.

As in \eqref{rem_S},
\begin{equation}\label{rem_slopes}
\sum_{i=2}^3\|\ml R_i\|_{2,\al;\ml C}
\leq C\Bigl(\|\widehat V\|_{2,\al;\dN_+}
+\|V_-\|_{2,\al;\ml N}
+\si\|\Phi\|_{2,\al;\ml N}\Bigr)^2
\leq C\si^2.
\end{equation}
Their wall identities are
\begin{equation}\label{rem_slopes_wall}
\begin{cases}
\ml R_2=\p_{y_2}^2\ml R_2=\p_{y_2}\ml R_3=0,&y_2=\pm1,\\
\ml R_3=\p_{y_3}^2\ml R_3=\p_{y_3}\ml R_2=0,&y_3=\pm1.
\end{cases}
\end{equation}
These follow by reflecting the two tangential momentum conditions
across each wall. We write $\ml R_i(y')$ for the evaluations in
\eqref{shock_slopes} when displaying the derivatives below.

Assuming \eqref{it_shape_comp} holds, we obtain
\begin{equation}\label{it_curl_data}
\begin{aligned}
&(\p_{y_2}V_3-\p_{y_3}V_2)(\bxi^*,y')\\
&\quad=\p_{y_2}\bigl(V_{3-}(\widehat\xi(y'),y')\bigr)
-\p_{y_3}\bigl(V_{2-}(\widehat\xi(y'),y')\bigr)\\
&\qquad-\frac1{a_0}\p_{y_2}\ml R_3(y')
+\frac1{a_0}\p_{y_3}\ml R_2(y').
\end{aligned}
\end{equation}
Evaluating the nonlinear term in \eqref{trace_curl}
at the input functions gives
\begin{equation}\label{it_curl_shock}
\begin{aligned}
&(\p_{y_2}V_3-\p_{y_3}V_2)(\bxi^*,y')\\
&\quad=\om_1(\bxi^*,y')
+\frac{L-\bxi^*}{L-\widehat\xi(y')}
\bigl(\widehat q_2\p_{y_1}\widehat V_3
-\widehat q_3\p_{y_1}\widehat V_2\bigr)(\bxi^*,y').
\end{aligned}
\end{equation}
Combining \eqref{it_curl_data} and \eqref{it_curl_shock}
yields the boundary condition
\begin{equation}\label{it_om1_bc}
\om_1(\bxi^*,y')=\ddot g_\om(y'),\qquad y'\in\ml C,
\end{equation}
where
\begin{equation}\label{it_om_data}
\begin{split}
&\ddot g_\om(\widehat V,\widehat{\bm q};
\ml U_-,\Phi;\ddot\xi^*)(y')\\
=&\p_{y_2}\bigl(V_{3-}(\widehat\xi(y'),y')\bigr)
-\p_{y_3}\bigl(V_{2-}(\widehat\xi(y'),y')\bigr)
-\frac1{a_0}\p_{y_2}\ml R_3(y')\\
&+\frac1{a_0}\p_{y_3}\ml R_2(y')
-\frac{L-\bxi^*}{L-\widehat\xi(y')}
\bigl(\widehat q_2\p_{y_1}\widehat V_3
-\widehat q_3\p_{y_1}\widehat V_2\bigr)(\bxi^*,y').
\end{split}
\end{equation}
By \eqref{up_invar}, the upstream velocity
is irrotational. Differentiating its values along the shock gives
\begin{equation*}
\begin{split}
&\p_{y_2}\bigl(V_{3-}(\widehat\xi(y'),y')\bigr)
-\p_{y_3}\bigl(V_{2-}(\widehat\xi(y'),y')\bigr)\\
=&\widehat q_2\p_{x_1}V_{3-}(\widehat\xi(y'),y')
-\widehat q_3\p_{x_1}V_{2-}(\widehat\xi(y'),y').
\end{split}
\end{equation*}
Thus, by \eqref{rem_slopes},
\begin{equation*}
\|\ddot g_\om\|_{1,\al;\ml C}
\leq C\left(
\sum_{i=2}^3\|\widehat q_i\|_{1,\al;\ml C}
\bigl(\|V_-\|_{2,\al;\ml N}
+\|\widehat V\|_{2,\al;\dN_+}\bigr)
+\sum_{i=2}^3\|\ml R_i\|_{2,\al;\ml C}
\right)\leq C\si^2.
\end{equation*}
Similarly, subtracting the constant background coefficient inside
\eqref{it_Fomega} gives
\begin{equation*}
\|f_\om\|_{1,\al;\dN_+}
\leq C\left\|
\frac{\rho^{\ga-1}}{(\ga-1)\widehat u_1}
-\frac{\brho_+^{\ga-1}}{(\ga-1)\bu_+}
\right\|_{2,\al;\dN_+}
\|V_4\|_{2,\al;\dN_+}\leq C\si^2.
\end{equation*}

Using the same characteristics as in
\eqref{S_chars}, we obtain
\begin{equation}\label{om_formula}
\begin{split}
\om_1(\y)
={}&\ddot g_\om\bigl(\bm\beta(\y;\widehat\xi,\widehat V)\bigr)
\exp\left(-\int_{\bxi^*}^{y_1}
\mu(t,\bm Y(t;\y))\,dt\right)\\
&+\int_{\bxi^*}^{y_1}f_\om(\tau,\bm Y(\tau;\y))
\exp\left(-\int_\tau^{y_1}
\mu(t,\bm Y(t;\y))\,dt\right)\,d\tau.
\end{split}
\end{equation}
Since $\dot\om_{1+}=0$, the transport equation and
its boundary data give
\begin{equation}\label{est_it_vort}
\|\om_1-\dot\om_{1+}\|_{1,\al;\dN_+}
\leq C\left(\|\ddot g_\om\|_{1,\al;\ml C}
+\|f_\om\|_{1,\al;\dN_+}\right)
\leq C\si^2.
\end{equation}
The boundary data and source satisfy
\begin{equation*}
\ddot g_\om=0\quad\text{on }\p\ml C,
\qquad f_\om=0\quad\text{on }\Ga_2^\pm\cup\Ga_3^\pm.
\end{equation*}
Restricting \eqref{it_om1_eq} to the walls therefore gives
\begin{equation}\label{it_om_wall}
\om_1=0\qquad\text{on }\Ga_2^\pm\cup\Ga_3^\pm.
\end{equation}
The remaining two components are given by \eqref{fix_om23}:
\begin{equation}\label{it_om23}
\begin{split}
\om_2&:=\frac{\widehat u_2}{\widehat u_1}\om_1
-\frac{\rho^{\ga-1}}{(\ga-1)\widehat u_1}
\ml D_3^{\widehat\xi}S,\\
\om_3&:=\frac{\widehat u_3}{\widehat u_1}\om_1
+\frac{\rho^{\ga-1}}{(\ga-1)\widehat u_1}
\ml D_2^{\widehat\xi}S.
\end{split}
\end{equation}
This completes the construction of the vorticity for the given
input and displacement.

\medskip
\noindent
\textit{Step~3. The deformation-curl system for the velocity.}
For each displacement $\ddot\xi^*$, separating the linear
principal part of \eqref{fix_dc_eqs} at the background state
gives the velocity system
\begin{equation}\label{it_vel_eqs}
\begin{cases}
(1-\bM_+^2)\p_{y_1}V_1+\p_{y_2}V_2+\p_{y_3}V_3=\ml F_0,\\
\p_{y_2}V_3-\p_{y_3}V_2=\ml G_1,\\
\p_{y_3}V_1-\p_{y_1}V_3=\ml G_2,\\
\p_{y_1}V_2-\p_{y_2}V_1=\ml G_3,
\end{cases}
\qquad\text{in }\dN_+.
\end{equation}
With $B=\bB$, the sound speed is
\begin{equation*}
\widehat c^2=(\ga-1)\left(
\bB-\frac12|\widehat\u|^2
-\si\Phi(X_1(\y;\widehat\xi),y')\right).
\end{equation*}
The remaining terms define the source
\begin{equation}\label{it_F0}
\begin{aligned}
\ml F_0(\widehat V,\widehat{\bm q};\Phi;\ddot\xi^*)(\y)
=&(1-\bM_+^2)\p_{y_1}\widehat V_1
+\p_{y_2}\widehat V_2+\p_{y_3}\widehat V_3\\
&-\frac{1}{\bc_+^2}
\sum_{i,j=1}^3(\widehat c^2\de_{ij}-\widehat u_i\widehat u_j)
\ml D_i^{\widehat\xi}\widehat u_j
+\frac{\si}{\bc_+^2}\sum_{i=1}^3\widehat u_i
\ml D_i^{\widehat\xi}\bigl(\Phi(X_1(\y;\widehat\xi),y')\bigr).
\end{aligned}
\end{equation}
For use in the estimates, set
\begin{equation*}
\widehat a_{ij}:=\frac{\widehat c^2\de_{ij}
-\widehat u_i\widehat u_j}{\bc_+^2},
\qquad \bar a_{ij}:=\de_{ij}-\bM_+^2\de_{i1}\de_{j1}.
\end{equation*}
Separating the coefficient and coordinate corrections gives
\begin{equation}\label{it_F0_split}
\begin{aligned}
\ml F_0={}&\sum_{i,j=1}^3(\bar a_{ij}-\widehat a_{ij})
\p_{y_i}\widehat V_j
+\sum_{i,j=1}^3\widehat a_{ij}
(\p_{y_i}-\ml D_i^{\widehat\xi})\widehat V_j
+\frac{\si}{\bc_+^2}\sum_{i=1}^3\widehat u_i
\ml D_i^{\widehat\xi}\bigl(\Phi(X_1(\y;\widehat\xi),y')\bigr).
\end{aligned}
\end{equation}
The curl equations in \eqref{fix_dc_eqs}, with the vorticity
from Step~2, give
\begin{equation}\label{it_G123}
\begin{aligned}
\ml G_1(\y):={}&\om_1(\y)
+\frac{L-y_1}{L-\widehat\xi(y')}
(\widehat q_2\p_{y_1}\widehat V_3
-\widehat q_3\p_{y_1}\widehat V_2),\\
\ml G_2(\y):={}&\om_2(\y)
+\frac{L-y_1}{L-\widehat\xi(y')}
\widehat q_3\p_{y_1}\widehat V_1
+\left(\frac{L-\bxi^*}{L-\widehat\xi(y')}-1\right)
\p_{y_1}\widehat V_3,\\
\ml G_3(\y):={}&\om_3(\y)
-\frac{L-y_1}{L-\widehat\xi(y')}
\widehat q_2\p_{y_1}\widehat V_1
-\left(\frac{L-\bxi^*}{L-\widehat\xi(y')}-1\right)
\p_{y_1}\widehat V_2.
\end{aligned}
\end{equation}
At $y_1=\bxi^*$, \eqref{fix_g1_exp} gives
\begin{equation}\label{it_shock_data}
\begin{aligned}
\ddot g_1(\widehat V,\widehat{\bm q};
\ml U_-,\Phi;\ddot\xi^*)(y')
=&g_1\bigl(\widehat V_2(\bxi^*,y'),
\widehat V_3(\bxi^*,y'),\ml U_-(\widehat\xi(y'),y');
\Phi(\widehat\xi(y'),y')\bigr)\\
=&b_u^sV_{1-}(\widehat\xi(y'),y')
+\si b_\Phi^u\Phi(\widehat\xi(y'),y')\\
&+\ml R_1\bigl(\widehat V_2(\bxi^*,y'),
\widehat V_3(\bxi^*,y'),V_-(\widehat\xi(y'),y'),
\si\Phi(\widehat\xi(y'),y')\bigr).
\end{aligned}
\end{equation}
Here $\ml R_1$ denotes $R_1$ evaluated at the displayed arguments.
As in \eqref{rem_S},
\begin{equation}\label{rem_vel}
\|\ml R_1\|_{2,\al;\ml C}
\leq C\bigl(\|\widehat V\|_{2,\al;\dN_+}
+\|V_-\|_{2,\al;\ml N}
+\si\|\Phi\|_{2,\al;\ml N}\bigr)^2\leq C\si^2.
\end{equation}
At $y_1=L$, expanding \eqref{def_exit_map} with the entropy
from Step~1 gives
\begin{equation}\label{it_exit_data}
\begin{aligned}
\ddot g_e(\widehat V,\widehat{\bm q};
\ml U_-,\Phi,p_{ex};\ddot\xi^*)(y')
=&g_e(\bS_++V_4,\widehat V_2,\widehat V_3;\Phi,p_{ex})(y')\\
=&-\frac{\bp_+}{(\ga-1)\bS_+\brho\bu}V_4(L,y')
-\frac{\si}{\brho\bu}p_{ex}(y')\\
&-\frac{\si}{\bu_+}\Phi(L,y')+\ml R_e(y').
\end{aligned}
\end{equation}
Here $\ml R_e(y')$ is the remainder of second order in the
expansion of \eqref{def_exit_map}, evaluated at
\begin{equation*}
\bigl(V_4(L,y'),\widehat V_2(L,y'),\widehat V_3(L,y'),
\si p_{ex}(y'),\si\Phi(L,y')\bigr).
\end{equation*}
Thus,
\begin{equation}\label{rem_exit}
\begin{split}
\|\ml R_e\|_{2,\al;\ml C}
\leq& C\bigl(\|V_4\|_{2,\al;\dN_+}
+\|\widehat V\|_{2,\al;\dN_+}
+\si\|p_{ex}\|_{2,\al;\ml C}
+\si\|\Phi\|_{2,\al;\ml N}\bigr)^2\\
\leq& C\si^2.
\end{split}
\end{equation}
Thus \eqref{it_vel_eqs} is endowed with
\begin{equation}\label{it_vel_bc}
\begin{cases}
V_1(\bxi^*,y')=\ddot g_1(y'),\qquad
V_1(L,y')=\ddot g_e(y'),\\
V_2=0\quad\text{on }\Ga_2^\pm,\qquad
V_3=0\quad\text{on }\Ga_3^\pm.
\end{cases}
\end{equation}
The wall conditions for the input and the upstream solution give
\begin{equation}\label{it_shock_wall_data}
\begin{cases}
\p_{y_2}\ddot g_1=\p_{y_2}\ml R_1=0,
&\text{on }y_2=\pm1,\\
\p_{y_3}\ddot g_1=\p_{y_3}\ml R_1=0,
&\text{on }y_3=\pm1,
\end{cases}
\end{equation}
and
\begin{equation}\label{it_exit_wall}
\begin{cases}
\p_{y_2}\ddot g_e=\p_{y_2}\ml R_e=0,
&\text{on }y_2=\pm1,\\
\p_{y_3}\ddot g_e=\p_{y_3}\ml R_e=0,
&\text{on }y_3=\pm1.
\end{cases}
\end{equation}
Moreover, \eqref{it_om_data} and \eqref{it_G123} imply
\begin{equation}\label{it_G1_shock}
\begin{aligned}
\ml G_1(\bxi^*,y')={}&
\p_{y_2}\bigl(V_{3-}(\widehat\xi(y'),y')\bigr)
-\p_{y_3}\bigl(V_{2-}(\widehat\xi(y'),y')\bigr)\\
&-\frac{1}{a_0}\p_{y_2}\ml R_3(y')
+\frac{1}{a_0}\p_{y_3}\ml R_2(y').
\end{aligned}
\end{equation}
The data also satisfy
\begin{equation}\label{it_G_wall}
\begin{cases}
\ml G_1=\ml G_3=\p_{y_2}\ml G_2=0,&\text{on }\Ga_2^\pm,\\
\ml G_1=\ml G_2=\p_{y_3}\ml G_3=0,&\text{on }\Ga_3^\pm,
\end{cases}
\end{equation}
and
\begin{equation}\label{it_F0_wall}
\p_{y_2}\ml F_0=0\quad\text{on }\Ga_2^\pm,
\qquad \p_{y_3}\ml F_0=0\quad\text{on }\Ga_3^\pm.
\end{equation}

The curl equations in \eqref{it_vel_eqs} require the
compatibility condition
\begin{equation}\label{it_curl_comp}
\p_{y_1}\ml G_1+\p_{y_2}\ml G_2+\p_{y_3}\ml G_3=0.
\end{equation}

For a given iteration input, this condition cannot hold in general.
An auxiliary function $\Pi$ is therefore introduced, giving the
enlarged system
\begin{equation}\label{it_dc}
\begin{cases}
(1-\bM_+^2)\p_{y_1}V_1
+\p_{y_2}V_2+\p_{y_3}V_3=\ml F_0,
&\text{in }\dN_+,\\
\p_{y_2}V_3-\p_{y_3}V_2+\p_{y_1}\Pi=\ml G_1,
&\text{in }\dN_+,\\
\p_{y_3}V_1-\p_{y_1}V_3+\p_{y_2}\Pi=\ml G_2,
&\text{in }\dN_+,\\
\p_{y_1}V_2-\p_{y_2}V_1+\p_{y_3}\Pi=\ml G_3,
&\text{in }\dN_+.
\end{cases}
\end{equation}
The velocity satisfies \eqref{it_vel_bc}, and the boundary conditions for $\Pi$
are given by
\begin{equation}\label{it_Pi_bc}
\p_{y_1}\Pi=0\quad\text{on }\Ga_s^*,
\qquad
\Pi=0\quad\text{on }\Ga_e\cup\Ga_2^\pm\cup\Ga_3^\pm.
\end{equation}

\begin{remark}\label{rem_Pi}
The auxiliary function $\Pi$ need not vanish during the iteration. At the
fixed point, \eqref{fix_Pi_eq} with the homogeneous boundary
conditions \eqref{it_Pi_bc} gives $\Pi=0$. Thus
\eqref{it_dc} reduces to \eqref{it_vel_eqs}, and the original
velocity equations are recovered.
\end{remark}

The following linear result gives the solvability and estimates
needed for this step.
\begin{lemma}\label{lem_it_dc}
Let $k=0$ or $1$. Suppose that
$F_0\in C^{k,\al}(\overline{\dN_+})$,
$\bm F\in C^{k,\al}(\overline{\dN_+})^3$, and
$h_1,h_e\in C^{k+1,\al}(\overline{\ml C})$.
Assume the wall compatibility conditions
\begin{equation*}
\begin{cases}
F_1=F_3=0 &\text{on }\Ga_2^\pm,\\
F_1=F_2=0 &\text{on }\Ga_3^\pm,\\
\p_{y_i}h_1=\p_{y_i}h_e=0
&\text{on }y_i=\pm1,\quad i=2,3.
\end{cases}
\end{equation*}
For $k=1$, assume in addition that
\begin{equation*}
\p_{y_2}F_0=\p_{y_2}F_2=0\quad\text{on }\Ga_2^\pm,
\qquad
\p_{y_3}F_0=\p_{y_3}F_3=0\quad\text{on }\Ga_3^\pm.
\end{equation*}
Then the problem
\begin{equation}\label{lin_mixed_dc}
\begin{cases}
(1-\bM_+^2)\p_{y_1}w_1
+\p_{y_2}w_2+\p_{y_3}w_3=F_0,
&\text{in }\dN_+,\\
\curl\w+\nabla\Pi=\bm F,
&\text{in }\dN_+,
\end{cases}
\end{equation}
subject to
\begin{equation}\label{lin_mixed_bc}
\begin{cases}
w_1=h_1,
&\text{on }\Ga_s^*,\\
w_1=h_e,
&\text{on }\Ga_e,\\
w_2=0,
&\text{on }\Ga_2^\pm,\\
w_3=0,
&\text{on }\Ga_3^\pm,\\
\p_{y_1}\Pi=0,
&\text{on }\Ga_s^*,\\
\Pi=0,
&\text{on }\Ga_e\cup\Ga_2^\pm\cup\Ga_3^\pm
\end{cases}
\end{equation}
has a unique solution if and only if
\begin{equation}\label{dc_compat}
(1-\bM_+^2)
\int_{\ml C}(h_e-h_1)\,dy'
=
\int_{\dN_+}F_0\,d\y.
\end{equation}
Moreover,
\begin{equation}\label{est_it_dc}
\begin{split}
\|\w\|_{k+1,\al;\dN_+}
+\|\Pi\|_{k+1,\al;\dN_+}
\leq& C\Bigl(
\|F_0\|_{k,\al;\dN_+}
+\|\bm F\|_{k,\al;\dN_+}\\
&+\|h_1\|_{k+1,\al;\ml C}
+\|h_e\|_{k+1,\al;\ml C}
\Bigr).
\end{split}
\end{equation}
Across $y_2=\pm1$, the functions $w_1,w_3$ extend evenly and
$w_2,\Pi$ extend oddly; across $y_3=\pm1$, $w_1,w_2$ extend
evenly and $w_3,\Pi$ extend oddly.
\end{lemma}

\begin{proof}
Integrating the first equation and using the prescribed normal
velocity gives the necessity of \eqref{dc_compat}.
Assume \eqref{dc_compat}. Following the enlarged deformation-curl
formulation in \cite{WengXin2026a}, first solve
\begin{equation}\label{it_Pi_problem}
\begin{cases}
\Delta\Pi=\div\bm F,
&\text{in }\dN_+,\\
\p_{y_1}\Pi=0,
&\text{on }\Ga_s^*,\\
\Pi=0,
&\text{on }\Ga_e\cup\Ga_2^\pm\cup\Ga_3^\pm.
\end{cases}
\end{equation}
For $k=0$, this boundary value problem is understood weakly:
for every test function
$\zeta$ vanishing on
$\Ga_e\cup\Ga_2^\pm\cup\Ga_3^\pm$,
\begin{equation*}
\int_{\dN_+}\nabla\Pi\cdot\nabla\zeta\,d\y
=
\int_{\dN_+}\bm F\cdot\nabla\zeta\,d\y
+\int_{\Ga_s^*}F_1\zeta\,dy'.
\end{equation*}
The Dirichlet part of the boundary is nonempty, and the mixed
problem has a unique solution satisfying
\begin{equation*}
\|\Pi\|_{k+1,\al;\dN_+}
\leq C\|\bm F\|_{k,\al;\dN_+}.
\end{equation*}
To obtain this estimate up to the boundary, extend $\Pi$ oddly
across both $y_2=\pm1$ and $y_3=\pm1$. The assumed extensions reduce the
estimates where walls meet to the usual interior and boundary
estimates; see \cite{GilbargTrudinger2001}. The components of $\nabla\Pi$
have the same even and odd extensions as $(F_1,F_2,F_3)$. Hence
$\bm F^d:=\bm F-\nabla\Pi$ has the same property.

The equation for $\Pi$ gives
\begin{equation*}
\div\bm F^d=0,
\qquad
\int_{\p\dN_+}\bm F^d\cdot\n\,dS=0.
\end{equation*}
The next step is the standard div-curl construction; see
\cite{KozonoYanagisawa2009} for the Helmholtz-Weyl decomposition.
Here we use the wall reflections to apply the estimates up to the
edges. Since $\dN_+$ is simply connected and its boundary is
connected, there is a unique vector field $\w^d$ satisfying
\begin{equation}\label{it_hodge_part}
\div\w^d=0,
\qquad
\curl\w^d=\bm F^d,
\qquad
\w^d\cdot\n=0
\quad\text{on }\p\dN_+,
\end{equation}
with the stated lateral extensions. To justify the estimate up to
the edges, the field can be extended as
$(\mathrm{even},\mathrm{odd},\mathrm{even})$ across $y_2=\pm1$
and as $(\mathrm{even},\mathrm{even},\mathrm{odd})$ across $y_3=\pm1$.
The equations imply $-\Delta\w^d=\curl\bm F^d$. On the two
axial faces $y_1=\bxi^*,L$, the boundary conditions become
\begin{equation*}
w_1^d=0,\qquad
\p_{y_1}w_2^d=F_3^d,\qquad
\p_{y_1}w_3^d=-F_2^d.
\end{equation*}
After reflection, the component equations have ordinary Dirichlet
or Neumann boundary conditions near these faces. The interior and
boundary estimates, in divergence form when $k=0$, therefore give
\begin{equation*}
\|\w^d\|_{k+1,\al;\dN_+}
\leq C\|\bm F^d\|_{k,\al;\dN_+}.
\end{equation*}

Write $\w=\w^d+\nabla\phi$. The function $\phi$ satisfies
\begin{equation}\label{it_phi_problem}
\begin{cases}
(1-\bM_+^2)\p_{y_1}^2\phi
+\p_{y_2}^2\phi+\p_{y_3}^2\phi=F_0+F_w,
&\text{in }\dN_+,\\
\p_{y_1}\phi=h_1,
&\text{on }\Ga_s^*,\\
\p_{y_1}\phi=h_e,
&\text{on }\Ga_e,\\
\p_{y_2}\phi=0,
&\text{on }\Ga_2^\pm,\\
\p_{y_3}\phi=0,
&\text{on }\Ga_3^\pm,
\end{cases}
\end{equation}
where
\begin{equation*}
F_w=-(1-\bM_+^2)\p_{y_1}w_1^d
-\p_{y_2}w_2^d-\p_{y_3}w_3^d.
\end{equation*}
Since $\w^d\cdot\n=0$ on the boundary, integration gives
\begin{equation*}
\int_{\dN_+}F_w\,d\y=0.
\end{equation*}
Consequently, the compatibility condition for \eqref{it_phi_problem}
is exactly \eqref{dc_compat}. We fix the additive constant in
$\phi$ by requiring $\int_{\dN_+}\phi\,d\y=0$.
Even reflection of $\phi$ across the lateral walls reduces the
edge estimates to the usual Neumann estimates, which give
\begin{equation*}
\begin{aligned}
\|\nabla\phi\|_{k+1,\al;\dN_+}
&\leq C\Bigl(
\|F_0\|_{k,\al;\dN_+}+\|F_w\|_{k,\al;\dN_+}
+\|h_1\|_{k+1,\al;\ml C}+\|h_e\|_{k+1,\al;\ml C}\Bigr)\\
&\leq C\Bigl(
\|F_0\|_{k,\al;\dN_+}+\|\w^d\|_{k+1,\al;\dN_+}
+\|h_1\|_{k+1,\al;\ml C}+\|h_e\|_{k+1,\al;\ml C}\Bigr).
\end{aligned}
\end{equation*}
Together with the estimates for $\Pi$ and $\w^d$, this proves
\eqref{est_it_dc}, including the estimates up to the walls.
The uniqueness follows from that of the three boundary value problems.
\end{proof}

For the data in \eqref{it_dc} and \eqref{it_vel_bc},
the solvability condition \eqref{dc_compat} in
Lemma~\ref{lem_it_dc} determines $\ddot\xi^*$.
It takes the form
\begin{equation}\label{pos_eq}
\ml I
\bigl(
\widehat V,\widehat{\bm q};\ml U_-,\Phi,p_{ex};\ddot\xi^*
\bigr)=0,
\end{equation}
where
\begin{equation}\label{pos_fun}
\begin{split}
\ml I
\bigl(
\widehat V,\widehat{\bm q};\ml U_-,\Phi,p_{ex};\ddot\xi^*
\bigr)
:=
-(1-\bM_+^2)\int_{\ml C}\ddot g_1\,dy'
+(1-\bM_+^2)\int_{\ml C}\ddot g_e\,dy'
-\int_{\dN_+}\ml F_0(\y)\,d\y.
\end{split}
\end{equation}
Here $\ddot g_e$ uses the entropy obtained in Step~1. The dependence
on the iteration input and displacement is omitted below when clear.

\begin{lemma}[Determination of the shock position]
\label{lem_it_pos}
There exist constants $r_*>0$ and $\si_1>0$ such that,
for $0<\si\leq\si_1$ and every
$(\widehat V,\widehat{\bm q})\in\mk K$,
equation \eqref{pos_eq} admits a solution $\ddot\xi^*$ satisfying
\begin{equation}\label{est_it_pos}
|\ddot\xi^*|\leq r_*\si.
\end{equation}
\end{lemma}

\begin{proof}
Throughout this proof, the iteration input and prescribed data are
fixed. For $\ml I$ in \eqref{pos_fun}, denote the boundary and
source terms by
\begin{equation}\label{pos_parts}
\begin{split}
\ml I^1
&:=-\lrb{1-\bM_+^2}\int_{\ml C}\ddot g_1\,dy',\\
\ml I^2
&:=\lrb{1-\bM_+^2}\int_{\ml C}\ddot g_e\,dy',\\
\ml I^3
&:=-\int_{\dN_+}\ml F_0\,d\y.
\end{split}
\end{equation}
Then $\ml I=\ml I^1+\ml I^2+\ml I^3$. The estimates below hold
for $|\ddot\xi^*|\leq r_*\si$, with $r_*>1$ to be chosen and
$\si$ sufficiently small.
By \eqref{it_input_shock},
\begin{equation}\label{pos_front}
\begin{split}
\widehat\xi(y_2,y_3)-\bxi^*-\ddot\xi^*
=
\int_1^{y_2}\widehat q_2(s_2,1)\,ds_2
+\int_1^{y_3}\widehat q_3(y_2,s_3)\,ds_3.
\end{split}
\end{equation}
The bounds on the iteration input give
\begin{equation}\label{pos_front_cmp}
\|\widehat\xi-\bxi^*-\ddot\xi^*\|_{3,\al;\ml C}
\leq C\sum_{i=2}^3\|\widehat q_i\|_{2,\al;\ml C}
\leq C\si.
\end{equation}

To estimate $\ml I^1+\ml I^2$, the upstream data in
\eqref{it_shock_data} are expanded about the linear solution at
$x_1=\bxi^*+\ddot\xi^*$. The relevant differences are
\begin{equation*}
\begin{split}
V_{1-}(\widehat\xi(y'),y')
={}&
\bigl(
V_{1-}(\widehat\xi(y'),y')
-\du_{1-}(\widehat\xi(y'),y')
\bigr)\\
&+\bigl(
\du_{1-}(\widehat\xi(y'),y')
-\du_{1-}(\bxi^*+\ddot\xi^*,y')
\bigr)\\
&+\du_{1-}(\bxi^*+\ddot\xi^*,y'),\\
\Phi(\widehat\xi(y'),y')
={}&\Phi(\bxi^*+\ddot\xi^*,y')
+\bigl(
\Phi(\widehat\xi(y'),y')
-\Phi(\bxi^*+\ddot\xi^*,y')
\bigr).
\end{split}
\end{equation*}
The first difference in the first identity is $O(\si^2)$ by
\eqref{up_quad_est}. For the second difference,
the fundamental theorem of calculus gives
\begin{equation*}
\du_{1-}(\widehat\xi(y'),y')
-\du_{1-}(\bxi^*+\ddot\xi^*,y')
=\int_{\bxi^*+\ddot\xi^*}^{\widehat\xi(y')}
\p_{x_1}\du_{1-}(t,y')\,dt.
\end{equation*}
Consequently,
\begin{equation*}
\begin{aligned}
&\|\du_{1-}(\widehat\xi(\cdot),\cdot)
-\du_{1-}(\bxi^*+\ddot\xi^*,\cdot)\|_{1,\al;\ml C}
\leq
C\|\widehat\xi-\bxi^*-\ddot\xi^*\|_{1,\al;\ml C}
\|\du_{1-}\|_{2,\al;\ml N}
\leq C\si^2.
\end{aligned}
\end{equation*}
The potential term is estimated in the same way.
Thus the remainder in the shock boundary data is
\begin{equation}
\label{def_pos_r1}
\begin{split}
r_1(y')=&b_u^s\bigl(
V_{1-}(\widehat\xi(y'),y')
-\du_{1-}(\widehat\xi(y'),y')
\bigr)\\
&+b_u^s\bigl(
\du_{1-}(\widehat\xi(y'),y')
-\du_{1-}(\bxi^*+\ddot\xi^*,y')
\bigr)\\
&+\si b_\Phi^u\bigl(
\Phi(\widehat\xi(y'),y')
-\Phi(\bxi^*+\ddot\xi^*,y')
\bigr)+\ml R_1(y').
\end{split}
\end{equation}
Similarly, the entropy remainder is
\begin{equation*}
\begin{aligned}
r_4(y')={}&\ml R_4(y')
+b_s^s\bigl(
V_{1-}(\widehat\xi(y'),y')
-\du_{1-}(\widehat\xi(y'),y')\bigr)\\
&+b_s^s\int_{\bxi^*+\ddot\xi^*}^{\widehat\xi(y')}
\p_{x_1}\du_{1-}(t,y')\,dt
+\si b_\Phi^s\int_{\bxi^*+\ddot\xi^*}^{\widehat\xi(y')}
\p_{x_1}\Phi(t,y')\,dt.
\end{aligned}
\end{equation*}
The boundary formulas now become
\begin{equation}\label{pos_shock_exp}
\begin{split}
\ddot g_1
={}&b_u^s\du_{1-}(\bxi^*+\ddot\xi^*,y')
+\si b_\Phi^u\Phi(\bxi^*+\ddot\xi^*,y')
+r_1,\\
\ddot g_4
={}&b_s^s\du_{1-}(\bxi^*+\ddot\xi^*,y')
+\si b_\Phi^s\Phi(\bxi^*+\ddot\xi^*,y')
+r_4,
\end{split}
\end{equation}
where the estimates above and the quadratic bounds for
$\ml R_1,\ml R_4$ give
\begin{equation*}
\begin{aligned}
\|r_1\|_{1,\al;\ml C}+\|r_4\|_{1,\al;\ml C}
\leq C\Bigl(&\|\ml R_1\|_{1,\al;\ml C}
+\|\ml R_4\|_{1,\al;\ml C}
+\|V_- -\ml{\dU}_-\|_{1,\al;\ml N}\\
&+\|\widehat\xi-\bxi^*-\ddot\xi^*\|_{1,\al;\ml C}
\bigl(\|\ml{\dU}_-\|_{2,\al;\ml N}
+\si\|\Phi\|_{2,\al;\ml N}\bigr)\Bigr)
\leq C\si^2.
\end{aligned}
\end{equation*}

The exit data require the entropy at $y_1=L$. Integrating
\eqref{it_S_eq} with $y'$ fixed and using its boundary
condition yields
\begin{equation*}
V_4(L,y')=\ddot g_4(y')
-\int_{\bxi^*}^{L}\sum_{i=2}^3
\widehat b_i(t,y')\p_{y_i}V_4(t,y')\,dt.
\end{equation*}
Thus
\begin{equation*}
V_4(L,y')
=b_s^s\du_{1-}(\bxi^*+\ddot\xi^*,y')
+\si b_\Phi^s\Phi(\bxi^*+\ddot\xi^*,y')
+\widetilde r_4(y'),
\end{equation*}
where we define the remainder at $y_1=L$ by
\begin{equation*}
\widetilde r_4(y')
:=r_4(y')-\int_{\bxi^*}^{L}\sum_{i=2}^3
\widehat b_i(t,y')\p_{y_i}V_4(t,y')\,dt.
\end{equation*}
Consequently,
\begin{equation*}
\|\widetilde r_4\|_{1,\al;\ml C}
\leq\|r_4\|_{1,\al;\ml C}
+C\sum_{i=2}^3\|\widehat b_i\|_{1,\al;\dN_+}
\|V_4\|_{2,\al;\dN_+}
\leq C\si^2.
\end{equation*}
Substitution into the exit expansion gives
\begin{equation}\label{pos_exit_exp}
\begin{aligned}
\ddot g_e(y')={}&-\frac{\si}{\brho\bu}p_{ex}(y')
-\frac{\si}{\bu_+}\Phi(L,y')\\
&-\frac{\bp_+}{(\ga-1)\bS_+\brho\bu}
\bigl(b_s^s\du_{1-}(\bxi^*+\ddot\xi^*,y')
+\si b_\Phi^s\Phi(\bxi^*+\ddot\xi^*,y')\bigr)
+r_e(y'),
\end{aligned}
\end{equation}
where
\begin{equation*}
r_e(y'):=\ml R_e(y')
-\frac{\bp_+}{(\ga-1)\bS_+\brho\bu}\widetilde r_4(y').
\end{equation*}
Here $\ml R_e$ is the remainder in \eqref{it_exit_data}. Hence
\begin{equation*}
\|r_e\|_{1,\al;\ml C}
\leq \|\ml R_e\|_{1,\al;\ml C}
+C\|\widetilde r_4\|_{1,\al;\ml C}
\leq C\si^2.
\end{equation*}

Substituting \eqref{pos_shock_exp} and
\eqref{pos_exit_exp} into the first two terms in
\eqref{pos_parts}, we first have
\begin{equation}\label{pos_I12_raw}
\begin{split}
\ml I^1+\ml I^2
=&(1-\bM_+^2)\int_{\ml C}
\Bigg\{
-\lrb{
b_u^s+\frac{\bp_+b_s^s}
{(\ga-1)\bS_+\brho\bu}
}
\du_{1-}(\bxi^*+\ddot\xi^*,y')\\
&
-\si\lrb{
b_\Phi^u+\frac{\bp_+b_\Phi^s}
{(\ga-1)\bS_+\brho\bu}
}
\Phi(\bxi^*+\ddot\xi^*,y')\\
&
-\frac{\si}{\brho\bu}p_{ex}(y')
-\frac{\si}{\bu_+}\Phi(L,y')
\Bigg\}\,dy'
+R_{12}^{(0)},
\end{split}
\end{equation}
where
\begin{equation*}
R_{12}^{(0)}
=\lrb{1-\bM_+^2}\int_{\ml C}(r_e-r_1)\,dy',
\qquad
|R_{12}^{(0)}|\leq C\si^2.
\end{equation*}
The linear upstream and downstream equations imply, respectively,
\begin{equation}\label{pos_lin_int}
\begin{split}
&(1-\bM_-^2)\frac{d}{dt}
\int_{\ml C}\du_{1-}(t,y')\,dy'
=\si\frac{\bu_-}{\bc_-^2}\frac{d}{dt}
\int_{\ml C}\Phi(t,y')\,dy',\\
&(1-\bM_+^2)\int_{\ml C}
\bigl(
\du_{1+}(L,y')-\du_{1+}(\bxi^*,y')
\bigr)\,dy'
=\si\frac{\bu_+}{\bc_+^2}
\int_{\ml C}
\bigl(
\Phi(L,y')-\Phi(\bxi^*,y')
\bigr)\,dy'.
\end{split}
\end{equation}

Integrating the first identity in \eqref{pos_lin_int} from
$0$ to $\bxi^*+\ddot\xi^*$ and using
$\du_{1-}(0,y')=\Phi(0,y')=0$, we obtain
\begin{equation}\label{pos_up_int}
\int_{\ml C}
\du_{1-}(\bxi^*+\ddot\xi^*,y')\,dy'
=
\si\frac{\bM_-^2}
{\bu_-(1-\bM_-^2)}
\int_{\ml C}
\Phi(\bxi^*+\ddot\xi^*,y')\,dy'.
\end{equation}

The two coefficients in \eqref{pos_I12_raw} can be simplified
using \eqref{def_RH_coefs} and \eqref{def_bue}:
\begin{equation*}
\frac{\bp_+}
{(\ga-1)\bS_+\brho\bu} b_s^s
=
\frac{(\bM_-^2-1)[\bp]}
{\brho\bu\bu_-}
=-b_u^{ex}.
\end{equation*}
Similarly, the definition of $b_\Phi^s$ gives
\begin{equation*}
\frac{\bp_+}{(\ga-1)\bS_+\brho\bu}b_\Phi^s
=\frac{b_1-\bu_+b_0}{\brho\bu}.
\end{equation*}
Consequently,
\begin{equation*}
\begin{split}
b_u^s+\frac{\bp_+}{(\ga-1)\bS_+\brho\bu}b_s^s
&= b_u^s-b_u^{ex},\\
b_\Phi^u+\frac{\bp_+}{(\ga-1)\bS_+\brho\bu}b_\Phi^s
&= b_\Phi^u+\frac{b_1-\bu_+b_0}{\brho\bu}.
\end{split}
\end{equation*}

Substituting these identities into \eqref{pos_I12_raw}, we have
\begin{equation*}
\begin{split}
\ml I^1+\ml I^2
={}&(1-\bM_+^2)(b_u^{ex}-b_u^s)
\int_{\ml C}
\du_{1-}(\bxi^*+\ddot\xi^*,y')\,dy'\\
&-\si(1-\bM_+^2)
\lrb{
b_\Phi^u+\frac{b_1-\bu_+b_0}{\brho\bu}
}
\int_{\ml C}
\Phi(\bxi^*+\ddot\xi^*,y')\,dy'\\
&-\si\frac{1-\bM_+^2}{\brho\bu}
\int_{\ml C}p_{ex}(y')\,dy'\\
&-\si\frac{1-\bM_+^2}{\bu_+}
\int_{\ml C}\Phi(L,y')\,dy'
+R_{12}^{(0)}.
\end{split}
\end{equation*}
Applying \eqref{pos_up_int} to the first term yields
\begin{equation}\label{pos_I12_up}
\begin{split}
\ml I^1+\ml I^2
={}&\si(1-\bM_+^2)
\left\{
\frac{b_u^{ex}-b_u^s}{\bu_-}
\frac{\bM_-^2}{1-\bM_-^2}
-\frac{b_1-\bu_+b_0}{\brho\bu}
-b_\Phi^u
\right\}
\int_{\ml C} \Phi(\bxi^*+\ddot\xi^*,y')\,dy'\\
&-\si\frac{1-\bM_+^2}{\brho\bu}
\int_{\ml C}p_{ex}(y')\,dy'
-\si\frac{1-\bM_+^2}{\bu_+}
\int_{\ml C}\Phi(L,y')\,dy'
+R_{12}^{(0)}.
\end{split}
\end{equation}

It follows from \eqref{def_barK} that
\begin{equation*}
\begin{split}
\frac{\bar K}{\brho \bu}
-\frac{1}{\bu_+}\frac{\bM_+^2}{1-\bM_+^2}
={}&\frac{b_u^{ex}-b_u^s}{\bu_-}
\frac{\bM_-^2}{1-\bM_-^2}
-\frac{b_1-\bu_+b_0}{\brho\bu}
-b_\Phi^u.
\end{split}
\end{equation*}
Substituting this identity into \eqref{pos_I12_up}, we obtain
\begin{equation}\label{pos_I12_K}
\begin{split}
\ml I^1+\ml I^2
={}&
\si\frac{1-\bM_+^2}{\brho\bu}
\left\{
\bar K\int_{\ml C}
\Phi(\bxi^*+\ddot\xi^*,y')\,dy'
-\int_{\ml C}p_{ex}(y')\,dy'
\right\}\\
&-\si\frac{\bM_+^2}{\bu_+}
\int_{\ml C}
\Phi(\bxi^*+\ddot\xi^*,y')\,dy'
-\si\frac{1-\bM_+^2}{\bu_+}
\int_{\ml C}\Phi(L,y')\,dy'
+R_{12}^{(0)}.
\end{split}
\end{equation}

By \eqref{pos_F} and the solvability condition,
\begin{equation*}
F(\bxi^*)=\int_{\ml C}p_{ex}(y')\,dy'.
\end{equation*}
Therefore,
\begin{equation}\label{pos_F_shift}
\begin{split}
&\bar K\int_{\ml C}
\Phi(\bxi^*+\ddot\xi^*,y')\,dy'
-\int_{\ml C}p_{ex}(y')\,dy'\\
=&
F(\bxi^*+\ddot\xi^*)-F(\bxi^*)
+\frac{\brho_+}{1-\bM_+^2}
\int_{\ml C}\Phi(L,y')\,dy'.
\end{split}
\end{equation}
The last term on the right-hand side of \eqref{pos_F_shift}
contributes
\begin{equation*}
\si\frac{1-\bM_+^2}{\brho\bu}
\frac{\brho_+}{1-\bM_+^2}
\int_{\ml C}\Phi(L,y')\,dy'
=
\frac{\si}{\bu_+}
\int_{\ml C}\Phi(L,y')\,dy'.
\end{equation*}

Combining this term with the third term in \eqref{pos_I12_K}, we obtain
\begin{equation*}
\begin{split}
&\frac{\si}{\bu_+}
\int_{\ml C}\Phi(L,y')\,dy'
-\si\frac{1-\bM_+^2}{\bu_+}
\int_{\ml C}\Phi(L,y')\,dy'\\
=&
\si\frac{\bM_+^2}{\bu_+}
\int_{\ml C}\Phi(L,y')\,dy'
=
\si\frac{\bu_+}{\bc_+^2}
\int_{\ml C}\Phi(L,y')\,dy'.
\end{split}
\end{equation*}

Therefore,
\begin{equation}\label{pos_I12_basic}
\begin{split}
\ml I^1+\ml I^2
={}&\frac{1-\bM_+^2}{\brho\bu}\si
\bigl(F(\bxi^*+\ddot\xi^*)-F(\bxi^*)\bigr)\\
&+\si\frac{\bu_+}{\bc_+^2}
\int_{\ml C}
\bigl(
\Phi(L,y')-\Phi(\bxi^*+\ddot\xi^*,y')
\bigr)\,dy'
+R_{12}^{(0)}.
\end{split}
\end{equation}
To combine these boundary terms with the source term $\ml I^3$,
use the second identity in \eqref{pos_lin_int} to rewrite
the last expression as
\begin{equation}\label{pos_I12}
\begin{split}
\ml I^1+\ml I^2
={}&\frac{1-\bM_+^2}{\brho\bu}\si
\frac{L-\bxi^*}{L-\bxi^*-\ddot\xi^*}
\bigl(
F(\bxi^*+\ddot\xi^*)-F(\bxi^*)
\bigr)\\
&-\frac{(1-\bM_+^2)\ddot\xi^*}
{L-\bxi^*-\ddot\xi^*}
\int_{\ml C}
\bigl(
\du_{1+}(L,y')-\du_{1+}(\bxi^*,y')
\bigr)\,dy'\\
&+\si\frac{\bu_+}{\bc_+^2}
\frac{L-\bxi^*}{L-\bxi^*-\ddot\xi^*}
\int_{\ml C}
\bigl(
\Phi(L,y')-
\Phi(\bxi^*+\ddot\xi^*,y')
\bigr)\,dy'
+R_{12},
\end{split}
\end{equation}
where
\begin{equation}\label{def_pos_R12}
R_{12}:={}R_{12}^{(0)}
-\frac{\si\ddot\xi^*}
{L-\bxi^*-\ddot\xi^*}
\left(
\frac{1-\bM_+^2}{\brho\bu}\bar K
-\frac{\bu_+}{\bc_+^2}
\right)
\int_{\ml C}
\bigl(
\Phi(\bxi^*+\ddot\xi^*,y')
-\Phi(\bxi^*,y')
\bigr)\,dy'.
\end{equation}
For $|\ddot\xi^*|\leq r_*\si$, the second term in
\eqref{def_pos_R12} is bounded by
$C\si|\ddot\xi^*|^2$. Hence
\begin{equation*}
|R_{12}|\leq C\si^2.
\end{equation*}
To estimate $\ml I^3$, expand \eqref{it_F0_split} using
\eqref{pos_front} and collect the terms containing only
$\ddot\xi^*$. By the coordinate change,
\begin{equation*}
\ml D_i^{\widehat\xi}\bigl(\Phi(X_{\widehat\xi}(\y))\bigr)
=(\p_{x_i}\Phi)(X_{\widehat\xi}(\y)),\qquad i=1,2,3.
\end{equation*}
Consequently,
\begin{equation}\label{pos_source_exp}
\begin{split}
\ml F_0
={}&-\frac{(1-\bM_+^2)\ddot\xi^*}
{L-\bxi^*-\ddot\xi^*}
\p_{y_1}\du_{1+}
+\si\frac{\bu_+}{\bc_+^2}
(\p_{x_1}\Phi)\left(
\bxi^*+\ddot\xi^*
+\frac{L-\bxi^*-\ddot\xi^*}{L-\bxi^*}
(y_1-\bxi^*),y'
\right)\\
&-\frac{(1-\bM_+^2)\ddot\xi^*}
{L-\bxi^*-\ddot\xi^*}
\p_{y_1}\lrb{\widehat V_1-\du_{1+}}
+R_F.
\end{split}
\end{equation}
Here
\begin{equation}\label{def_pos_RF}
\begin{aligned}
R_F:={}&
\sum_{i,j=1}^3(\bar a_{ij}-\widehat a_{ij})
\p_{y_i}\widehat V_j\\
&+
\sum_{i,j=1}^3\widehat a_{ij}
(\p_{y_i}-\ml D_i^{\widehat\xi})\widehat V_j
+\frac{(1-\bM_+^2)\ddot\xi^*}
{L-\bxi^*-\ddot\xi^*}\p_{y_1}\widehat V_1\\
&+
\frac{\si}{\bc_+^2}\left\{
\sum_{i=1}^3\widehat u_i
\ml D_i^{\widehat\xi}
\bigl(\Phi(X_{\widehat\xi}(\y))\bigr)
-\bu_+(\p_{x_1}\Phi)\left(
\bxi^*+\ddot\xi^*
+\frac{L-\bxi^*-\ddot\xi^*}{L-\bxi^*}
(y_1-\bxi^*),y'
\right)
\right\}.
\end{aligned}
\end{equation}
From these expressions we obtain
\begin{equation*}
\begin{aligned}
\|R_F\|_{0,\al;\dN_+}
\leq& C\Biggl\{\sum_{i,j=1}^3
\|\bar a_{ij}-\widehat a_{ij}\|_{0,\al;\dN_+}
\|\widehat V\|_{1,\al;\dN_+}\\
&+\Bigl(
\|\widehat\xi-\bxi^*-\ddot\xi^*\|_{1,\al;\ml C}
+\sum_{i=2}^3\|\widehat q_i\|_{0,\al;\ml C}\\
&+|\ddot\xi^*|\sum_{i,j=1}^3
\|\bar a_{ij}-\widehat a_{ij}\|_{0,\al;\dN_+}
\Bigr)\|\widehat V\|_{1,\al;\dN_+}\\
&+\si\Bigl(
\|\widehat V\|_{0,\al;\dN_+}
+\|\widehat\xi-\bxi^*-\ddot\xi^*\|_{0,\al;\ml C}
\Bigr)\|\Phi\|_{2,\al;\ml N}\Biggr\}\\
\leq &C\si^2.
\end{aligned}
\end{equation*}
For the second line of \eqref{def_pos_RF}, we used
\begin{equation*}
\frac{L-\bxi^*}{L-\widehat\xi(y')}
-\frac{L-\bxi^*}{L-\bxi^*-\ddot\xi^*}
=\frac{(L-\bxi^*)
(\widehat\xi(y')-\bxi^*-\ddot\xi^*)}
{(L-\widehat\xi(y'))(L-\bxi^*-\ddot\xi^*)}.
\end{equation*}
For the potential terms in \eqref{def_pos_RF}, the coordinate
difference is
\begin{equation*}
\begin{aligned}
X_1(\y;\widehat\xi)
-\left(
\bxi^*+\ddot\xi^*
+\frac{L-\bxi^*-\ddot\xi^*}{L-\bxi^*}(y_1-\bxi^*)
\right)
=
\frac{L-y_1}{L-\bxi^*}
\bigl(\widehat\xi(y')-\bxi^*-\ddot\xi^*\bigr).
\end{aligned}
\end{equation*}

It follows that
\begin{equation*}
\left|\int_{\dN_+}R_F\,d\y\right|
\leq C\|R_F\|_{0,\al;\dN_+}
\leq C\si^2.
\end{equation*}
By \eqref{it_space},
\begin{equation}\label{pos_key_error}
\left|
\frac{(1-\bM_+^2)\ddot\xi^*}
{L-\bxi^*-\ddot\xi^*}
\int_{\dN_+}
\p_{y_1}\lrb{\widehat V_1-\du_{1+}}\,d\y
\right|
\leq C\si^{3/2}|\ddot\xi^*|.
\end{equation}
Integrating the first two terms in
\eqref{pos_source_exp} with respect to $y_1$, we find
\begin{equation}\label{pos_I3}
\begin{split}
\ml I^3
={}&\frac{(1-\bM_+^2)\ddot\xi^*}
{L-\bxi^*-\ddot\xi^*}
\int_{\ml C}
\bigl(
\du_{1+}(L,y')-\du_{1+}(\bxi^*,y')
\bigr)\,dy'\\
&-\si\frac{\bu_+}{\bc_+^2}
\frac{L-\bxi^*}{L-\bxi^*-\ddot\xi^*}
\int_{\ml C}
\bigl(
\Phi(L,y')-
\Phi(\bxi^*+\ddot\xi^*,y')
\bigr)\,dy'\\
&+\frac{(1-\bM_+^2)\ddot\xi^*}
{L-\bxi^*-\ddot\xi^*}
\int_{\dN_+}\p_{y_1}(\widehat V_1-\du_{1+})\,d\y
-\int_{\dN_+}R_F\,d\y.
\end{split}
\end{equation}

Adding \eqref{pos_I12} and \eqref{pos_I3}, the second and third
lines in \eqref{pos_I12} cancel. Thus
\begin{equation}\label{pos_exp}
\ml I
=\frac{1-\bM_+^2}{\brho\bu}\si
\frac{L-\bxi^*}{L-\bxi^*-\ddot\xi^*}
\bigl(F(\bxi^*+\ddot\xi^*)-F(\bxi^*)\bigr)+R_I,
\end{equation}
where we define
\begin{equation}\label{def_pos_RI}
R_I:=R_{12}
+\frac{(1-\bM_+^2)\ddot\xi^*}
{L-\bxi^*-\ddot\xi^*}
\int_{\dN_+}\p_{y_1}(\widehat V_1-\du_{1+})\,d\y
-\int_{\dN_+}R_F\,d\y.
\end{equation}
The preceding estimates give
\begin{equation}\label{est_pos_rem}
|R_I|
\leq C\bigl(\si^2+\si^{3/2}|\ddot\xi^*|
+\si|\ddot\xi^*|^2\bigr).
\end{equation}
The dependence on the displacement can be estimated by differentiating
the preceding formulas. In particular, differentiating the entropy
problem \eqref{it_S_eq} with \eqref{it_S_data} gives
\begin{equation*}
\begin{aligned}
&\p_{y_1}\p_{\ddot\xi^*}V_4
+\sum_{i=2}^3\widehat b_i\p_{y_i}\p_{\ddot\xi^*}V_4
=-\sum_{i=2}^3(\p_{\ddot\xi^*}\widehat b_i)\p_{y_i}V_4
\quad\text{in }\dN_+,\\
&\p_{\ddot\xi^*}V_4(\bxi^*,y')
=\p_{\ddot\xi^*}\ddot g_4(y').
\end{aligned}
\end{equation*}
The coefficient and boundary formulas, together with the transport
estimate, yield
\begin{equation*}
\sum_{i=2}^3
\|\p_{\ddot\xi^*}\widehat b_i\|_{1,\al;\dN_+}
+\|\p_{\ddot\xi^*}\ddot g_4\|_{1,\al;\ml C}
+\|\p_{\ddot\xi^*}V_4\|_{1,\al;\dN_+}
\leq C\si.
\end{equation*}
By \eqref{pos_front}, $\p_{\ddot\xi^*}\widehat\xi=1$. The formulas for $r_1$, $r_4$,
$\widetilde r_4$ and $r_e$, and \eqref{def_pos_RF}, then imply
\begin{equation*}
\|\p_{\ddot\xi^*}r_1\|_{0,\al;\ml C}
+\|\p_{\ddot\xi^*}r_e\|_{0,\al;\ml C}
+\|\p_{\ddot\xi^*}R_F\|_{0,\al;\dN_+}
\leq C\si^2.
\end{equation*}
Using these bounds in \eqref{def_pos_R12} and \eqref{def_pos_RI},
and applying \eqref{it_space} to the term in \eqref{pos_key_error},
we obtain
\begin{equation*}
\left|\p_{\ddot\xi^*}R_I\right|
\leq C\bigl(\si^{3/2}+\si|\ddot\xi^*|\bigr).
\end{equation*}

The force condition \eqref{lin_pos_mono} gives
\begin{equation*}
F'(\bxi^*)
=\bar K\int_{\ml C}\p_{x_1}\Phi(\bxi^*,y')\,dy'\neq0.
\end{equation*}
Thus \eqref{pos_exp} yields
\begin{equation}\label{pos_deriv}
\begin{aligned}
\p_{\ddot\xi^*}\ml I
&=\frac{1-\bM_+^2}{\brho\bu}\si F'(\bxi^*)
+O\bigl(\si|\ddot\xi^*|+\si^{3/2}\bigr),\\
\left|\p_{\ddot\xi^*}\ml I\right|&\geq c_*\si,
\end{aligned}
\end{equation}
for sufficiently small $\si$, where
\begin{equation*}
c_*:=\frac{1-\bM_+^2}{2\brho\bu}|F'(\bxi^*)|>0.
\end{equation*}
At $\ddot\xi^*=0$, \eqref{it_low_bound} and the entropy transport
estimate give
\begin{equation*}
\|V_4\|_{1,\al;\dN_+}\leq C\si,
\qquad
|\ml I(0)|\leq C_0\si^2,
\end{equation*}
with $C_0$ independent of $M$. Choose
$r_*>\max\{1,2C_0/c_*\}$, independently of $M$, and then take
$\si_1$ sufficiently small. The implicit function theorem,
applied in its quantitative form to $\si^{-1}\ml I=0$, gives
\begin{equation*}
|\ddot\xi^*|\leq\frac{2C_0}{c_*}\si<r_*\si.
\end{equation*}
This proves \eqref{est_it_pos}.
\end{proof}

Evaluate the data at the displacement $\ddot\xi^*$ given by
Lemma~\ref{lem_it_pos}. They then satisfy the integral
solvability condition \eqref{dc_compat}.
The required wall compatibility conditions follow from
\eqref{it_shock_wall_data} and \eqref{it_exit_wall} for the
boundary data,
\eqref{it_G_wall} for the curl data, and
\eqref{it_F0_wall} for the source term.
Therefore, Lemma~\ref{lem_it_dc} gives a unique solution of
\eqref{it_dc} with the boundary conditions \eqref{it_vel_bc}
and \eqref{it_Pi_bc}, together with the estimate
\eqref{est_it_dc}.

The coefficient differences in \eqref{it_F0_split} vanish at the background
state, and the differences between $\ml D_i^{\widehat\xi}$ and
$\p_{y_i}$ are controlled by $\widehat\xi-\bxi^*$. Hence
\begin{equation*}
\|\ml F_0\|_{1,\al;\dN_+}
\leq C\Bigl\{\si+
\bigl(\|\widehat V\|_{1,\al;\dN_+}
+\|\widehat\xi-\bxi^*\|_{2,\al;\ml C}+\si\bigr)
\|\widehat V\|_{2,\al;\dN_+}\Bigr\}
\leq C\si.
\end{equation*}
Substituting \eqref{it_om23} into \eqref{it_G123} gives
\begin{equation*}
\begin{aligned}
\sum_{i=1}^3\|\ml G_i\|_{1,\al;\dN_+}
\leq& C\Bigl\{
(1+\|\widehat V\|_{1,\al;\dN_+})
\|\om_1\|_{1,\al;\dN_+}
+\|V_4\|_{2,\al;\dN_+}\\
&+\|\widehat\xi-\bxi^*\|_{2,\al;\ml C}
\|\widehat V\|_{2,\al;\dN_+}\Bigr\}
\leq C\si.
\end{aligned}
\end{equation*}
The boundary functions in \eqref{it_shock_data} and
\eqref{it_exit_data} satisfy
\begin{equation*}
\begin{aligned}
\|\ddot g_1\|_{2,\al;\ml C}
+\|\ddot g_e\|_{2,\al;\ml C}
\quad\leq& C\Bigl\{
\|V_-\|_{2,\al;\ml N}+\|V_4\|_{2,\al;\dN_+}+\si\\
&+\bigl(\|\widehat V\|_{2,\al;\dN_+}
+\|V_-\|_{2,\al;\ml N}
+\|V_4\|_{2,\al;\dN_+}+\si\bigr)^2\Bigr\}\\
\leq& C\si.
\end{aligned}
\end{equation*}
The constants here may depend on $M$, which will be fixed below.

For comparison, \eqref{lin_down_eq} and \eqref{lin_down_const}
give the data of the linear solution:
\begin{equation*}
\begin{aligned}
\dot F_0(\y)&=\si\frac{\bu_+}{\bc_+^2}\p_{y_1}\Phi(\y),
&\dot\om_{1+}&=0,\\
\dot\om_{2+}&=-\frac{\brho_+^{\ga-1}}{(\ga-1)\bu_+}
\p_{y_3}\dS_+,
&\dot\om_{3+}&=\frac{\brho_+^{\ga-1}}{(\ga-1)\bu_+}
\p_{y_2}\dS_+.
\end{aligned}
\end{equation*}
For the forcing term, the chain rule gives
\begin{equation*}
\ml D_i^{\widehat\xi}\bigl(\Phi(X_1(\y;\widehat\xi),y')\bigr)
=(\p_{x_i}\Phi)(X_1(\y;\widehat\xi),y'),
\qquad i=1,2,3.
\end{equation*}
Consequently, \eqref{it_F0_split} yields
\begin{equation*}
\begin{aligned}
\|\ml F_0-\dot F_0\|_{0,\al;\dN_+}
\leq& C\Bigl\{
\bigl(\|\widehat V\|_{1,\al;\dN_+}
+\|\widehat\xi-\bxi^*\|_{1,\al;\ml C}+\si\bigr)
\|\widehat V\|_{1,\al;\dN_+}\\
&+\si\bigl(\|\widehat V\|_{0,\al;\dN_+}
+\|\widehat\xi-\bxi^*\|_{0,\al;\ml C}\bigr)\Bigr\}
\leq C\si^2.
\end{aligned}
\end{equation*}
Since $\p_{y_1}\dS_+=0$,
\begin{equation*}
\ml D_i^{\widehat\xi}\dS_+=\p_{y_i}\dS_+,
\qquad i=2,3.
\end{equation*}
The last two equations in \eqref{it_G123} therefore give
\begin{equation*}
\begin{aligned}
\ml G_2-\dot\om_{2+}
={}&\frac{\widehat u_2}{\widehat u_1}\om_1
-\left(\frac{\rho^{\ga-1}}{(\ga-1)\widehat u_1}
-\frac{\brho_+^{\ga-1}}{(\ga-1)\bu_+}\right)
\ml D_3^{\widehat\xi}V_4\\
&-\frac{\brho_+^{\ga-1}}{(\ga-1)\bu_+}
\ml D_3^{\widehat\xi}(V_4-\dS_+)
+\frac{L-y_1}{L-\widehat\xi(y')}
\widehat q_3\p_{y_1}\widehat V_1\\
&+\left(\frac{L-\bxi^*}{L-\widehat\xi(y')}-1\right)
\p_{y_1}\widehat V_3,\\
\ml G_3-\dot\om_{3+}
={}&\frac{\widehat u_3}{\widehat u_1}\om_1
+\left(\frac{\rho^{\ga-1}}{(\ga-1)\widehat u_1}
-\frac{\brho_+^{\ga-1}}{(\ga-1)\bu_+}\right)
\ml D_2^{\widehat\xi}V_4\\
&+\frac{\brho_+^{\ga-1}}{(\ga-1)\bu_+}
\ml D_2^{\widehat\xi}(V_4-\dS_+)
-\frac{L-y_1}{L-\widehat\xi(y')}
\widehat q_2\p_{y_1}\widehat V_1\\
&-\left(\frac{L-\bxi^*}{L-\widehat\xi(y')}-1\right)
\p_{y_1}\widehat V_2.
\end{aligned}
\end{equation*}
Together with the first equation in \eqref{it_G123}, these
identities imply
\begin{equation*}
\begin{aligned}
\sum_{i=1}^3\|\ml G_i-\dot\om_{i+}\|_{0,\al;\dN_+}
\leq& C\Bigl\{
(1+\|\widehat V\|_{0,\al;\dN_+})\|\om_1\|_{0,\al;\dN_+}
+\|V_4-\dS_+\|_{1,\al;\dN_+}\\
&+
\bigl(\|\widehat V\|_{0,\al;\dN_+}
+\|V_4\|_{0,\al;\dN_+}+\si\bigr)
\|V_4\|_{1,\al;\dN_+}\\
&+
\|\widehat\xi-\bxi^*\|_{1,\al;\ml C}
\|\widehat V\|_{1,\al;\dN_+}\Bigr\}
\leq C\si^2.
\end{aligned}
\end{equation*}
The last inequality follows from \eqref{est_it_S}, the estimate
for $\om_1$, and \eqref{est_it_pos}.

The linear terms of $g_1$ and $g_e$ agree with the boundary data
of $\du_{1+}$. Subtracting these terms in
\eqref{it_shock_data} and \eqref{it_exit_data} gives
\begin{equation*}
\begin{aligned}
&\|\ddot g_1-\dot g_1\|_{1,\al;\ml C}
+\|\ddot g_e-\dot g_e\|_{1,\al;\ml C}\\
\leq& C\Bigl\{
\|V_- -\ml{\dU}_-\|_{1,\al;\ml N}
+\si\|\widehat\xi-\bxi^*\|_{1,\al;\ml C}
+\|V_4-\dS_+\|_{1,\al;\dN_+}\\
&+
\bigl(\|\widehat V\|_{1,\al;\dN_+}
+\|V_-\|_{1,\al;\ml N}
+\|V_4\|_{1,\al;\dN_+}+\si\bigr)^2\Bigr\}\\
\leq& C\si^2.
\end{aligned}
\end{equation*}
Thus the complete data estimates are
\begin{equation}\label{it_data_est}
\begin{aligned}
&\|\ml F_0\|_{1,\al;\dN_+}
+\sum_{i=1}^3\|\ml G_i\|_{1,\al;\dN_+}
+\|\ddot g_1\|_{2,\al;\ml C}
+\|\ddot g_e\|_{2,\al;\ml C}
\leq C\si,\\
&\|\ml F_0-\dot F_0\|_{0,\al;\dN_+}
+\sum_{i=1}^3\|\ml G_i-\dot\om_{i+}\|_{0,\al;\dN_+}+
\|\ddot g_1-\dot g_1\|_{1,\al;\ml C}
+\|\ddot g_e-\dot g_e\|_{1,\al;\ml C}
\leq C\si^2.
\end{aligned}
\end{equation}

Subtracting the linear velocity equations from
\eqref{it_dc} gives
\begin{equation}\label{vel_lin_diff}
\begin{cases}
(1-\bM_+^2)\p_{y_1}(V_1-\du_{1+})
+\displaystyle\sum_{i=2}^3\p_{y_i}(V_i-\du_{i+})
=\ml F_0-\dot F_0,\\
\p_{y_2}(V_3-\du_{3+})-\p_{y_3}(V_2-\du_{2+})
+\p_{y_1}\Pi=\ml G_1-\dot\om_{1+},\\
\p_{y_3}(V_1-\du_{1+})-\p_{y_1}(V_3-\du_{3+})
+\p_{y_2}\Pi=\ml G_2-\dot\om_{2+},\\
\p_{y_1}(V_2-\du_{2+})-\p_{y_2}(V_1-\du_{1+})
+\p_{y_3}\Pi=\ml G_3-\dot\om_{3+},
\end{cases}
\quad\text{in }\dN_+,
\end{equation}
endowed with
\begin{equation*}
\begin{cases}
(V_1-\du_{1+})(\bxi^*,y')
=\ddot g_1(y')-\dot g_1(y'),\\
(V_1-\du_{1+})(L,y')
=\ddot g_e(y')-\dot g_e(y'),\\
V_2-\du_{2+}=0\quad\text{on }\Ga_2^\pm,
\qquad V_3-\du_{3+}=0\quad\text{on }\Ga_3^\pm,
\end{cases}
\end{equation*}
and the conditions \eqref{it_Pi_bc} for $\Pi$. The
compatibility condition follows by subtracting that of the linear
solution from \eqref{pos_eq}. Applying
Lemma~\ref{lem_it_dc} first to \eqref{it_dc} with
$k=1$, and then to \eqref{vel_lin_diff} with $k=0$,
we obtain
\begin{equation}\label{est_it_vel}
\begin{aligned}
&\sum_{i=1}^3\|V_i\|_{2,\al;\dN_+}
+\|\Pi\|_{2,\al;\dN_+}\\
\leq& C\Bigl(
\|\ml F_0\|_{1,\al;\dN_+}
+\sum_{i=1}^3\|\ml G_i\|_{1,\al;\dN_+}
+\sum_{i=1,e}^{}\|\ddot g_i\|_{2,\al;\ml C} \Bigr)\\
\leq& C\si,\\
&\sum_{i=1}^3\|V_i-\du_{i+}\|_{1,\al;\dN_+}
+\|\Pi\|_{1,\al;\dN_+}\\
\leq&
C\Bigl(
\|\ml F_0-\dot F_0\|_{0,\al;\dN_+}
+\sum_{i=1}^3\|\ml G_i-\dot\om_{i+}\|_{0,\al;\dN_+}+
\sum_{i=1,e}^{}\|\ddot g_i-\dot g_i\|_{1,\al;\ml C}
\Bigr)\\
\leq& C\si^2.
\end{aligned}
\end{equation}

\medskip
\noindent
\textit{Step~4. Recovery of the shock.}
The velocity just obtained determines the updated shock derivatives
through \eqref{shock_slopes}:
\begin{equation}\label{it_slopes}
q_i(y')=a_0\bigl(V_i(\bxi^*,y')
-V_{i-}(\widehat\xi(y'),y')\bigr)+\ml R_i(y'),
\qquad i=2,3.
\end{equation}
Here $\ml R_i(y')$ denotes the same evaluation at the iteration
input as in \eqref{shock_slopes}. By the first curl equation in
\eqref{it_dc}, \eqref{it_G1_shock}, and
$\p_{y_1}\Pi=0$ on $y_1=\bxi^*$,
\begin{equation*}
\begin{aligned}
\p_{y_2}q_3-\p_{y_3}q_2
={}&a_0\ml G_1(\bxi^*,y')
-a_0\p_{y_2}\bigl(V_{3-}(\widehat\xi(y'),y')\bigr)\\
&+a_0\p_{y_3}\bigl(V_{2-}(\widehat\xi(y'),y')\bigr)
+\p_{y_2}\ml R_3(y')-\p_{y_3}\ml R_2(y')=0.
\end{aligned}
\end{equation*}
Thus
\begin{equation}\label{it_slope_curl}
\p_{y_2}q_3=\p_{y_3}q_2\qquad\text{in }\ml C.
\end{equation}
The output shock is
\begin{equation}\label{it_output_front}
\xi(y_2,y_3)=\bxi^*+\ddot\xi^*
+\int_1^{y_2}q_2(s,1)\,ds
+\int_1^{y_3}q_3(y_2,s)\,ds.
\end{equation}
Indeed, differentiating this formula and using
\eqref{it_slope_curl} gives
\begin{equation*}
\p_{y_2}\xi=q_2,\qquad \p_{y_3}\xi=q_3,
\qquad \xi(1,1)=\bxi^*+\ddot\xi^*.
\end{equation*}
Since $q_2,q_3\in C^{2,\al}$, these derivative identities give
$\xi\in C^{3,\al}$. The wall conditions for the velocity, the upstream
solution, and the remainders give
\begin{equation*}
q_2=\p_{y_2}^2q_2=0\quad\text{on }y_2=\pm1,
\qquad q_3=\p_{y_3}^2q_3=0\quad\text{on }y_3=\pm1.
\end{equation*}
Hence the updated derivatives satisfy \eqref{it_shape_wall}.

The preceding steps define the iteration map
\begin{equation}\label{iter_map}
\ml T:(\widehat V,\widehat{\bm q})\longmapsto(V,\bm q).
\end{equation}

\begin{proposition}\label{prop_iter_map}
There exists $M>1$, and for this $M$ there exists $\si_2>0$, such
that, for $0<\si\leq\si_2$,
\begin{equation}\label{iter_self_map}
\ml T(\mk K)\subset\mk K.
\end{equation}
Moreover,
\begin{equation}\label{est_iter_output}
\begin{aligned}
&\|V\|_{2,\al;\dN_+}
+\sum_{i=2}^3\|q_i\|_{2,\al;\ml C}+|\ddot\xi^*|\leq C\si,\\
&\|V-\ml{\dU}_+\|_{1,\al;\dN_+}
+\sum_{i=2}^3\|q_i-\p_{y_i}\dot\xi\|_{1,\al;\ml C}
\leq C\si^2.
\end{aligned}
\end{equation}
\end{proposition}

\begin{proof}
The position bound follows from Lemma~\ref{lem_it_pos}.
The entropy formula and \eqref{est_it_S}, together with
\eqref{est_it_vel} and $V_5=\dB_+=0$, give the bounds for $V$.
For the shock derivatives, \eqref{it_slopes} first gives
\begin{equation*}
\sum_{i=2}^3\|q_i\|_{2,\al;\ml C}
\leq C\left(
\sum_{i=2}^3\|V_i\|_{2,\al;\dN_+}
+\|V_-\|_{2,\al;\ml N}
+\sum_{i=2}^3\|\ml R_i\|_{2,\al;\ml C}\right)
\leq C\si.
\end{equation*}
Subtracting the linear shock relations in
\eqref{def_lin_shock_shape}, we obtain, for $i=2,3$,
\begin{equation*}
\begin{aligned}
q_i(y')-\p_{y_i}\dot\xi(y')
={}&a_0(V_i-\du_{i+})(\bxi^*,y')
-a_0(V_{i-}-\du_{i-})(\widehat\xi(y'),y')\\
&-a_0\int_{\bxi^*}^{\widehat\xi(y')}
\p_{x_1}\du_{i-}(s,y')\,ds+\ml R_i(y').
\end{aligned}
\end{equation*}
Consequently,
\begin{equation}\label{est_it_slopes}
\begin{aligned}
&\sum_{i=2}^3\|q_i-\p_{y_i}\dot\xi\|_{1,\al;\ml C}\\
\leq& C\Bigl(
\sum_{i=2}^3\|V_i-\du_{i+}\|_{1,\al;\dN_+}
+\|V_- -\ml{\dU}_-\|_{1,\al;\ml N}\\
&+
\|\widehat\xi-\bxi^*\|_{1,\al;\ml C}
\|\ml{\dU}_-\|_{2,\al;\ml N}
+\sum_{i=2}^3\|\ml R_i\|_{1,\al;\ml C}\Bigr)\\
\leq& C\si^2.
\end{aligned}
\end{equation}
The preceding estimates, together with \eqref{it_space}
and \eqref{it_low_bound}, yield
\begin{equation}\label{raw_iter_output}
\begin{aligned}
&\|V\|_{2,\al;\dN_+}
+\sum_{i=2}^3\|q_i\|_{2,\al;\ml C}
\leq C_0\si+C_1(1+M)^2\si^2,\\
&\|V-\ml{\dU}_+\|_{1,\al;\dN_+}
+\sum_{i=2}^3\|q_i-\p_{y_i}\dot\xi\|_{1,\al;\ml C}
\leq C_1(1+M)^2\si^2,
\end{aligned}
\end{equation}
where $C_0,C_1$ are independent of $M$ and $\si$. Take $C_0$
also larger than the constant in \eqref{est_lin_approx}.
Choose $M>2C_0$ and then decrease $\si_2$ so that
\begin{equation*}
C_1(1+M)^2\si\leq\frac M2,
\qquad C_1(1+M)^2\si^2\leq\si^{3/2}.
\end{equation*}
The transport equations, Lemma~\ref{lem_it_dc}, and
\eqref{it_slope_curl} preserve the wall conditions and the
derivative relation defining $\mk X$. Together with
\eqref{raw_iter_output}, this proves
\eqref{iter_self_map}.
\end{proof}

The update therefore preserves the iteration class and its wall
conditions. It remains to prove convergence and to check
that the auxiliary function introduced in the velocity problem
vanishes at the limit.

\subsection{Convergence and completion of the proof}
\label{sub:contraction}

The contraction estimate uses one derivative less than the bounds
for the iteration map. On $\mk K$, define the distance
\begin{equation}\label{it_distance}
d\bigl((V^{(1)},\bm q^{(1)}),(V^{(2)},\bm q^{(2)})\bigr)
:=\|V^{(1)}-V^{(2)}\|_{1,\al;\dN_+}
+\sum_{i=2}^3\|q_i^{(1)}-q_i^{(2)}\|_{1,\al;\ml C}.
\end{equation}
The set $\mk K$ is complete with respect to the distance
\eqref{it_distance}. By \eqref{est_lin_approx} and the
choice of $M$,
\begin{equation*}
\bigl(\ml{\dU}_+,
(\p_{y_2}\dot\xi,\p_{y_3}\dot\xi)\bigr)\in\mk K,
\end{equation*}
so $\mk K$ is nonempty.

\begin{lemma}\label{lem_contraction}
For sufficiently small $\si>0$, the iteration map satisfies
\begin{equation}\label{contraction_est}
\begin{split}
d\bigl(\ml T(\widehat V^{(1)},\widehat{\bm q}^{(1)}),
\ml T(\widehat V^{(2)},\widehat{\bm q}^{(2)})\bigr)
\leq C\si\,
d\bigl((\widehat V^{(1)},\widehat{\bm q}^{(1)}),
(\widehat V^{(2)},\widehat{\bm q}^{(2)})\bigr)
\end{split}
\end{equation}
for all $(\widehat V^{(j)},\widehat{\bm q}^{(j)})\in\mk K$, $j=1,2$.
\end{lemma}

\begin{proof}
Let
\begin{equation*}
\delta:=d\bigl((\widehat V^{(1)},\widehat{\bm q}^{(1)}),
(\widehat V^{(2)},\widehat{\bm q}^{(2)})\bigr),
\qquad
(V^{(j)},\bm q^{(j)})
:=\ml T(\widehat V^{(j)},\widehat{\bm q}^{(j)}).
\end{equation*}
We first compare the position conditions for these two inputs.
Fix a displacement $r$ in the interval given by
Lemma~\ref{lem_it_pos}. For $j=1,2$, define
\begin{equation*}
\widehat\xi^{(j)}(y_2,y_3;r)
:=\bxi^*+r
+\int_1^{y_2}\widehat q_2^{(j)}(s,1)\,ds
+\int_1^{y_3}\widehat q_3^{(j)}(y_2,s)\,ds.
\end{equation*}
Here the superscript $(j)$ identifies the input used in the construction.
In particular, we write
\begin{equation*}
\ddot g_k^{(j)}(y';r)
:=\ddot g_k\bigl(\widehat V^{(j)},\widehat{\bm q}^{(j)};
\ml U_-,\Phi;r\bigr)(y'),
\qquad k=1,4.
\end{equation*}
We use the same superscript for the entropy, coefficients, and remaining
data obtained from this input and displacement. All prescribed data,
including $p_{ex}$, are the same in the two constructions.

Subtracting the reconstruction formulas gives
\begin{equation*}
\|\widehat\xi^{(1)}(\cdot;r)-\widehat\xi^{(2)}(\cdot;r)
\|_{2,\al;\ml C}
\leq C\sum_{i=2}^3
\|\widehat q_i^{(1)}-\widehat q_i^{(2)}\|_{1,\al;\ml C}
\leq C\delta.
\end{equation*}
The upstream values satisfy
\begin{equation*}
\begin{aligned}
V_-\bigl(\widehat\xi^{(1)}(y';r),y'\bigr)
-V_-\bigl(\widehat\xi^{(2)}(y';r),y'\bigr)
&\quad=\int_{\widehat\xi^{(2)}(y';r)}^{\widehat\xi^{(1)}(y';r)}
\p_{x_1}V_-(s,y')\,ds.
\end{aligned}
\end{equation*}
This identity and the corresponding identity for $\Phi$ yield
\begin{equation*}
\begin{split}
&\|V_-(\widehat\xi^{(1)}(\cdot;r),\cdot)
-V_-(\widehat\xi^{(2)}(\cdot;r),\cdot)\|_{1,\al;\ml C}
+\si\|\Phi(\widehat\xi^{(1)}(\cdot;r),\cdot)
-\Phi(\widehat\xi^{(2)}(\cdot;r),\cdot)
\|_{1,\al;\ml C}\\
\leq& C\bigl(\|V_-\|_{2,\al;\ml N}
+\si\|\Phi\|_{2,\al;\ml N}\bigr)
\|\widehat\xi^{(1)}(\cdot;r)-\widehat\xi^{(2)}(\cdot;r)
\|_{1,\al;\ml C}\\
\leq& C\si\delta.
\end{split}
\end{equation*}
The shock remainders are of second order in the perturbations.
Together with \eqref{it_S_coeff}, this gives
\begin{equation*}
\begin{split}
&\sum_{k=1,4}\|\ddot g_k^{(1)}(\cdot;r)
-\ddot g_k^{(2)}(\cdot;r)\|_{1,\al;\ml C}
\leq C\si\delta,\\
&\sum_{i=2}^3\|\widehat b_i^{(1)}(\cdot;r)
-\widehat b_i^{(2)}(\cdot;r)\|_{1,\al;\dN_+}
\leq C\delta.
\end{split}
\end{equation*}

For the fixed $r$, subtracting the two entropy equations yields
\begin{equation}\label{diff_S_eq}
\begin{split}
&\p_{y_1}(V_4^{(1)}-V_4^{(2)})
+\sum_{i=2}^3\widehat b_i^{(1)}
\p_{y_i}(V_4^{(1)}-V_4^{(2)})=-\sum_{i=2}^3
(\widehat b_i^{(1)}-\widehat b_i^{(2)})\p_{y_i}V_4^{(2)}
\end{split}
\end{equation}
in $\dN_+$ endowed with
\begin{equation*}
(V_4^{(1)}-V_4^{(2)})(\bxi^*,y';r)
=\ddot g_4^{(1)}(y';r)-\ddot g_4^{(2)}(y';r).
\end{equation*}
The dependence on $r$ is omitted in \eqref{diff_S_eq}.
The transport estimate gives
\begin{equation*}
\begin{split}
&\|V_4^{(1)}(\cdot;r)-V_4^{(2)}(\cdot;r)\|_{1,\al;\dN_+}\\
\leq& C\Bigl(
\|\ddot g_4^{(1)}(\cdot;r)-\ddot g_4^{(2)}(\cdot;r)
\|_{1,\al;\ml C}\\
&+\sum_{i=2}^3\|\widehat b_i^{(1)}(\cdot;r)
-\widehat b_i^{(2)}(\cdot;r)\|_{1,\al;\dN_+}
\|V_4^{(2)}(\cdot;r)\|_{2,\al;\dN_+}
\Bigr)\\
\leq& C\si\delta.
\end{split}
\end{equation*}
The exit data in \eqref{it_exit_data} and \eqref{it_F0_split}
therefore satisfy
\begin{equation*}
\begin{split}
&\|\ddot g_e^{(1)}(\cdot;r)-\ddot g_e^{(2)}(\cdot;r)
\|_{1,\al;\ml C}
+\|\ml F_0^{(1)}(\cdot;r)-\ml F_0^{(2)}(\cdot;r)
\|_{0,\al;\dN_+}\\
\leq& C\Bigl(
\|V_4^{(1)}(\cdot;r)-V_4^{(2)}(\cdot;r)\|_{1,\al;\dN_+}
+\si\delta\Bigr)\\
\leq& C\si\delta.
\end{split}
\end{equation*}
Let $\ml I^{(j)}(r)$ denote the functional \eqref{pos_fun}
for the $j$th input. By its integral formula,
\begin{equation}\label{pos_input_lip}
\begin{split}
|\ml I^{(1)}(r)-\ml I^{(2)}(r)|
\leq& C\Bigl(
\sum_{i=1,e}^{}\|\ddot g_i^{(1)}-\ddot g_i^{(2)}\|_{0,\al;\ml C}
+\|\ml F_0^{(1)}-\ml F_0^{(2)}\|_{0,\al;\dN_+}\Bigr)\\
\leq& C\si\delta.
\end{split}
\end{equation}
All quantities on the right are evaluated at $r$. Let $\ddot\xi_j^*$ be the zero of $\ml I^{(j)}$. Since
\begin{equation*}
\ml I^{(2)}(\ddot\xi_1^*)-\ml I^{(2)}(\ddot\xi_2^*)
=\ml I^{(2)}(\ddot\xi_1^*)-\ml I^{(1)}(\ddot\xi_1^*),
\end{equation*}
the mean value theorem, \eqref{pos_deriv}, and
\eqref{pos_input_lip} yield
\begin{equation}\label{diff_pos}
c_*\si|\ddot\xi_1^*-\ddot\xi_2^*|
\leq|\ml I^{(2)}(\ddot\xi_1^*)-\ml I^{(1)}(\ddot\xi_1^*)|
\leq C\si\delta,
\qquad
|\ddot\xi_1^*-\ddot\xi_2^*|\leq C\delta.
\end{equation}

We now evaluate all quantities for the $j$th input at its selected
displacement $\ddot\xi_j^*$ and omit the position argument. In particular,
$\widehat\xi^{(j)}(y'):=\widehat\xi^{(j)}(y';\ddot\xi_j^*)$.
The reconstruction formula and \eqref{diff_pos} give
\begin{equation}\label{diff_front}
\|\widehat\xi^{(1)}-\widehat\xi^{(2)}\|_{2,\al;\ml C}
\leq |\ddot\xi_1^*-\ddot\xi_2^*|
+C\sum_{i=2}^3
\|\widehat q_i^{(1)}-\widehat q_i^{(2)}\|_{1,\al;\ml C}
\leq C\delta.
\end{equation}
The same integral identity for the upstream values therefore gives
\begin{equation}\label{diff_up_trace}
\begin{split}
&\|V_-(\widehat\xi^{(1)}(\cdot),\cdot)
-V_-(\widehat\xi^{(2)}(\cdot),\cdot)\|_{1,\al;\ml C}
+\si\|\Phi(\widehat\xi^{(1)}(\cdot),\cdot)
-\Phi(\widehat\xi^{(2)}(\cdot),\cdot)
\|_{1,\al;\ml C}\\
\leq& C\si\|\widehat\xi^{(1)}-\widehat\xi^{(2)}\|_{1,\al;\ml C}
\leq C\si\delta.
\end{split}
\end{equation}
Equation \eqref{diff_S_eq} and its boundary condition remain
valid for the selected displacements. Repeating the preceding estimates
using \eqref{diff_front}, we obtain
\begin{equation}\label{diff_S_data}
\begin{split}
&\|V_4^{(1)}-V_4^{(2)}\|_{1,\al;\dN_+}
+\|\ddot g_1^{(1)}-\ddot g_1^{(2)}\|_{1,\al;\ml C}
+\|\ddot g_e^{(1)}-\ddot g_e^{(2)}\|_{1,\al;\ml C}
\leq C\si\delta,\\
&\sum_{i=2}^3\|\widehat b_i^{(1)}-\widehat b_i^{(2)}
\|_{1,\al;\dN_+}\leq C\delta.
\end{split}
\end{equation}

Next, subtracting \eqref{it_om1_eq} for the two constructions
gives
\begin{equation}\label{diff_om_eq}
\begin{split}
&\p_{y_1}(\om_1^{(1)}-\om_1^{(2)})
+\sum_{i=2}^3\widehat b_i^{(1)}
\p_{y_i}(\om_1^{(1)}-\om_1^{(2)})
+\mu^{(1)}(\om_1^{(1)}-\om_1^{(2)})\\
=&f_\om^{(1)}-f_\om^{(2)}
-\sum_{i=2}^3(\widehat b_i^{(1)}-\widehat b_i^{(2)})
\p_{y_i}\om_1^{(2)}
-(\mu^{(1)}-\mu^{(2)})\om_1^{(2)}
\end{split}
\end{equation}
in $\dN_+$, endowed with
\begin{equation*}
(\om_1^{(1)}-\om_1^{(2)})(\bxi^*,y')
=\ddot g_\om^{(1)}(y')-\ddot g_\om^{(2)}(y').
\end{equation*}
By \eqref{it_om_data}, \eqref{it_Fomega}, and the preceding
estimates, the boundary data and source term satisfy
\begin{equation*}
\begin{split}
&\|\ddot g_\om^{(1)}-\ddot g_\om^{(2)}\|_{0,\al;\ml C}
+\|f_\om^{(1)}-f_\om^{(2)}\|_{0,\al;\dN_+}\\
\leq& C\si\Bigl(
\|\widehat V^{(1)}-\widehat V^{(2)}\|_{1,\al;\dN_+}
+\|\widehat\xi^{(1)}-\widehat\xi^{(2)}\|_{2,\al;\ml C}
+\|V_4^{(1)}-V_4^{(2)}\|_{1,\al;\dN_+}\Bigr)\\
\leq& C\si\delta,
\end{split}
\end{equation*}
and
\begin{equation*}
\|\mu^{(1)}-\mu^{(2)}\|_{0,\al;\dN_+}\leq C\delta.
\end{equation*}
The transport estimate for \eqref{diff_om_eq}, together with
the bound for $\om_1^{(2)}$ in Step~2, gives
\begin{equation}\label{diff_vort}
\begin{split}
&\|\om_1^{(1)}-\om_1^{(2)}\|_{0,\al;\dN_+}\\
\leq &C\Bigl(
\|\ddot g_\om^{(1)}-\ddot g_\om^{(2)}\|_{0,\al;\ml C}
+\|f_\om^{(1)}-f_\om^{(2)}\|_{0,\al;\dN_+}\\
&
+\Bigl(\sum_{i=2}^3
\|\widehat b_i^{(1)}-\widehat b_i^{(2)}\|_{0,\al;\dN_+}
+\|\mu^{(1)}-\mu^{(2)}\|_{0,\al;\dN_+}\Bigr)
\|\om_1^{(2)}\|_{1,\al;\dN_+}\Bigr)
\leq C\si\delta.
\end{split}
\end{equation}

Subtracting the two linear velocity problems gives
\begin{equation}\label{diff_vel_eqs}
\begin{split}
&(1-\bM_+^2)\p_{y_1}(V_1^{(1)}-V_1^{(2)})
+\sum_{i=2}^3\p_{y_i}(V_i^{(1)}-V_i^{(2)})
=\ml F_0^{(1)}-\ml F_0^{(2)},\\
&\p_{y_2}(V_3^{(1)}-V_3^{(2)})
-\p_{y_3}(V_2^{(1)}-V_2^{(2)})
+\p_{y_1}(\Pi^{(1)}-\Pi^{(2)})=\ml G_1^{(1)}-\ml G_1^{(2)},\\
&\p_{y_3}(V_1^{(1)}-V_1^{(2)})
-\p_{y_1}(V_3^{(1)}-V_3^{(2)})
+\p_{y_2}(\Pi^{(1)}-\Pi^{(2)})=\ml G_2^{(1)}-\ml G_2^{(2)},\\
&\p_{y_1}(V_2^{(1)}-V_2^{(2)})
-\p_{y_2}(V_1^{(1)}-V_1^{(2)})
+\p_{y_3}(\Pi^{(1)}-\Pi^{(2)})=\ml G_3^{(1)}-\ml G_3^{(2)}
\end{split}
\end{equation}
in $\dN_+$, endowed with
\begin{equation}\label{diff_vel_bc}
\begin{cases}
V_1^{(1)}-V_1^{(2)}=\ddot g_1^{(1)}-\ddot g_1^{(2)},
&\text{on }\Ga_s^*,\\
V_1^{(1)}-V_1^{(2)}=\ddot g_e^{(1)}-\ddot g_e^{(2)},
&\text{on }\Ga_e,\\
V_i^{(1)}-V_i^{(2)}=0,
&\text{on }\Ga_i^\pm,\quad i=2,3,\\
\p_{y_1}(\Pi^{(1)}-\Pi^{(2)})=0,
&\text{on }\Ga_s^*,\\
\Pi^{(1)}-\Pi^{(2)}=0,
&\text{on }\Ga_e\cup\Ga_2^\pm\cup\Ga_3^\pm.
\end{cases}
\end{equation}
The integral condition follows by subtracting the two position
equations. Substituting \eqref{it_om23} into \eqref{it_G123},
and using \eqref{it_F0_split}, \eqref{diff_front}, and the
preceding estimates, gives
\begin{equation}\label{diff_dc_data}
\begin{split}
&\|\ml F_0^{(1)}-\ml F_0^{(2)}\|_{0,\al;\dN_+}
+\sum_{i=1}^3\|\ml G_i^{(1)}-\ml G_i^{(2)}\|_{0,\al;\dN_+}\\
\leq& C\Bigl(
\|\om_1^{(1)}-\om_1^{(2)}\|_{0,\al;\dN_+}
+\|V_4^{(1)}-V_4^{(2)}\|_{1,\al;\dN_+}\\
&
+\si\|\widehat V^{(1)}-\widehat V^{(2)}\|_{1,\al;\dN_+}
+\si\|\widehat\xi^{(1)}-\widehat\xi^{(2)}\|_{2,\al;\ml C}
\Bigr)\\
\leq& C\si\delta.
\end{split}
\end{equation}
Applying Lemma~\ref{lem_it_dc} with $k=0$ to
\eqref{diff_vel_eqs} with \eqref{diff_vel_bc} yields
\begin{equation}\label{diff_vel}
\begin{split}
&\sum_{i=1}^3\|V_i^{(1)}-V_i^{(2)}\|_{1,\al;\dN_+}
+\|\Pi^{(1)}-\Pi^{(2)}\|_{1,\al;\dN_+}\\
\leq& C\Bigl(
\|\ml F_0^{(1)}-\ml F_0^{(2)}\|_{0,\al;\dN_+}
+\sum_{i=1}^3\|\ml G_i^{(1)}-\ml G_i^{(2)}\|_{0,\al;\dN_+}\\
&
+\|\ddot g_1^{(1)}-\ddot g_1^{(2)}\|_{1,\al;\ml C}
+\|\ddot g_e^{(1)}-\ddot g_e^{(2)}\|_{1,\al;\ml C}
\Bigr)\\
\leq& C\si\delta.
\end{split}
\end{equation}
Together with \eqref{diff_S_data} and $V_5^{(j)}=0$, this controls
the full difference $V^{(1)}-V^{(2)}$ in $C^{1,\al}$.

Finally, let $\ml R_i^{(j)}(y')$, $i=2,3$, denote the shock remainder
in \eqref{it_slopes} evaluated at the $j$th input and its selected
displacement. Subtracting the two slope relations gives, for $i=2,3$,
\begin{equation*}
\begin{split}
q_i^{(1)}(y')-q_i^{(2)}(y')
={}&a_0\bigl(V_i^{(1)}(\bxi^*,y')-V_i^{(2)}(\bxi^*,y')\bigr)\\
&-a_0\bigl(V_{i-}(\widehat\xi^{(1)}(y'),y')
-V_{i-}(\widehat\xi^{(2)}(y'),y')\bigr)
+\ml R_i^{(1)}(y')-\ml R_i^{(2)}(y').
\end{split}
\end{equation*}
Since the remainders are of second order in the perturbations,
\eqref{diff_up_trace} gives
\begin{equation*}
\sum_{i=2}^3\|\ml R_i^{(1)}-\ml R_i^{(2)}\|_{1,\al;\ml C}
\leq C\si\delta.
\end{equation*}
Together with the velocity estimate on $y_1=\bxi^*$, this yields
\begin{equation}\label{diff_slopes}
\begin{split}
\sum_{i=2}^3\|q_i^{(1)}-q_i^{(2)}\|_{1,\al;\ml C}
\leq& C\Bigl(
\sum_{i=2}^3\|V_i^{(1)}-V_i^{(2)}\|_{1,\al;\dN_+}
+\si\|\widehat\xi^{(1)}-\widehat\xi^{(2)}\|_{1,\al;\ml C}\\
&+\sum_{i=2}^3\|\ml R_i^{(1)}-\ml R_i^{(2)}\|_{1,\al;\ml C}
\Bigr)\\
\leq& C\si\delta.
\end{split}
\end{equation}
This proves \eqref{contraction_est}.
\end{proof}

Choose $M$ as in Proposition~\ref{prop_iter_map} and then take
$\si$ small enough that $C\si<1/2$. The contraction mapping theorem
gives a unique fixed point
\begin{equation}\label{fix_point}
(V,\bm q)\in\mk K,
\qquad \ml T(V,\bm q)=(V,\bm q).
\end{equation}
Let $\ddot\xi^*$ be its displacement determined by
\eqref{pos_eq}. The shock used in Steps~1 to~3 and the shock
recovered in Step~4 have the same value $\bxi^*+\ddot\xi^*$ at $(1,1)$
and the same derivatives $q_2,q_3$. Hence
\begin{equation}\label{fix_front_eq}
\widehat\xi=\xi\quad\text{in }\ml C.
\end{equation}

It remains to show that $\Pi$ vanishes, so that the fixed point
satisfies the original velocity equations. At the fixed point,
$\widehat V=V$ and \eqref{it_dc} gives
\begin{equation}\label{fix_Pi_gradient}
\begin{split}
\p_{y_1}\Pi
&=\om_1-(\ml D_2^\xi u_3-\ml D_3^\xi u_2),\\
\p_{y_2}\Pi
&=\om_2-(\ml D_3^\xi u_1-\ml D_1^\xi u_3),\\
\p_{y_3}\Pi
&=\om_3-(\ml D_1^\xi u_2-\ml D_2^\xi u_1).
\end{split}
\end{equation}
Apply $\ml D_i^\xi$ to the $i$th equality and add the three results.
Substituting the algebraic formulas \eqref{it_om23} and using
the equation \eqref{it_om1_eq} for $\om_1$, we obtain
\begin{equation}\label{fix_Pi_eq}
\begin{split}
&\frac{L-\bxi^*}{L-\xi(y')}\p_{y_1}^2\Pi
+\p_{y_2}^2\Pi+\p_{y_3}^2\Pi-\frac{L-y_1}{L-\xi(y')}
\sum_{i=2}^3\p_{y_i}\xi(y')\p_{y_1y_i}\Pi=0
\end{split}
\end{equation}
in $\dN_+$. This equation is uniformly elliptic for sufficiently small $\si$.
The maximum principle applied to \eqref{fix_Pi_eq} with the
homogeneous boundary conditions \eqref{it_Pi_bc} yields
\begin{equation}\label{fix_Pi_zero}
\Pi\equiv0\qquad\text{in }\dN_+.
\end{equation}
The velocity equations now have exactly the form of the nonlinear
equations in Subsection~\ref{sub:fixed}. The entropy
equation, the equation for $\om_1$, the algebraic formulas for
$\om_2,\om_3$, and the boundary conditions agree with that formulation
because $\widehat V=V$ and $\widehat\xi=\xi$.
At the fixed point, \eqref{shock_slopes} agrees with
\eqref{fix_slope_exp}, giving
\begin{equation*}
\p_{y_i}\xi(y')
=\ml S_i\bigl(\ml U(\bxi^*,y'),\ml U_-(\xi(y'),y');
\Phi(\xi(y'),y')\bigr),\qquad i=2,3.
\end{equation*}
Together with the boundary functions $g_1,g_4$ and $B=\bB$,
these are all the conditions in \eqref{fix_RH}. Thus the fixed
point solves the nonlinear problem in the fixed domain.

\begin{proof}[Proof of Theorem~\ref{thm_fix_domain}]
For sufficiently small $\si$, Proposition~\ref{prop_iter_map}
and Lemma~\ref{lem_contraction} show that $\ml T$ maps the
complete metric space $\mk K$ into itself and is a contraction.
The contraction mapping theorem therefore gives a unique fixed
point, which solves the nonlinear problem on the fixed domain
by the preceding verification.
The stated regularity and estimates follow from
\eqref{est_iter_output} and the reconstruction formula for the
shock. The wall compatibility conditions follow from the
construction.
\end{proof}

\begin{proof}[Proof of Theorem~\ref{thm_nonlin}]
Returning to the original coordinates by the inverse of
\eqref{shock_flatten}, the solution in
Theorem~\ref{thm_fix_domain}, together with the upstream
solution from Proposition~\ref{prop_up}, solves
Problem~\ref{prob1} by Lemma~\ref{lem_equivalence}.
For sufficiently small $\si$, the coordinate transformation
preserves the stated regularity and estimates.
The supersonic and subsonic conditions, as well as the entropy
condition, follow from \eqref{bg_shock_cond} and the smallness
of the perturbation.
\end{proof}

\textbf{Acknowledgment.}
Weng is supported by the National Natural Science Foundation of China under grants 12571240 and 12221001. Yang is supported by the National Natural Science Foundation of China under grant 12601414. The first author would like to thank Prof. Fang Beixiang at Shanghai Jiao Tong University for discussions on the stability of steady transonic shocks in flat nozzles with general sections.

\textbf{Data Availability Statement.}
Data sharing is not applicable because no data sets were generated or
analyzed in this study.

\textbf{Conflict of interest.}
The authors declare that they have no conflict of interest.

\bibliographystyle{amsplain}
\bibliography{RefEuler}

\end{document}